\documentclass[a4paper,11pt]{amsart}
\usepackage[latin1]{inputenc}

\usepackage{multicol}
\usepackage{amssymb, amsmath, amsthm}
\usepackage{latexsym}
\usepackage{dsfont}
\usepackage{graphicx}

\usepackage{curves} 
\usepackage{epic} 
\usepackage{eepic}
\usepackage{xcolor}
 
\input xy
\xyoption{all}

\usepackage{verbatim}

\usepackage{hyperref}

\newtheorem{theorem}{Theorem}[section]

\newtheorem{corollary}[theorem]{Corollary}

\newtheorem{lemma}[theorem]{Lemma}

\newtheorem{proposition}[theorem]{Proposition}

	{\par\noindent{\bf Proposition \ref{res:hiper}.}\!\!
	\nopagebreak\normalsize\it}%

	{\par\noindent{\bf Theorem \ref{result43}.}\!\!
	\nopagebreak\normalsize\it}%

	{\par\noindent{\it Sketch of the proof}.  
	\nopagebreak\normalsize}%
	{\hfill\linebreak[2]\hspace*{\fill}$\circlearrowleft$}
	
\newenvironment{proof1}%
	{\par\noindent{\it Proof of Theorem }\ref{lfdhi}.\nopagebreak\normalsize}%
	{\hfill\linebreak[2]\hspace*{\fill}$\circlearrowleft$}
	
\newenvironment{proof428}%
	{\par\noindent{\it Proof of Theorem }\ref{317em}.\nopagebreak\normalsize}%
	{\hfill\linebreak[2]\hspace*{\fill}$\circlearrowleft$}

{\theoremstyle{definition}
       \newtheorem{definition}[theorem]{Definition}
	\newtheorem{claim}[theorem]{Claim}
       
       \newtheorem{remark}[theorem]{Remark}
       \newtheorem{example}[theorem]{Example}
       \newtheorem{parrafo}[theorem]{{\!}}  }

\numberwithin{equation}{theorem}

 \newcommand{\cali}{{\mathcal {I}}}
\newcommand{\calo}{{\mathcal {O}}}

\DeclareMathOperator{\Max}{\underline{Max}}

\DeclareMathOperator{\ord}{ord}

\DeclareMathOperator{\Reg}{Reg}
\DeclareMathOperator{\Sing}{Sing}
\DeclareMathOperator{\Spec}{Spec}
\DeclareMathOperator{\word}{w-ord}

\newcommand{\B}{{\mathcal B}}

\newcommand{\id}[1]{\langle #1 \rangle}

\definecolor{darkpurple}{rgb}{0.28,0.24,0.55}
\definecolor{lightblue}{rgb}{0,0.75,1}

\title{Some natural properties of constructive resolution of singularities} 
\tiny
\author{Ang\'elica Benito}
\address{Dpto. Matem\'aticas,  Universidad
Aut\'onoma de Madrid and Instituto de Ciencias Matem\'aticas CSIC-UAM-UC3M-UCM, 
Cantoblanco, 28049 Madrid, Spain}
\email{angelica.benito@uam.es}

\author{Santiago Encinas}
\address{Dpto. Matem\'atica Aplicada, Universidad de Valladolid, 47014 Valladolid, Spain}
\email{sencinas@maf.uva.es}

\author{Orlando E. Villamayor U.}
\address{Dpto. Matem\'aticas,  Universidad
Aut\'onoma de Madrid and Instituto de Ciencias Matem\'aticas CSIC-UAM-UC3M-UCM, Cantoblanco, 28049 Madrid, Spain}
\email{villamayor@uam.es}

 \thanks{The authors are partially supported by MTM2009-07291.}
\date{\today}

\dedicatory{To Professor H. Hironaka on his 80th birthday}

\begin{document}

\begin{abstract}
These expository notes, addressed to non-experts, are intended to present some of Hironaka's ideas on his theorem of resolution of singularities. We focus particularly on those aspects which have played a central role in the constructive proof of this theorem.


In fact, algorithmic proofs of the theorem of resolution grow, to a large extend, from the so called Hironaka's fundamental invariant. Here we underline the influence of this invariant in the proofs of the natural properties of constructive resolution, such as: equivariance, compatibility with open restrictions, with pull-backs by smooth morphisms, with changes of the base field, independence of the embedding, etc.

\end{abstract}
\maketitle

{\tableofcontents}

\section{Introduction}

\begin{parrafo} Hironaka's fundamental theorem in \cite{Hir64} proves resolution of singularities in characteristic zero. 

\begin{theorem}\label{buenosair}{\rm ({\bf Hironaka})}. If $X$ is a variety over a field of characteristic zero, there is a proper and birational morphism
\begin{equation}\label{lopri}
X  \stackrel{\pi}{\longleftarrow} X'
\end{equation}
 such that:
\begin{itemize}
\item[i)] $\pi$ is an isomorphism over the open set $U=\Reg(X)$ of
regular points.

\item[ii)] $X'$ is smooth.

\item[iii)] $\pi^{-1}(\Sing(X))$ is a union of smooth hypersurfaces in $X'$ having only normal crossings.
\end{itemize}

\end{theorem}

A morphism $X  \stackrel{\pi}{\longleftarrow} X'$, as above, is called {\em a resolution of singularities} of $X$. He also shows that this morphism can be constructed as a composition, say
 \begin{equation}\label{rsni}
\xymatrix@C=2.5pc{
 X_0=X & X_1\ar[l]_{\ \ \ \ \ \pi_{Y_1}} &\dots  \ar[l]_{\pi_{Y_2}} &  X_r=X'\ar[l]_{\!\!\!\!\pi_{Y_r}} 
}
\end{equation}
where each $\pi_{Y_i}$ is the blow-up along a smooth center $Y_i (\subset X_i)$.

There can be many resolutions of $X$, and the proof in \cite{Hir64} shows that a resolution always exists. {\em Constructive proofs} of Hironaka's theorem go one step beyond. They provide, for each singular reduced scheme $X$, a specific resolution, called the {\em constructive resolution} of $X$. In other words, they give a procedure to resolve singularities in such a way that it has very natural properties.
 Suppose now that (\ref{rsni}) is {\em the} resolution of $X$ provided by this procedure, it has the property:
\begin{enumerate}
\item \emph{Compatibility with smooth morphisms}: A smooth morphism $\widetilde{X}\overset{\sigma}{\longrightarrow}X$ provides, by taking successive pull-backs, a sequence
 \begin{equation}\label{eqin2}
\xymatrix@C=2.5pc{
 \widetilde{X} & \widetilde{X}_1\ar[l]_{\pi_{\widetilde{Y}_1}} &\dots  \ar[l]_{\pi_{\widetilde{Y}_2}} &  \widetilde{X}_r\ar[l]_{\pi_{\widetilde{Y}_r}} 
};
\end{equation}
and the property is that (\ref{eqin2}) is the constructive resolution of $\widetilde{X}$.

\item \emph{Lifting of group action}: If a group acts on $X$ it also acts on the resolved scheme $X_r$.  
\item \emph{Compatibility with change of the base field} (Similar to (1)).
\end{enumerate}
We shall later formulate the theorem of {\em embedded} resolution, say of $X$ embedded in a smooth scheme $W$, and another property of constructive resolution is that it can be easily adapted so that the resolution $X$ is independent of the embedding in $W$. These 
matters will be discussed along this paper.
\end{parrafo}

\begin{parrafo}
Essential to the argument used in Hironaka's proof is a reduction of the problem of resolution of singularities. A new problem is formulated, in terms of the data $(W,(J,b))$, where $W$ is a smooth scheme over a field $k$, $J$ a non-zero sheaf of ideals, and $b$ a positive integer. This new problem is stated as a {\em resolution} of $(W,(J,b))$ as we describe below. These data define a closed subset in $W$
$$\Sing(J,b)=\{x\in W\ |\ \nu_x(J)\geq b\},$$
where $\nu_x(J)$ denotes the order of $J$ at the point $x$. If $Y$ is a smooth irreducible subscheme included in $\Sing(J,b)$, and if $W\overset{\pi_Y}{\longleftarrow} W_1$ is the blow-up along $Y$, then there is a factorization
$$J\calo_{W_1}=I(H)^{b}J_1,$$
where $H\subset W_1$ denotes the exceptional hypersurface. The pair $(J_1,b)$ is said to be the \emph{transform} of $(J,b)$. We will consider $J_1$ together with the factorization 
$$J_1=I(H)^{c}\overline{J}_1,$$
where $c=\nu_y(J)-b$  and $y$ is the generic point of $Y$.

Similarly, given an iteration of transformations, say
 \begin{equation}\label{eqin31}
\xymatrix@C=2pc@R=0pc{
(J,b) & (J_1,b) &  &(J_r,b)\\
W & W_1\ar[l]_{\pi_{Y_1}} &\dots  \ar[l]_{\pi_{Y_2}} &  W_r\ar[l]_{\pi_{Y_r}}  
}
\end{equation}
the transform $J_r$ is endowed with a factorization of the form $J_r=I(H_1)^{c_1}\dots I(H_r)^{c_r}\overline{J}_r$, where each $H_i$ is the exceptional hypersurface introduced by the blow-up $W_{i-1}\overset{\pi_{Y_i}}{\longleftarrow}W_i$.

We will always assume that such a sequence is constructed in such a way that the strict transforms of the $r$ exceptional hypersurfaces, say $\{H_1,\dots,H_r\}$, have normal crossings in $W_r$. 

Hironaka's reformulation (the reduction) of the problem of resolution can be stated as follows: 

\vspace{0.4cm}

\noindent{\bf Problem:} Given $(W,(J,b))$, construct a sequence (\ref{eqin31}) in such a way that $\Sing(J_r,b)=\emptyset$. 

We will explain in Section \ref{sect2}, starting in
 \ref{reform},  why the solution of this problem leads to resolution of singularities.
\vspace{0.3cm}

Given a sequence of transformations (\ref{eqin31}), Hironaka defines, for all index $i$, the functions:
$$\ord_i:\Sing(J_i,b)\longrightarrow \mathbb{Q},\quad\ord_i(x)=\frac{\nu_x(J_i)}{b}.$$
These are called {\em Hironaka's functions}, and they paved the way to constructive proofs of resolution of singularities. 
A by-product of Hironaka's functions, which makes use of the factorization 
$J_i=I(H_1)^{c_1}\dots I(H_i)^{c_i}\overline{J}_i$, are the functions
$$\word_i:\Sing(J_i,b)\longrightarrow\mathbb{Q},\quad\word_i(x)=\frac{\nu_x(\overline{J}_i)}{b}.$$
All these functions take only finitely many values. Denote by $\max\word_i$ the maximum value achieved by $\word_i$. If (\ref{eqin31}) is constructed with the extra condition that $Y_i\subset\Max \word_i$, where
$$\Max \word_i=\{x\in W_i\ |\ \word_i(x)=\max\word_i\},$$ then it can be shown that
\begin{equation}\label{indes1}
\max\word_0\geq \max\word_1\geq\dots\geq\max\word_r.
\end{equation}
Note that $\max\word_r=0$ if and only if $\overline{J}_r=\calo_{W_r}$ and $J_r=I(H_1)^{c_1}\dots I(H_r)^{c_r}$. If this happens, then it is easy to extend the sequence (\ref{eqin31}) to a resolution.

\vskip 0.5cm

In the constructive resolution treated here, we view $W$ as a smooth subscheme of a smooth scheme $N$. We also consider a set $F$ of smooth hypersurfaces in $N$ having only normal crossings.
%

The sequence (\ref{eqin31}) gives rise to a sequence 
 \begin{equation}\label{eqin41}
\xymatrix@C=2.5pc{
N & N_1\ar[l]_{\pi_{Y_1}} &\dots  \ar[l]_{\pi_{Y_2}} &  N_r\ar[l]_{\pi_{Y_r}}  
}
\end{equation}
(same blow-ups) together with closed immersions $W_i\subset N_i$ (for any index $i$), where each $W_i$ is identified with the strict transform of $W_{i-1}$.  A main problem in constructive resolution can be formulated as follows:

\vspace{0.2cm}

\noindent{\bf Main Problem}: Fix:
\begin{itemize}
\item $(W,(J,b))$,  as above.
\item A smooth scheme $N$ and $F=\{H'_1, \dots ,H'_s\}$ a set of smooth hypersurfaces in $N$ with only normal crossings.
\item A closed embedding $W\subset N$.
\end{itemize}
 The problem is to construct a sequence (\ref{eqin31}), with centers $Y_i$, so that:
\begin{enumerate}
\item  $\Sing(J_r,b)=\emptyset$, and 
 
\item the sequence (\ref{eqin41}) of $r$ blow-ups over the smooth scheme  $N$, induced by the sequence (\ref{eqin31}), is such that the strict transform of hypersurfaces in $F$ together with the $r$ exceptional hypersurfaces introduced (all together) have normal crossings in $N_r$.
\end{enumerate}
\vspace{0.2cm}
 
Note here that there is an added difficulty over the previous Problem. The task is to construct a sequence (\ref{eqin31}), which fulfills the property (\ref{indes1}) and leads to $\Sing(J_r,b)=\emptyset$ ,  with an additional constraint on each center $Y_i$. For instance the first center $Y_1$, included in $W_1 (\subset N_1)$, should have normal crossings with the hypersurfaces in $F$, and there is no information on how hypersurfaces of $F$ intersect the closed subscheme $W$.

For each index $i$, let $F_i$ be the set of hypersurfaces with normal crossings in 
$N_i$, which consists of the strict transform of hypersurfaces in $F$, together with the $i$ exceptional hypersurfaces introduced.

This main problem can be solved by giving a canonical choice of centers $Y_i$. The key for this will be our two coordinate functions:
$$(\word_i^{(d)},n_i^{(d)}):\Sing(J_i,b)\longrightarrow \mathbb{Q}\times\mathbb{Z},$$
where $d=\dim W$, called the {\em inductive functions}. Essentially, it is through these function, and induction on $d$, that we can define a string of invariants which indicates how to choose each  $Y_i$.  

Here each $\word_i^{(d)}$ is the function previously defined, and $n_i^{(d)}$ is defined by counting the hypersurfaces of a certain subset of $F_i$, say $F_i^-$. This set $F_i^-$ is entirely defined in terms of the inequalities (\ref{indes1}). Hence, the function $n_i^{(d)}$ is also a by-product of Hironaka's functions.

We follow here, with some variations, the scheme and notation in \cite{Villa89}  and \cite{Villa92}, particularly on the latter where 
some of the natural properties of constructive resolution are addressed.
The goal of this presentation is to explain Hironaka's approach in \cite{Hironaka77}.
 So we put here special emphasis on Hironaka's reformulation (reduction) of the resolution problem, and why this led to the proofs of the natural properties of constructive resolution mentioned before. 
We refer to \cite{BHV} (see also \cite{EncVil97:Tirol}) for the relation of this development with more recent work on the subject, and for some technical aspects of the algorithm. 

We thank the referee for here/his many useful suggestions, which have helped us to improve the presentation of this paper.

\end{parrafo}

\section{First definitions and formulation of the Main Theorems}\label{sect2}
\begin{parrafo} 




%

Recall that \emph{Constructive Resolution} is a procedure that indicates, given a singular reduced scheme $X$, how to choose the centers $Y_i$ to construct a resolution of singularities as in (\ref{rsni}). We can think of it as an algorithm in which, at each step, the input are the equations defining $X$, and the output are the equations defining the center $Y$; and the same for each index $1\leq i \leq r$. But if we think of equations defining $X$, it is natural to embed the scheme in a smooth scheme, say $X\subset W$. This can be done locally, as we shall always consider here $X$ to be a scheme of finite type (see \cite{GabSch} for an implementation).

\end{parrafo}
 
\begin{parrafo}\label{p1}  Let us fix some notation needed for the formulation of the theorem of \emph{embedded resolution}. A pair $(W_0,E_0)$ denotes here a smooth scheme $W_0$, and a set  $E_0=\{H_1, \dots, H_s\}$ of hypersufaces with normal crossings in $W_0$. The sequence 
\begin{equation}\label{ressol}
\xymatrix@C=2.5pc{
(W_{0}, E_0) & (W_{1}, E_1)\ar[l]_{\ \ \pi_{Y_1}} &\dots  \ar[l]_{\ \ \ \ \ \pi_{Y_2}} &  (W_{r}, E_r)\ar[l]_{\!\!\! \pi_{Y_r}} 
}\end{equation}
denotes a composition of blow-ups, where each center  $Y_j\subset
W_j$ is closed, smooth, and has normal crossings with $E_j$. 
If $H_{s+j+1}\subset W_{j+1}$ denotes the exceptional
hypersurface of $W_{j}\stackrel{\pi_{Y_{j+1}}}{\longleftarrow}  W_{j+1}$, then $E_{j+1}$ is defined as the union of the strict transform of hypersurfaces in $E_j$, together with $H_{s+j+1}$. 
Centers $Y_j$ which have normal crossings with $E_j$ are said to be {\em permissible} for $(W_j,E_j)$.
\end{parrafo}

\begin{theorem}[{\bf Embedded Resolution of
Singularities}] \label{classical}
Given a smooth scheme $W_0$ over a field $k$ of characteristic zero, and $X_0$, a closed and reduced subscheme of $W_0$, there is a sequence of blow-ups as in (\ref{ressol}), say
 \begin{equation}
\label{res2sol}
\xymatrix@C=2.5pc@R=0pc{
X_0 & X_1 & & X_r\\
(W_{0}, E_0) & (W_{1}, E_1)\ar[l]_{\ \ \pi_{Y_1}} &\dots  \ar[l]_{\ \ \ \ \ \pi_{Y_2}} &  (W_{r}, E_r)\ar[l]_{\!\!\! \pi_{Y_r}} 
}\end{equation}
where each $X_i$ denotes the strict transform of $X_{i-1}$, and so that:
\begin{enumerate}
\item[(i)] The hypersurfaces of $E_r$ have normal crossings in $W_r$.

 \item[(ii)]
 $W_0\setminus [ \Sing(X_0) \cup \bigcup_{H_i\in E_0}H_i]
 = W_r\setminus\bigcup_{H_i\in E_r}H_i$.
\item[(iii)]  $X_{r}$ is smooth
and has normal crossings with  $\bigcup_{H_i\in E_r}H_i$.
\end{enumerate}

So if $E_0=\emptyset$, the morphism
$ X_0 \stackrel{{\pi}}{\longleftarrow}  X_r$, induced by $(\ref{res2sol})$, is a resolution of singularities as in \ref{buenosair}.
\end{theorem}

\begin{parrafo}{\em On Constructive Resolution.}
Fix a topological space $X$, and a totally ordered set $(T, \geq
)$. In this work an upper semi-continuous function $g: X \rightarrow T$ is a function with the following properties: 
\begin{enumerate}
\item[(i)] $g$ takes  
only finitely many values, and 
\item[(ii)] for any $\alpha \in T$
the set $\{x\in X \ |\ g(x)\geq \alpha \}$ is closed in $X$. 
\end{enumerate}
The largest value achieved by $g$ will be denoted by
$\max g$. And $\Max g$ will denote the set of points in 
$X$ where $g$ takes its highest value ($\max g$). So $\Max g$ is closed in $X$.

In Theorem \ref{classical}, $X_0$ is a reduced closed subscheme in a smooth scheme $W_0$. Constructive resolution also applies in this context. It makes use of a specific totally ordered set $(T, \geq)$. Fix a closed immersion, say $X_0\subset W_0$, and $(W_0,E_0)$ as before, then either $X_0$ is smooth and has normal crossings with $E_0$, or an upper semi-continuous function
$ f_0: X_0 \longrightarrow T$
is defined. It has the property that if $Y_0=\Max f_0$ (the set of points where the function takes its maximum value), then $Y_0$ is smooth, has normal crossings with $E_0$, and the blow-up along $Y_0$ provides a diagram, say: 
 \begin{equation}\label{la1}
\xymatrix@R=0pc@C=2.5pc{
 X_0 &  X_1\\
(W_{0}, E_0)  & (W_1,E_1)\ar[l]_{\ \pi_{Y_{0}}}
}\end{equation}
where $X_1\subset W_1$ is  the strict transform of $X_0$. 
Again, either $X_1$ is smooth and has normal crossings with $E_1$ or $ f_1: X_1 \longrightarrow T$ is defined. 

In this latter case, the function is such that $Y_1=\Max f_1$ is smooth, and has normal crossing with $E_1$. The blow-up along $Y_1$ provides 
\begin{equation}\label{la2}
\xymatrix@R=0pc@C=2.5pc{
 X_0 &  X_1 & X_2\\
(W_{0}, E_0)  & (W_1,E_1)\ar[l]_{\ \pi_{Y_{0}}} & (W_2,E_2)\ar[l]_{\ \ \pi_{Y_{1}}}
}
\end{equation}
Assume inductively that for a given index $s$, a sequence
 \begin{equation}\label{la3}
\xymatrix@R=0pc@C=2.9pc{
 X_0 &  X_1 & X_2 & & X_{s}\\
(W_{0}, E_0)  & (W_1,E_1)\ar[l]_{\ \pi_{Y_{0}}} & (W_2,E_2)\ar[l]_{\ \ \pi_{Y_{1}}} & \dots\ar[l]_{\ \ \ \ \ \pi_{Y_{2}}} & (W_{s},E_{s})\ar[l]_{\!\!\!\!\pi_{Y_{s-1}}}
}\end{equation}
is constructed by setting $Y_i$ in terms of the function $f_i:X_i \longrightarrow T$, $0\leq i \leq s-1$. Then,  either $X_{s}$ is smooth and has normal crossings with 
$E_{s}$ (the sequence is an embedded resolution), or a function $ f_s: X_s \longrightarrow T$
is defined with the property that  $Y_s=\Max f_s$ is smooth, and has normal crossings with $E_s$.

Note here that we take as initial data: $X_0\subset W_0$, $(W_0,E_0)$, and that {\em the functions with values on} $T$ enables us to construct a  sequence (\ref{la3}). The point is that for some index $r$ the sequence
\begin{equation}
\label{ressm}
\xymatrix@R=0pc@C=2.9pc{
 X_0 &  X_1 & X_2 & & X_{r}\\
(W_{0}, E_0)  & (W_1,E_1)\ar[l]_{\ \pi_{Y_{0}}} & (W_2,E_2)\ar[l]_{\ \ \pi_{Y_{1}}} & \dots\ar[l]_{\ \ \ \ \ \pi_{Y_{2}}} & (W_{r},E_{r})\ar[l]_{\!\!\!\!\!\pi_{Y_{r-1}}}
}\end{equation}
is such that $X_r$ is smooth and has normal crossings with the hypersurfaces in $E_r$. Moreover, all centers $Y_i$ will be included in $\Sing(X_i) \bigcup (\cup H_i)$, which ensures that this is an embedded resolution.

It is essential to point out that the set $T$ is {\em universal}, namely  for {\em any}
$X_0\subset W_0$, $(W_0,E_0)$, the functions $f_i$, which provide the embedded resolution (\ref{ressm}), take values on the same $T$. We shall indicate later how $T$ and the functions $f_i$ are defined. Here (\ref{ressm}) is said to be the {\em Constructive Resolution} 
of the data $X_0\subset W_0$, $(W_0,E_0)$.
It is constructed by the upper semi-continuous functions:
\begin{equation}\label{fdri}
f_0: X_0 \longrightarrow T, \quad \ f_1: X_1 \longrightarrow T, \ \dots \ \ \ f_{r-1}: X_{r-1} \longrightarrow T,
\end{equation}
and these functions depend also on $E_0$ (and on each  $E_i$).
Note that  (\ref{ressm})  is determined by the centers $Y_i$, and 
$Y_i=\Max f_i (\subset X_i$). So (\ref{ressm}) can be reconstructed from (\ref{fdri}).

 Neglecting the ambient spaces $W_i$, $0\leq i \leq r$ in  (\ref{ressm}) we get
 \begin{equation}\label{rsn4}
\xymatrix@C=2.5pc{
 X_0 & X_1\ar[l]_{\ \ \ \ \ \pi_{Y_1}} &\dots  \ar[l]_{\pi_{Y_2}} &  X_r\ar[l]_{\!\!\!\!\pi_{Y_r}} 
,}\end{equation}
which is also determined by the functions in (\ref{fdri}), as $Y_i=\Max f_i$, and $f_i$ is a function on $X_i$.

Note that if $E_0=\emptyset$ then the latter fulfills the three conditions in Theorem \ref{buenosair}. In other words if the initial data is $X_0\subset W_0$, $(W_0,E_0=\emptyset)$, then the functions in 
(\ref{fdri}) (which depend on $E_0$), provide a resolution of singularities of $X_0$. We say that the embedded resolution of $X_0\subset W_0$ defines a non-embedded resolution of $X_0$.

A property of the algorithm is that it can be defined so that the non-embedded resolution is independent of the embedding in $W_0$. In fact it suffices that $X_0$ be only locally embedded, and the induced non-embedded resolutions is well defined: Suppose that $X_0$ is embedded in another smooth scheme, say $X_0\subset W'_0$, and we take as initial data $X_0\subset W'_0$, $(W'_0,E'_0=\emptyset)$. A property of this procedure will be that, 
in this case, we get the same data (\ref{fdri}) (the same schemes $X_i$, and the same functions $f_i$). In particular, when the input datum is simply $X_0$, we get the same
resolution (\ref{rsn4}).
In this way constructive resolution leads to the resolution of abstract varieties.

\begin{parrafo}\label{pt5}{\bf A Property of Constructive Resolution}.  

Take $X_0\subset W_0$, $(W_0, E_0)$, and let (\ref{ressm}) be the  sequence obtained by the constructive resolution. Let $\sigma_0: V_0 \longrightarrow W_0$ be a smooth morphism. For each index $i$, $0\leq i \leq r$, consider the sequence of blow-ups
 \begin{equation}\label{}
\xymatrix@C=2.5pc{
 W_0 & W_1\ar[l]_{\ \ \ \ \ \pi_{Y_1}} &\dots  \ar[l]_{\pi_{Y_2}} &  W_i\ar[l]_{\!\!\!\!\pi_{Y_i}} 
,}\end{equation}
By taking fiber products of $\sigma_0$ with this sequence one obtains smooth schemes $V_i$ and smooth morphisms, 
say $\sigma_i: V_i \longrightarrow W_i$. Now set 
$X'_i= \sigma_i^{-1}(X_i)$, and let $E'_i$ be defined by taking pull-backs of the hypersurfaces in $E_i$. Define also 
$$ f'_i: X'_i\longrightarrow T$$
by setting $f'_i(x)=f_i(\sigma_i(x))$, for each $f_i$ in (\ref{fdri}).

In this way, if we fix $X_0\subset W_0$, $(W_0, E_0)$, and  
a smooth morphism $\sigma_0: V_0 \longrightarrow W_0$, the constructive resolution (\ref{ressm}) gives rise to a sequence 
\begin{equation}\label{resssns}
\xymatrix@R=0pc@C=3.7pc{
 X'_0 &  X'_1  & & X'_{r}\\
(V_{0}, E'_0)  & (V_1,E'_1)\ar[l]_{\ \pi_{\sigma_0^{-1}(Y_0)}} & \ \ \dots\ \ \ar[l]_{\ \ \ \ \ \pi_{\sigma_1^{-1}(Y_1)}} & (V_{r},E'_{r})\ar[l]_{\!\!\pi_{\sigma_{r-1}^{-1}(Y_{r-1})}}
}
\end{equation}
and to functions:
\begin{equation}\label{fdrj1}
 f'_0: X'_0\longrightarrow T, \quad \ f'_1: X'_1\longrightarrow T, \ \ \dots \quad, f'_{r-1}: X'_{r-1} \longrightarrow T.
 \end{equation}

With the setting as above, the data in (\ref{fdrj1}), and hence  (\ref{resssns}), will coincide with those obtained from the constructive resolution when the input data are  $X'_0\subset V_0$, $(V_0,E'_0)$. Namely, the functions 
$f'_i: X'_i\longrightarrow T$ coincide with the functions defining the constructive resolution. 

We express this property by saying that the constructive resolution is {\em compatible with pull-backs obtained by smooth morphisms}.
This will encompass restrictions to open sets, and also to \'etale neighborhoods. This last property is useful when applying arguments which require \'etale topology. Further properties of constructive resolution, such as equivariance (by group actions on $X$), compatibility with change of the base fields (at least for finite field extensions), and others, grow from this naive compatibility. 




A similar property will hold for the non-embedded case, or say for the constructive resolution of an abstract variety. Let  (\ref{rsn4})
be the resolution obtained when the input datum is $X_0$, and let $\sigma_0: X'_0 \longrightarrow X_0$ be a smooth morphism. By taking fiber products with (\ref{rsn4}) we get
\begin{equation}\label{rsn6}
\xymatrix@C=2.5pc{
 X'_0 & X'_1\ar[l]_{\ \ \ \ \ \pi_{Y'_1}} &\dots  \ar[l]_{\pi_{Y'_2}} &  X'_r=X'\ar[l]_{\!\!\!\!\pi_{Y'_r}} 
}
\end{equation}
($ Y'_i=\sigma_i^{-1} (Y_i) $), and smooth morphisms $\sigma_i:X_i'\longrightarrow X_i$, and also functions, say
$$ f'_0: X'_0\longrightarrow T, \quad \ f'_1: X'_1\longrightarrow T, \ \ \dots \quad,  f'_{r-1}: X'_{r-1} \longrightarrow T,$$
($f'_i(x)=f_i(\sigma_i(x))$).
The property is that this is the constructive resolution of $X'_0$. In other words, the functions $f'_i: X'_i\longrightarrow T$ are those defined by the constructive resolution when the input datum is $X'_0$.
\end{parrafo}

Closely related to resolution of singularities is Hironaka's principalization theorem. Also known as Log-principalization of Ideals.
\end{parrafo}

\begin{theorem}[{\bf Principalization of
ideals}] \label{principalization}
Let $W_0$ be a smooth scheme over a field $k$ of characteristic zero.
Given $ J\subset\calo_{W_{0}} $, a non-zero sheaf of ideals, there
is a sequence of blow-ups along closed smooth centers, say (\ref{ressol}),  such that:
\begin{enumerate}
\item[(i)] The morphism $W_0\longleftarrow W_r$ is an
isomorphism over $W_0\setminus V(J)$. \item[(ii)] The sheaf
$J\calo_{W_r}$ is invertible and supported on a divisor with normal
crossings, i.e.,
\begin{equation}
{\mathcal {L}}=J\calo_{W_r}=\cali(H_1)^{c_1}
\cdot\cdots\cdot\cali(H_s)^{c_s},
\end{equation}
where $E^{\prime}=\{ H_1,H_2,\cdots, H_s\}$ are regular
hypersurfaces with only normal crossings.
\end{enumerate}

\end{theorem}
Here some hypersurfaces in $E^{\prime}$ might not be components of the exceptional locus of  $W_0\longleftarrow W_r$.

\begin{parrafo} {\bf  A first reformulation of the problem}.\label{reform} 

There is a reformulation of the resolution problems. It applies to Theorem of Embedded Resolution of Singularities (and consequently to Theorem of Non-Embedded Resolution of Singularities ), and also to Theorem of Principalization of ideals.
The rest of this Section 2 is devoted to the discussion of this reformulation, which appears already in \cite{Hir64}.

Recall that the input data in constructive (embedded) resolution are of the form $X\subset W$, $(W,E)$, and that the outcome is a resolution (\ref{ressm}). Each step is obtained by blowing up along a smooth scheme:
 \begin{equation}\label{la11}
\xymatrix@C=2pc@R=0pc{
 X_0 & X_1\\
(W_0 , E_0) & (W_1 ,E_1)\ar[l]_ {\pi_{Y_{0}}}
}\end{equation}
Here $X_1$ is the strict transform of $X_0$, and one can view  each 
$X_i$ as a closed subset of $W_i$ ($i=0,1$).
Hironaka points out that there is another context in which the data undergo a very similar law of transformation: Fix $(W_0 , E_0) $ as before, a coherent non-zero sheaf of ideals $J_0$ in $\calo_{W_0}$, and an integer $b>0$. We say that the 2-tuple $(J_0,b)$ is a \emph{pair}, and that $\B_0=(W_0, (J_0,b), E_0)$ is a \emph{basic object}.

Let $\nu_x(J_0)$ denote the order of $(J_0)_x$ at $\calo_{W_0,x}$. Define the \emph{singular locus of} $(J,b)$ as:
$$\Sing(J_0,b)=\{ x\in W_0\ |\ \nu_x(J_0)\geq b\}$$
which is a closed subset in $W_0$.

Let $Y_0$ be a closed smooth subscheme included in $\Sing(J_0,b)$ $(Y_0\subset \Sing(J_0,b)$), and assume that it has normal crossings with the hypersurfaces in $E_0$.
Let $W_0 \overset{\pi_{Y_0}}{\longleftarrow} W_1$ denote the blow-up along $Y_0$, and let $H_1$ be the exceptional hypersurface. There is a factorization
$$J_0\calo_{W_1}=I(H_1)^bJ_1$$
for some sheaf of ideals $J_1$ in $W_1$.
Define now $(J_1,b)$ as \emph{the transform of} $(J_0,b)$ in $W_1$; and set
 \begin{equation}\label{la12}
\xymatrix@R=0pc@C=2pc{
 (J_0,b) & (J_1,b) \\
(W_0 , E_0) &  (W_1 ,E_1)\ar[l]_{\pi_{Y_{0}}}
}\end{equation}
We also say that $\B_1=(W_1,(J_1,b),E_1)$ is the \emph{transform} of $\B_0$. Note that 
 the data are $\B_0=(W_0,(J_0,b), E_0)$, and that the transformation 
is defined when the center $Y_0$ is included in $\Sing(J_0,b)$ and has normal crossings with $E_0$. So, we always require this condition $Y_0 \subset \Sing(J_0,b)$, and if not, the transformation is not defined.

Given $\B_0=(W_0,(J_0,b),E_0)$, we say that
\begin{equation}\label{riic}
\xymatrix@R=0pc@C=2pc{
 (J_0,b)  & (J_1,b)  & & (J_{r},b)\\
(W_{0}, E_0) & (W_{1}, E_1)\ar[l]_{\pi_{Y_{0}}} & \cdots\ar[l]_{\pi_{Y_{1}}}   &(W_{r}, E_r)\ar[l]_ {\pi_{Y_{r-1}}}
}
\end{equation}
is a {\em resolution of the basic object} $\B_0$ if $\Sing(J_r,b)=\emptyset$.

We introduce now the notion of the pull-back of a basic objects by a smooth morphism. This might seem artificial at first sight as the resolution problem involves only blow-ups, which are birational, whereas a smooth morphism might not be birational. 
Special attention will be drawn here to the compatibility of constructive resolution with the pull backs by smooth morphisms(see \ref{pt5}), and the importance of this property will show up along the exposition.

Fix a basic object $\B=(W,(J,b),E)$  and a smooth morphism 
$\sigma: W' \longrightarrow W$. Define $\B'=(W',(J',b) ,E')$ where
$J'=J\calo_{W'}$, and $E'$ the set of pull-backs of hypersurfaces in 
$E$. Note that $\Sing(J',b)=\sigma^{-1}(\Sing(J,b))$.
Here $\B'$ is called the \emph{pull-back of} $\B$ by $\sigma: W' \longrightarrow W$, and it is denoted by
\begin{equation}\label{ldob}
\xymatrix@R=0pc@C=2pc{
 (J,b)  & (J',b) \\
(W , E) & (W' ,E')\ar[l]
}
\end{equation}
To avoid confusion it will be explicitly indicated when this notation applies to the pull-backs of 
basic objects, and when to the transformations obtained by a blow-up.
\end{parrafo}

\begin{parrafo}\label{crobo} 

Constructive resolution also applies to basic objects. It constructs a {\em resolution of a basic object} by means of suitably defined upper semi-continuous functions $g_i$ with values on a fixed totally ordered set $T$.
More precisely, given the input data $\B_0=(W_0,(J_0,b),E_0)$, it provides a specific  resolution sequence as (\ref{riic}), defined by upper semi-continuous functions:
\begin{equation}\label{nge}
 g_0: \Sing(J_0,b) \longrightarrow T, \quad \ g_1: \Sing(J_1,b) \longrightarrow T, \ \ \dots \quad , g_{r-1}: \Sing(J_{r-1},b) \longrightarrow T,
 \end{equation}
where 
 \begin{equation}
\begin{array}{ccc}
 (J_0,b) & & (J_1,b) \\
(W_0 , E_0) & \stackrel{\pi_{Y_{0}}}{\longleftarrow}  & (W_1 ,E_1)\\ 
\end{array}
\end{equation}
is defined by setting $Y_0=\Max g_0$, the set of points where $g_0: \Sing(J_0,b) \longrightarrow T$ takes its maximum value. Then set 
 \begin{equation}\label{rcdbo}
\begin{array}{ccc}
 (J_1,b) & & (J_2,b) \\
(W_1 , E_1) & \stackrel{\pi_{Y_{1}}}{\longleftarrow}  & (W_2 ,E_2)\\ 
\end{array}
\end{equation}
where $Y_1$ is defined as above by $g_1: \Sing(J_1,b) \longrightarrow T$, and so on. And it has the property that for some $r$ it provides a resolution. Namely, $\Sing(J_r,b)=\emptyset$.

A smooth morphism $\sigma_0: W_0' \longrightarrow W_0$ defines, say 
$\B'_0=(W'_0,(J_0',b),E_0')$, by taking pull-backs, and   $\Sing(J_0',b)=\sigma_0^{-1}(\Sing(J_0,b))$. Moreover, by taking successive pull-backs, (\ref{riic}) defines 
\begin{equation}\label{riicpr}
\xymatrix@R=0pc@C=2.8pc{
 (J'_0,b) & (J'_1,b) &  &  (J'_{r},b)\\
 (W'_{0}, E'_0) & (W'_{1}, E'_1)\ar[l]_{\pi_{Y'_0}}  & \cdots\ar[l]_{\ \ \ \ \ \ \pi_{Y_1'}}  &
(W'_{r}, E'_r)\ar[l]_{\!\!\!\!\pi_{Y'_{r-1}}},
}\end{equation}
together with smooth morphisms $\sigma_i:W'_i\longrightarrow W_i$, and functions:
\begin{equation}\label{ngee}
 g'_0: \Sing(J'_0,b) \longrightarrow T, \quad \ g'_1: \Sing(J'_1,b) \longrightarrow T, \ \ \dots \quad , g'_{r-1}: \Sing(J'_{r-1},b) \longrightarrow T,
 \end{equation}
by setting $g'_i(x)=g_i(\sigma_i(x)).$

This constructive procedure that leads to the resolution of basic objects, also has the property of compatibility with pull-backs, as discussed in \ref{pt5}. In fact, this last compatibility will lead us to the compatibility of constructive resolution with smooth morphism in \ref{pt5}.
\end{parrafo}

\begin{parrafo}{\bf Why do we consider basic objects?}\label{210}
  
  Hironaka points out that the problem of resolution of basic objects appears as a common ancestor of the theorem of embedded resolution and that of Log-principalization. 
  
  As for the latter, it is simple to check that a resolution of a basic object of the form $\B_0=(W_0, (J,1), E_0)$ (with $b=1$) defines 
  a Log principalization of $J$.
  
 Let us focus here on the relation with resolution of singularities.
Unfortunately it is not possible to attach to $X_0\subset W_0$, $(W_0,E_0)$, a basic object $\B_0=(W_0,(J_0,b), E_0)$, with the condition that $\Sing(J_0,b)=\Sing(X_0)$, and that this equality be preserved by transformations. When applying a blow-up as in (\ref{la11}), the law of transformations relating the ideal of definition of $X_0$, say $I(X_0) \subset \calo_{W_0}$, with that of $I(X_1) \subset \calo_{W_1}$ is called the {\em strict transform} of ideals. A law which is quite involved, whereas the law of transformation of basic objects in \ref{reform} is very simple.
 
 So the relation of resolution of basic objects with that of resolution of singularities requires some clarification. This leads us to the so called \emph{Hilbert Samuel function} and the \emph{Hilbert Samuel stratification}.

 Fix $x\in X$, and define $h: \mathbb N \longrightarrow \mathbb N$, where
 $h(k)=l(\calo_{X,x}/m^k_x)$ (length of the artinian ring). 
 The graph is an element in $\mathbb N^ \mathbb N$, say $h$ again. Order $\mathbb N^ \mathbb N$  lexicographically, and define the function 
 $$HS_X: X \longrightarrow \mathbb N^ \mathbb N, \ \ \ HS_X(x)=h.$$
 
This function is upper semi-continuous along the closed spectrum of $X$ (the subset of closed points of $X$), and can be easily extended (uniquely) to an upper semi-continuous function on $X$. Let $\overline h =\max HS_X$ denote the maximum value achieved by the function, and let $X(\overline h)\subset X$ be the set of points where such value is achieved. A closed point $x$ is in $X(\overline h)$ if and only if $HS_X(x)=h$.

 In general $X(\overline h)$ is not smooth. A Theorem of Bennett (see \cite{Ben}) states that if $Y\subset  X(\overline h)$ is a closed and  smooth subscheme, and if $X'$ is as in (\ref{la11}) (the strict transform of X), then 
  \begin{equation}\label{jurdw}
HS_{X}(\pi(x)) \geq HS_{X'}(x) 
 \end{equation}
 for any $x\in X'$. This is known as \emph{Bennett's inequality}. It ensures, in particular, that
 \begin{equation}\label{jurd}
\overline h= \max HS_{X}\geq \max HS_{X'}. 
 \end{equation}
 
 Now set $X'(\overline h)$ as the points in $X'$ where the function $HS_{X'}$ takes value $\overline h$. If $\overline h= \max HS_{X}> \max HS_{X'} $ then $X'(\overline h)=\emptyset$. But if 
 $\overline h= \max HS_{X'}$,  then 
 $X'(\overline h)$ is not empty and it makes sense to define
 \begin{equation}\label{la22}
\begin{array}{ccccc}
 X & & X'& &X''\\
W & \overset{\pi_Y}{\longleftarrow} & W' &\overset{\pi_{Y'}}{\longleftarrow} &W''
\end{array}
\end{equation}
by taking $Y'\subset X'(\overline h)$ as before. In this case 
$\overline h= \max HS_{X}= \max HS_{X'}\geq \max HS_{X''}. $

Define now $X''(\overline h)$ as before, thus it is empty if and only if 
$\overline h> \max HS_{X''}$. 
The following result of Hironaka shows that basic objects relate to this 
setting, in which we start with a closed immersion $X\subset W$, and consider $\overline h= \max HS_{X}$. The next theorem assigns to these data a basic object, and one should draw attention to the fact that this assignment is only local.
 
\begin{theorem}\label{H3}{\rm({\bf Hironaka})} (First version)  Fix $X\subset W$ and $(W,E)$, and set $\overline h =\max HS_X$. After replacing $X\subset W$ and $(W,E)$ by restriction to an \'etale cover of $W$, we may assume that there is a basic object  $\B=(W,(J,b),E)$, so that 

\begin{enumerate}
\item $\Sing(J,b)= X(\overline h)$.

\item Set $X_0=X$, $W_0=W$, $E_0=E$, $J_0=J$.
For any sequence of transformations
\begin{equation}\label{riicww1}
\begin{array}{cccccccc}
 X_0& & X_1 & & X_2& &  &X_r\\
(W_{0}, E_0) & \overset{\pi_{Y_0}}{\longleftarrow} & (W_{1}, E_1) & \overset{\pi_{Y_1}}{\longleftarrow} &
(W_{2}, E_2) &\cdots  &\overset{\pi_{Y_{r-1}}}{\longleftarrow} &
(W_{r}, E_r)
\end{array}
\end{equation}
constructed with centers $Y_i\subset X_i(\overline h)$, there is a sequence of transformations of basic objects, obtained with the same centers, say 
\begin{equation}\label{riicww}
\begin{array}{cccccccc}
 (J_0,b) & & (J_1,b) & & (J_2,b) & &  &(J_r,b)\\
(W_{0}, E_0) & \overset{\pi_{Y_0}}{\longleftarrow} & (W_{1}, E_1) & \overset{\pi_{Y_1}}{\longleftarrow} &
(W_{2}, E_2) &\cdots  &\overset{\pi_{Y_{r-1}}}{\longleftarrow} &
(W_{r}, E_r)
\end{array}
\end{equation}
and $$\Sing(J_i,b)=X_i(\overline h)$$
for all index $0\leq i \leq r$. Moreover, any sequence (\ref{riicww}) induces a sequence (\ref{riicww1}), and the previous equalities hold.
\end{enumerate}
\end{theorem}
The theorem says that if $HS_{X_i}$ is the Hilbert Samuel function of $X_i$, then $\Sing(J_i,b)$ is the closed set of points $x$ where $HS_{X_i}(x)=\overline h$. It also says that
$$ \overline h=\max HS_X > \max HS_{X_r}$$
if and only if (\ref{riicww}) is a resolution of the basic object.

In particular, if (\ref{riicww}) is a resolution of the basic object, then 
$$  \overline h=\max HS_X =\max HS_{X_1}=\cdots = \max HS_{X_{r-1}}> \max HS_{X_r}.$$

Note here that the Theorem says that {\em after restriction to a cover} there are basic objects attached to the highest value $\overline h$.  Theorem \ref{H3} guarantees the
existence of a basic object satisfying the properties as described there. However,
there may be many such basic objects, and a priori we don't know if the resolutions
of these basic objects, only locally defined over an \'etale cover, would patch to
provide the global resolution. Still, we state the following optimism that such a
patching can be done.
\begin{claim} (Optimistic) 
If we
know how to construct resolutions of the basic objects attached to the highest value $\overline{h}$ by this Theorem, then a sequence of blow-ups at closed smooth centers over $X\subset W$, $(W,E)$, can be constructed so that the maximum value of the Hilbert Samuel function, say $\overline h$, drops. 
\end{claim}
Before we carry on with the discussion and justification of the claim, let us indicate that there is another Theorem of Hironaka which says that, in order to prove resolution of singularities of $X$, it suffices to prove 
that given $X\subset W$, $(W,E)$, a sequences of transformation as above can be constructed so that 
$ \overline h=\max HS_X > \max HS_{X_r}$ (see Chapter 8 in \cite{EncVil97:Tirol}). Namely, this procedure will not go for ever, and if $X$ is a reduced variety, then this procedure applied finitely many times, will lead to say $ X_{r'}$ regular.
%

In other words, this last Theorem together with Theorem \ref{H3} says that resolution of singularities can be achieved if we know how to obtain resolution of basic objects in some clever way (so that the claim holds).

Theorem \ref{H3} says that given $X\subset W$, $(W,E)$, and setting $\overline h=\max HS_X$, one can attach to this value a basic object {\em after restriction to a cover of $W$}.
Let us emphasize that this provides local
solutions to the problem of resolution of singularities over the open subsets of
a cover, and that the remaining issue is to figure out how to patch
these local solutions to provide a global solution.

\end{parrafo}
\begin{parrafo} {\bf On weak equivalence and a closer view of our Optimistic Claim:}

Fix $X\subset W$, $(W,E)$ and a smooth morphism $\sigma: W' \longrightarrow W$, then new data, say $X'\subset W'$, $(W',E')$, are obtained by taking pull-backs. If $\sigma: W' \longrightarrow W$ is an open immersion what we get is the usual restriction.
Of course the formulation of resolution of singularities involves only blow-ups along smooth centers,
and not smooth morphisms. Pull-backs by smooth morphisms are to be thought of as {\em auxiliary} maps, and they will be essential in the proof of the previous claim.

%

\begin{definition}\label{lsbo}
Given a basic object $\B_0=(W_0,(J_0,b),E_0)$, define a {\em local sequence} as 
\begin{equation}\label{rlcww}
\begin{array}{cccccccc}
 (J_0,b) & & (J_1,b) & & (J_2,b) & &  &(J_r,b)\\
(W_{0}, E_0) & \longleftarrow & (W_{1}, E_1) & \longleftarrow &
(W_{2}, E_2) &\cdots  &\longleftarrow &
(W_{r}, E_r)\end{array}
\end{equation}
where each 
 \begin{equation}\label{nls}
\begin{array}{ccc}
 (J_i,b) & & (J_{i+1},b) \\
(W_i , E_i) & \longleftarrow & (W_{i+1} ,E_{i+1})
\end{array}
\end{equation}
is either:
\begin{enumerate}
\item[A)] a blow-up with center $Y_i \subset \Sing(J_i,b)$ as in (\ref{ldob}), say
  \begin{equation}\label{nls1}
\begin{array}{ccc}
 (J_i,b) & & (J_{i+1},b) \\
(W_i , E_i) & \overset{\pi_{Y_i}}{\longleftarrow} & (W_{i+1} ,E_{i+1})
\end{array}
\end{equation}

\item[B)] or obtained by a smooth morphism $\sigma_i: W_{i+1} \longrightarrow W_i$, and setting $(J_{i+1},b)$ and $(W_{i+1} ,E_{i+1})$ by pull-backs.
\end{enumerate}
\end{definition}

A basic object $\B_0=(W_0,(J_0,b),E_0)$ defines a closed set 
$\Sing(J_0,b)$ in $W_0$. Moreover, for any local sequence as above it defines the closed set $\Sing(J_i,b)$ in $W_i$, $0\leq i \leq r$. There can be many local sequences defined for $\B_0$. So there are many closed sets defined in different smooth schemes, starting with $\B_0$ and considering all possible local sequences (\ref{rlcww}). 

Now we introduce an equivalence among basic objects, so that two basic objects are equivalent if and only if they define the same family of closed sets:

\begin{definition}\label{weq}
Let $\B_0=(W_0,(J_0,b),E_0)$, and $\B_0'=(W_0,(K_0,d),E_0)$ be two basic objects (same $(W_0,E_0)$). They are said to be  {\em weakly equivalent} if every local sequence of $\B_0$, say 
\begin{equation}\label{rlcww1}
\begin{array}{cccccccc}
 (J_0,b) & & (J_1,b) &  & &  &(J_r,b)\\
(W_{0}, E_0) & \longleftarrow & (W_{1}, E_1) & \longleftarrow &\cdots  &\longleftarrow &
(W_{r}, E_r)
\end{array}
\end{equation}
defines a local sequence of $\B_0'$, say 
\begin{equation}\label{rlcww2}
\begin{array}{cccccccc}
 (K_0,d) & & (K_1,d) & & (K_2,d) & &  &(K_r,d)\\
(W_{0}, E_0) & \longleftarrow & (W_{1}, E_1) & \longleftarrow &
(W_{2}, E_2) &\cdots  &\longleftarrow &
(W_{r}, E_r)
\end{array}
\end{equation}
and $\Sing(J_i,b)=\Sing(K_i,d)$ $ 0\leq i \leq r.$
And conversely, any local sequence of $\B_0'$ defines a local sequence of $\B_0$ and both define the same closed sets.
\end{definition}

Intuitively we think of a basic object as an tool to define closed sets. 
In fact $\B_0=(W_0,(J_0,b), E_0)$ defines a closed set in $W_0$, namely $\Sing(J_0,b)$, it also defines closed sets by taking pull-backs by smooth morphisms and also by taking transforms as defined in (\ref{nls1}). So two basic objects $\B_0=(W_0,(J_0,b), E_0)$, and $\B_0'=(W_0,(K_0,d),E_0)$ are weakly equivalent when they define the same closed sets.
As a first example one can check that  $\B_0=(W_0,(J_0,b), E_0)$ and $\B'_0=(W_0,(J_0^2,2b), E_0)$ are equivalent. This abstract notion of equivalence will find now its justification as we formulate below the link of Hilbert Samuel stratification with basic objects in Theorem \ref{H4}.

Let us first extend the notion of local sequence for data of the form $X \subset W$, $(W, E)$.

\begin{definition}\label{fed14}
Given $X \subset W$, $(W, E)$, set $X_0=X$, $W_0=W$, $E_0=E$, and define a {\em local sequence} as 
\begin{equation}\label{rlcwwW}
\begin{array}{cccccccc}
 X_0 & & X_1 & & &  &X_r\\
(W_{0}, E_0) & \longleftarrow & (W_{1}, E_1) & \longleftarrow &\cdots  &\longleftarrow &
(W_{r}, E_r)
\end{array}
\end{equation}
where each 
 \begin{equation}
\begin{array}{ccc}
 X_i & & X_{i+1}\\
(W_i , E_i) & \longleftarrow & (W_{i+1} ,E_{i+1})
\end{array}
\end{equation}
is either:
\begin{enumerate}
\item[A)] A blow-up with center $Y_i$ having normal crossings with $E_i$ and included in $X_i$ as in (\ref{la1}), say
 \begin{equation}
\begin{array}{ccc}
 X_i & & X_{i+1} \\
(W_i , E_i) & \overset{\pi_{Y_i}}{\longleftarrow} & (W_{i+1} ,E_{i+1})
\end{array}
\end{equation}
Here $X_{i+1}$ denotes the strict transform of $X_i$.

\item[B)] A smooth morphism $\sigma_i: W_{i+1} \longrightarrow W_i$, and setting $X_{i+1}$ and $(W_{i+1} ,E_{i+1})$ by pull-backs.
\end{enumerate}
\end{definition}

\begin{parrafo}\label{ladeHS} The Hilbert Samuel function $HS_X: X \longrightarrow \mathbb N^\mathbb N$ can be defined for any scheme $X_i$ in the sequence, and they were compared in (\ref{jurdw}) for transformations of type A).

If  $\sigma_i: W_{i+1} \longrightarrow W_i$ is smooth and $x\in X_{i+1}$, we cannot claim that 
$HS_{X_{i+1}}(x)=HS_{X_{i}}(\sigma_i (x))$
unless the morphism is \'etale. But the value $HS_{X_{i+1}}(x)$ 
can be obtain from $HS_{X_{i}}(\sigma_i (x)) $ if we know the dimension of the fibers of $\sigma_i: W_{i+1} \longrightarrow W_i$. So even if it is not strictly true we will say that 
\begin{equation}\label{casosm}
HS_{X_{i+1}}(x)=HS_{X_{i}}(\sigma_i (x)).
\end{equation}
 Strictly speaking $HS_{X_{i+1}}$ stands here for a function which is not the Hilbert Samuel function but gives equivalent information. A precise statement about these facts can be found in \cite{EncVil97:Tirol}.
The following is a natural generalization of (\ref{jurd}).
\end{parrafo}
\begin{remark}\label{rk4}
Fix $X \subset W$, $(W, E)$. Set $X_0=X$, $W_0=W$, $E_0=E$, and fix a local sequence  (\ref{rlcwwW}). Let the functions $HS_{X_i}$ be defined as above, and set 
$$F_i=\Max HS_{X_i}=\{ x\in X_i\ |\ HS_{X_i}(x)=\max HS_{X_i}\}.$$
Assume that $Y_i \subset  F_i$ for every transformation of type A) in the sequence. Then 
$$ \max HS_{X_0}\geq  \max HS_{X_1} \geq \dots \geq  \max HS_{X_r}.$$
\end{remark}

For each sequence as above set  $\overline{h}=\max HS_X$ and let 
$X_i(\overline{h})$ be the set of points $x\in X_i$ so that $HS_{X_i}(x)=\overline{h}$. Note that $X_r(\overline{h})$ is not empty if and only if 
$$ \max HS_{X_0}= \max HS_{X_1} = \dots =  \max HS_{X_r}.$$
 \end{parrafo}
 
 The following extends the results in Theorem \ref{H3}, as it extends the class of morphisms involved among the transformations. It says that the link between the basic objects and Hilbert Samuel functions is even stronger as pull-backs obtained by smooth morphism are also considered.
 
 \begin{theorem}{\rm({\bf Hironaka})}\label{H4} Fix $X\subset W$ and $(W,E)$, and set $\overline h =\max HS_X$. After replacing $X\subset W$ and $(W,E)$ by restriction to a finite \'etale cover of $W$, a basic object $(W,(J,b),E)$ can be defined so that: 
\begin{enumerate}
\item  $\Sing(J,b)= X(\overline h)$.
\item Set $X_0=X$, $W_0=W$, $E_0=E$, $J_0=J$.
If a local sequence
\begin{equation}\label{rlcwwWW}
\begin{array}{cccccccc}
 X_0 & & X_1 & & &  &X_r\\
(W_{0}, E_0) & \longleftarrow & (W_{1}, E_1) & \longleftarrow &\cdots  &\longleftarrow &
(W_{r}, E_r)
\end{array}
\end{equation}
is constructed with centers $Y_i\subset X_i(\overline h)$ every time when $W_i \longleftarrow W_{i+1}$ is a transformation of type A), then the same sequence is a local sequence of the basic object, say 
\begin{equation}\label{seqboT}
\begin{array}{cccccccc}
 (J_0,b) & & (J_1,b) &  & &  &(J_r,b)\\
(W_{0}, E_0) & \overset{\pi_{Y_0}}{\longleftarrow} & (W_{1}, E_1) & \overset{\pi_{Y_1}}{\longleftarrow} &
\cdots  &\overset{\pi_{Y_{r-1}}}{\longleftarrow} &
(W_{r}, E_r)
\end{array}
\end{equation}
and $$\Sing(J_i,b)=X_i(\overline h)$$
for all index $0\leq i \leq r$. Moreover, any local sequence of the basic object (\ref{seqboT}) induces a local sequence (\ref{rlcwwWW}) with the previous conditions.
\end{enumerate}
\end{theorem}

 \begin{remark}\label{rkH4} 1) Fix $X\subset W$ and $(W,E)$,  $\overline h =\max HS_X$. Let $\B=(W,(J,b),E)$ be a basic object attached to the value $\overline h$ as in Theorem \ref{H4}. 
 
 Assume now that there is another basic object $\B'=(W,(K,d),E)$,  which is weakly equivalent to $\B=(W,(J,b),E)$. Then one can replace  $\B$ by $\B'=(W,(K,d),E)$ in Theorem \ref{H4}. In fact both $\B$ and $\B'$ are basic objects that provide the same closed sets. 
 \vskip 0.5cm
 
 2) {\em On the cover and the problem of patching}. 
 Note that Theorem \ref{H4} 
 does not claim that given $X\subset W$, $(W,E)$,  there is a basic object $\B=(W,(J,b),E)$ attached to the value $\overline h$ with the prescribed property. It says that this holds \emph{after restriction to an \'etale cover}. Let us insist on this point as it is the key for the definition of the functions in Constructive Resolution.

 Suppose that $U_{\lambda}$, $U_{\beta}$ are two charts of the cover of $W$, and that $X_{\lambda}\subset U_{\lambda}$, $(U_{\lambda},E_{\lambda})$ and $X_{\beta}\subset U_{\beta}$, $(U_{\beta},E_{\beta})$ are obtained by restriction. Hironaka asserts that there is a basic object $\B_\lambda=(U_{\lambda},(J_{\lambda},b_{\lambda}),E_{\lambda})$  attached to the value $\overline h=\max HS_X$ (and a basic object $\B_\beta=(U_{\beta},(J_{\beta},b_{\beta}),E_{\beta})$  attached to the same value $\overline h=\max HS_X$). Note in particular that a resolution of this first basic object defines a sequence of blow-ups of $X_{\lambda}\subset U_{\lambda}$, $(U_{\lambda},E_{\lambda})$ so that the final strict transform of $X_\lambda$ has a Hilbert-Samuel function which takes values strictly smaller than $\overline h$ at any point. As was indicated before, if we want to claim that there is a similar global statement for $X\subset W$, $(W,E)$
we have to make sure that the resolutions of the different basic objects, say  
$\B_\lambda$ 
and $\B_\beta$, somehow patch to provide a sequence of blow-ups along $W$. 
In this case the resolutions of these locally defined basic objects can be expressed as restrictions of a sequence of blow-ups along $(W,E)$.

Set formally $U_{\lambda, \beta}=U_{\lambda}\cap U_{\beta}$. 
Here $U_{\lambda, \beta}\longrightarrow U_{\lambda}$ is smooth and defines 
pull-backs both of $X_{\lambda}\subset U_{\lambda}$, $(U_{\lambda},E_{\lambda})$ and of  $\B_\lambda=(U_{\lambda},(J_{\lambda},b_{\lambda}), E_{\lambda})$.
We denote them by
$X_{\lambda, \beta}\subset U_{\lambda, \beta}$, $(U_{\lambda,\beta},E_{\lambda, \beta})$ and  $\B^\beta_\lambda=(U_{\lambda, \beta},(J^{\beta}_{\lambda},b_{\lambda}),E_{\lambda, \beta})$, respectively. 

Similarly,  $U_{\lambda, \beta}\longrightarrow U_{\beta}$ is smooth and defines 
pull-backs both of $X_{\beta}\subset U_{\beta}$, $(U_{\beta},E_{\beta})$ and of  $(J_{\beta},b_{\beta})$, $(U_{\beta},E_{\beta})$, say 
$X_{\lambda, \beta}\subset U_{\lambda, \beta}$, $(U_{\lambda,\beta},E_{\lambda, \beta})$ and  $\B_\beta^\lambda=(U_{\lambda, \beta},(J^{\lambda}_{\beta},b_{\beta}),E_{\lambda, \beta})$. 

Recall now Definition \ref{weq} in which two basic objects are defined to be weakly equivalent when they define the same closed sets, in a very precise way, involving all possible local sequences.

\vspace{0.3cm}

\noindent{\bf Main observations:} 
\begin{enumerate}
\item Each  $\B_\lambda=(U_{\lambda},(J_{\lambda},b_{\lambda}),E_{\lambda})$ 
is well defined only up to weak equivalence.

\item The basic objects $\B_\lambda^\beta$ and $\B_\beta^{\lambda}$ are weakly equivalent.
\end{enumerate}

These two main observations follow from the fact that the Theorem applies to  the same value $\overline h=\max HS_X$, both for $X_{\lambda}\subset U_{\lambda}$, $(U_{\lambda},E_{\lambda})$  and the basic object $\B_{\lambda}$, and also for  $X_{\beta}\subset U_{\beta}$, $(U_{\beta},E_{\beta})$  and the basic object $\B_{\beta}$.

\vspace{0.2cm}

\noindent{\bf Main Challenge:} Define, as in \ref{crobo}, a totally ordered set $T$ and a procedure of resolution of basic objects, by means of upper semi-continuous functions with values on $T$, so that two basic objects, 
say $\B_0=(W_0,(J_0,b),E_0)$ and $\B'_0=(W_0,(K_0,d),E_0)$, which are  weakly equivalent, undergo the same resolution (\ref{weq}).

The constructive resolution of basic objects will accomplish this requirement. Moreover, suppose that the constructive resolution of $\B_0$ is
\begin{equation}\label{riicc1}
\begin{array}{cccccccc}
 (J_0,b) & & (J_1,b) & & &  &(J_{r},b)\\
(W_{0}, E_0) & \overset{\pi_{Y_0}}{\longleftarrow} & (W_{1}, E_1) & \overset{\pi_{Y_1}}{\longleftarrow} &\cdots  &\overset{\pi_{Y_{r-1}}}{\longleftarrow} &
(W_{r}, E_r)
\end{array}
\end{equation}
defined recursively in terms of functions $ h_i: \Sing(J_i,b) \longrightarrow T$; and that
\begin{equation}
\label{riicc2}
\begin{array}{cccccccc}
 (K_0,d) & & (K_1,d) & & &  &(K_{r},d)\\
(W_{0}, E_0) & \overset{\pi_{Y'_0}}{\longleftarrow} & (W_{1}, E_1) & \overset{\pi_{Y'_1}}{\longleftarrow} &
\cdots  &\overset{\pi_{Y'_{r-1}}}{\longleftarrow} &
(W_{r}, E_r)
\end{array}
\end{equation}
is the resolution of $\B_0'$, defined in terms of functions, say $h'_i: \Sing(K_i,d) \longrightarrow T$. Then,
$$h_i=h'_i, \ \ 0\leq i \leq r$$
as functions on $\Sing(J_i,b)=\Sing(K_i,d)$, and in particular $Y_i=Y'_i$. This guarantees that two weakly equivalent basic objects will undergo the same constructive resolution.

The resolution of each basic object obtained by the constructive procedure will be defined so as to be compatible with weak equivalence. This will ensure the patching required, in order to come from Theorem \ref{H4} to that of Resolution of singularities {\em a la Hironaka}, namely by lowering, successively, the highest value of the Hilbert-Samuel function. 

This is what we had previously formulated here as the Optimistic Claim. 
 \end{remark}

\begin{parrafo} {\bf The elegance of Hironaka's philosophy: Functions compatible with weak equivalence.}

Theorem \ref{H4} indicates that basic objects are to be considered
up to weak equivalence. So Hironaka suggests us to view $\mathcal B=(W_0,(J_0,b),E_0)$) simply as a tool to define closed sets.
And our Main Challenge is to find a totally ordered set $T$, and a procedure of constructive resolution of basic objects, by means of functions with values on $T$, in such a way that two weakly equivalent basic objects are treated in the same manner. This means that the upper semi-continuous functions defining the resolution should be the same for two basic objects that are weakly equivalent.  

The functions defining a resolution are expected to take maximum value on a {\em smooth} subschemes. But let us first leave aside this aspect of smoothness at this point, and simply draw attention on the definition of functions on basic objects which are compatible with weak equivalence.

The strategy is simple: 
\begin{center}
\begin{tabular}{|p{5.8in}|}
\hline  Find a totally ordered set $T$ and  assign to 
 {\em any} $\B=(W_0,(J_0,b)E_0)$
 an upper semi-continuous function $$h_\B:\Sing(J_0,b) \longrightarrow T$$ in such a way 
that the value $h_\mathcal B(x)\in T$ ($x\in \Sing(J_0,b)$) can be expressed in terms of the closed sets defined by $\mathcal B$.\\
\hline
\end{tabular}
\end{center}

\vspace{0.2cm}

 Recall here that 
by ``closed sets defined by a basic object'' Hironaka does not mean the closed set 
$\Sing(J_0,b)$, he means the closed sets $\Sing(J_i,b)$ in $W_i$, {\em for all possible local sequences of $\mathcal B$} (\ref{rlcww}).

No matter how abstract this approach might seem at first sight, what is  clear is that if $\B=(W_0,(J_0,b),E_0)$, and $\B'=(W_0,(K_0,d),E_0)$ are weakly equivalent the two functions  $$h_\mathcal B:\Sing(J_0,b) \longrightarrow T \mbox{ and  } h_\mathcal {B'}:\Sing(K_0,d) \longrightarrow T$$ are the same (recall that $\Sing(J_0,b)=\Sing(K_0,d)$).

The following is the main example of a totally ordered set $T$ and of functions $h_\B$ which fulfill the previous condition:

\begin{theorem}\label{lfdhi} Take $T=\mathbb Q$ as the totally ordered set. Given $\B=(W_0,(J_0,b),E_0)$, and a point 
$x\in \Sing(J_0,b)$, the rational number $\frac{\nu_x(J_0)}{b}$ can be expressed by the closed sets defined by $\mathcal B$.
Here $\nu_x(J_0)$ denotes the order of $J_0$ at the regular local ring 
$\calo_{W_0,x}$.
\end{theorem}
 As $\nu_x(J_0)$ is an upper semi-continuous function on $x$, the function 
\begin{equation}\label{ord}
\ord_\mathcal B: \Sing(J_0,b)\longrightarrow \mathbb Q ; \ \ \ord_\mathcal B(x)=\frac{\nu_x(J_0)}{b}
\end{equation}
is also upper semi-continuous. 

The proof will be addressed in \ref{trdh} (see also \ref{slbre}).  The discussion in the following example already gives a nice motivation.

Take $W_0={\mathbb A}_k^1=\Spec(k[X])$ and set 
$\B=(W_0,(J_0,b),E_0=\emptyset)$, $J_0=\langle X^{a}\rangle$,  for some $a\in\mathbb{Z}_{>0}$ such that $a\geq b$ (otherwise $\Sing(J_0,b)=\emptyset$).
Now $\Sing(J_0,b)=\mathbb O$ is the origin and $\ord_\mathcal B(\mathbb O)=\frac{a}{b}$.

If we blow up $\mathbb O$ what we get is the identity map, but if we follow the law of transformation of basic objects we obtain $(J_1,b)$ where $J_1=\id{X^{a-b}}$. If $a-b\geq b$ then $\Sing(J_1,b)=\mathbb O$ and if we blow up this point we obtain 
$(J_2,b)$ where $J_2=\id{X^{a-2b}}$. We can blow up r-times at 
$\mathbb O$ if and only if $ a-(r-1)b \geq b$. In particular the biggest integer $r$ for which one can blow up $r$-times, is the biggest positive 
integer $r$ so that $$ a-(r-1)b \geq b.$$ Note that such integer, say $r_0$, is the integral part of the fraction $\frac{a}{b}$, say $r_0=\lfloor \frac{a}{b}\rfloor.$

So $r_0$ can be expressed as the largest integer $r$ for which there is a sequence of local transformations of $\mathcal B$ consisting of $r$ successive blow ups at $\mathbb O$. There is a local sequence of length $r_0$:
\begin{equation}\label{}
\begin{array}{cccccccc}
 (J_0,b) & & (J_1,b) & &  & &  &(J_{r_0},b)\\
({\mathbb A}^1, E_0) & \longleftarrow & ({\mathbb A}^1, E_1) & \longleftarrow & &\cdots  &\longleftarrow &
({\mathbb A}^1, E_{r_0})
\end{array}
\end{equation}
obtained by blowing up at the origin at each step, and $\Sing(J_{r_0},b)=\emptyset$. So there is no local sequence of length $r_0+1$ obtained by blowing up at the origin.

In other words $r_0=\lfloor\frac{a}{b}\rfloor$ is information encoded by the closed set defined by this particular sequence of transformations of  $\mathcal B$: 
$$\Sing(J_0,b)=\mathbb{O},\ \Sing(J_1,b)=\mathbb{O},\ \dots, \ \Sing(J_{{r_0}-1},b)=\mathbb{O}, \ \Sing(J_{r_0},b)=\emptyset.$$
Of course the integral part of the fraction is only an approximation, but there are many other local sequences of transformations of $\mathcal B$ as local sequences also allow us to take pull-backs by smooth morphisms. What Hironaka shows, and it is well illustrated in the proof in \ref{trdh}, is that using this larger class of local sequences, one can find out exactly the value $\frac{a}{b}$. 

\begin{corollary}\label{rfj}
If $\B=(W_0,(J_0,b),E_0)$ and $\B'=(W_0,(K_0,d),E_0)$ are weakly equivalent, then $$ \frac{\nu_x(J_0)}{b}=\frac{\nu_x(K_0)}{d}$$
at any $x\in \Sing(J_0,b)=\Sing(K_0,d)$.
\end{corollary}
\begin{proof}This occurs because the weakly equivalent basic objects $\B$ and $\mathcal B'$ define the same closed sets.
\end{proof}
\begin{remark}\label{slbre}

 The argument that Hironaka uses in his proof of Theorem \ref{lfdhi}, to be developed in  \ref{trdh},
 can be expressed roughly as follows: Fix a basic object $\B_0=(W_0,(J_0,b),E_0)$, and a point $x_0\in \Sing(J_0,b)$. Consider a local sequence, say 
\begin{equation}
\begin{array}{cccccccc}
 (J_0,b) & & (J_1,b) & & (J_2,b) & &  &(J_r,b)\\
(W_{0}, E_0) & \longleftarrow & (W_{1}, E_1) & \longleftarrow &
(W_{2}, E_2) &\cdots  &\longleftarrow &
(W_{r}, E_r)\end{array}
\end{equation}
together with points $x_i\in \Sing(J_i,b) (\subset W_i)$, $0\leq i \leq r$, so that each
$x_i$ maps to $x_{i-1}$ (in particular all $x_i$ map to $x_0$).
He proves that the rational number 
$\ord_{\mathcal B_0}(x_0)=\frac{\nu_{x_0}(J_0)}{b}$ can be specified
once you know the local codimension of $\Sing(J_i,b) (\subset W_i)$
at $x_i$, {\em for all} local sequences, and {\em for all}  choices of $x_i$ as above. The argument makes essential use of maps of type B) in the definition of a local sequence in
\ref{lsbo} (namely, of pull-backs of basic objects by smooth morphisms).





\end{remark}

\begin{parrafo}{\bf A common frame in the previous discussions.}

Once we fix $\B=(W_0,(J_0,b),E_0)$, the  local sequences of transformations of $\B$ define closed sets. The following definitions will apply naturally to this situation.

\begin{definition}

Fix $(W_0,E_0)$ as in \ref{p1} and define a {\em sequence over} $(W_0,E_0)$ as:
\begin{equation}\label{nrs}
\xymatrix@C=2pc{
(W_0,E_0) & (W_1,E_1)\ar[l] & \ \ \dots\ \ \ar[l] &  (W_{r},E_{r})\ar[l]
}
\end{equation} 
for some integer $r$, where each 
$(W_i,E_i)\longleftarrow (W_{i+1},E_{i+1}) $ is either:
\begin{itemize}
\item[A)] A blow-up along a smooth center $Y_i$ having normal crossings with $E_i$, in which case $E_{i+1}$ is as in \ref{p1}.

\item[B)] A pull-back by a smooth morphism $W_i \longleftarrow W_{i+1}$ in which case $E_{i+1}$ is defined as in \ref{pt5}.
\end{itemize}
\end{definition}

Many  sequences can be constructed over $(W_0,E_0)$. We now introduce a notion of an {\em assignment of closed set} over $(W_0,E_0)$. The idea is to assign closed sets $F_i$ in $W_i$, $0\leq i \leq r$, to a sequence (\ref{nrs}); however such assignment of closed sets will be defined only to  {\em some} of these sequences.

For example, if we take $\B=(W_0,(J_0,b),E_0)$, then we will assign 
$F_0=\Sing(J_0,b)$ to $W_0$. If $(W_0,E_0)\longleftarrow (W_1,E_1)$ is obtained by a blow-up along $Y$, then a closed set will
be assigned to $W_1$, namely $F_1=\Sing(J_1,b)$, {\em only} if $Y\subset F_0$ (see \ref{ejdebo}).

Therefore the definition of an assignment has to indicate which are the sequences for which closed sets will be assigned. We do this by induction on the integer $r$.
%

\begin{definition}\label{afcs}
Define an {\em assignment of closed sets over
$(W_0,E_0)$}, say $(\mathcal F, (W_0,E_0))$, to be given by:

(1) A (unique) closed set $F_0\subset W_0$.

(2) Fix a sequence (\ref{nrs}), and assume that, for the sequence defined by the first $r-1$ steps, an assignment of closed sets, say 
\begin{equation}\label{}
\begin{array}{cccccccc}
 F_0 & & F_1 & &  & & F_{r-1} & \\
(W_0,E_0)&\longleftarrow & (W_1,E_1)& \longleftarrow &\dots &\longleftarrow &(W_{r-1},E_{r-1}) & \\
\end{array}
\end{equation}
is defined. Here $F_0$ is as above.
We now give conditions  in order to decide when closed sets are assigned to (\ref{nrs}). In such a case we will denote them by 
\begin{equation}
\begin{array}{ccccccccc}
 F_0 & & F_1 & &  & & F_{r-1} & &F_{r} \\
(W_0,E_0)&\longleftarrow & (W_1,E_1)& \longleftarrow &\dots &\longleftarrow &(W_{r-1},E_{r-1}) & \longleftarrow &(W_{r},E_{r})
\end{array}
\end{equation}
(same $F_0, \dots ,F_{r-1}$ as above):

2A) If $(W_{r-1},E_{r-1})\overset{\pi_{r-1}}{\longleftarrow} (W_{r},E_{r}) $ is a blow-up along a smooth center, we require that the center $Y_{r-1}$ be included in 
$ F_{r-1}$. If so, a unique closed $F_{r}$ is assigned in $W_r$,  with the property that
$$F_{r}\setminus H_{r}=F_{r-1}\setminus Y_{r-1},$$
where $H_r=\pi_{r-1}^{-1}(Y_{r-1})$ denotes the exceptional hypersurface of the blow-up.

Sometimes we will impose some extra condition on the choice of the smooth centers; but these conditions will arise quite naturally.

2B) If $(W_{r-1},E_{r-1})\longleftarrow (W_{r},E_{r}) $ is a pull-back by a smooth morphism, say $\sigma:W_{r} \longrightarrow W_{r-1} $, then a closed set 
$F_r$ is assigned to $W_r$, and moreover $$F_{r}=\sigma^{-1}(F_{r-1}).$$

\end{definition}

We would like to emphasize that, as it was previously indicated, an assignment of closed sets over $(W_0,E_0)$ defines closed sets for a sequence (\ref{nrs}) only when this sequence fulfills the specific extra condition indicated as above. In each case we dictate the specification telling which sequences are the ones for which there is an assignment. Note that the specific and extra condition is imposed only for transformations of type 2A).


\begin{example}\label{ejdebo}
A basic object  $\B=(W_0,(J_0,b), E_0)$ defines an assignment of closed sets on $(W_0,E_0)$, say
$$ (\mathcal F_{\mathcal B}, (W_0,E_0)),$$
in the following way:

1) $F_0=\Sing(J_0,b)$.

2) For any local sequence  (\ref{rlcww}), the assignment of closed sets is
\begin{equation}\label{ejyy}
\begin{array}{cccccccc}
 \Sing(J_0,b) & &\Sing (J_1,b) & & &  &\Sing(J_r,b)\\
(W_{0}, E_0) & \longleftarrow & (W_{1}, E_1) & \longleftarrow &
\cdots  &\longleftarrow &
(W_{r}, E_r)
\end{array}
\end{equation}
\end{example}
\begin{remark}\label{ej123}

1) In Example \ref{ejdebo} (of the assignment of closed sets defined by $\B=(W_0,(J_0,b), E_0)$), the only sequences over $(W_0,E_0)$ for which closed sets are assigned are the local sequences of the basic objects in Definition \ref{lsbo}. There are many sequences (\ref{nrs}) over $(W_0,E_0)$, but closed sets are assigned only to those in Definition \ref{lsbo}.

Recall that a local sequence of the basic object
$\B=(W_0,(J_0,b), E_0)$, say (\ref{ejyy}), 
was defined with the only condition that for any index $i$ for which 
 \begin{equation}
\begin{array}{ccc}
 \Sing(J_i,b)  & & \Sing(J_{i+1},b) \\
(W_i , E_i) & \longleftarrow & (W_{i+1} ,E_{i+1})
\end{array}
\end{equation}
is given by a blow-up, its center $Y_i$ should be included in $\Sing(J_i,b) $ (and have normal crossings with 
$E_i$). It is clear that properties 2A) and 2B) in \ref{afcs} hold for $ (\mathcal F_{\mathcal B}, (W_0,E_0)).$

2) {\bf  Important}: Two basic objects $\mathcal B$ and $\B'=(W_0,(K_0,d), E_0)$ are weakly equivalent if and only if they define the same assignment of closed sets.


\end{remark}

The main example of assignment of closed sets is the one defined by a basic object as above. The following parallels the definition in (\ref{riic}).

\begin{definition}{\rm ({\bf Resolution of an assignment of closed sets})}. \label{racs} Let $(\mathcal{F},(W,E))$ be an assignment of closed sets.  A sequence
\begin{equation}
\begin{array}{ccccccc}
 F_0 & & F_1 & &  &  &F_{r} \\
(W_0,E_0)&\overset{\pi_0}{\longleftarrow} & (W_1,E_1)& \overset{\pi_1}{\longleftarrow} &\dots &\overset{\pi_{r-1}}{\longleftarrow} &(W_{r},E_{r})
\end{array}
\end{equation}
(with closed sets $F_i$ assigned to $W_i$, $0\leq i \leq r$) is a \emph{resolution} of  $(\mathcal{F},(W,E))$, if $F_r=\emptyset$ and 
each $\pi_i$ is a blow-up ( see 2A) in Definition \ref{afcs}).
\end{definition}

\begin{example}\label{agte} Fix a reduced subscheme $X_0\subset W_0$ and $(W_0,E_0)$, then an assignment of closed sets is defined in \ref{fed14} by taking:

1) $F_0=X_0$ in $W_0$.

2) for any local sequence, set $F_i=X_i$, namely
\begin{equation}\label{1271}
\begin{array}{cccccccc}
 X_0 & & X_1 & &  &  &X_r\\
(W_{0}, E_0) & \longleftarrow & (W_{1}, E_1) & \longleftarrow &\cdots  &\longleftarrow &
(W_{r}, E_r)
\end{array}
\end{equation} 

So closed sets are assigned in this case, {\em only} for sequences over $(W_0,E_0)$ that arise from a local sequence as those defined in \ref{fed14}.
\end{example}

\begin{parrafo}\label{aoech} { \bf Assignments of closed sets and functions; a useful Lemma.}

The reader might want to look first into Example \ref{34} to motivate the following definition.

\begin{definition}\label{def229} Define an \emph{assignment of closed sets and functions} over $(W_0,E_0)$, say $(\mathcal EF, (W_0,E_0),T)$, as an assignment of closed sets over $(W_0,E_0)$, say $(\mathcal F, (W_0,E_0))$, together with upper semi-continuous functions on a totally ordered set $T$. Namely, 

\begin{enumerate}
\item A (unique) closed set $F_0$  in $ W_0$, and a unique upper semi-continuous function $g_0: F_0 \rightarrow T$.

\item  Given a sequence over $(W_0,E_0)$, say 
\begin{equation}\label{n12rs}
(W_0,E_0)\longleftarrow (W_1,E_1) \longleftarrow\ \ \dots\ \ \longleftarrow (W_r,E_r)
\end{equation} 
and assuming that closed set are assign to it by  $(\mathcal F, (W_0,E_0))$, say 
\begin{equation}\label{ruij46}
\begin{array}{cccccccc}
 F_0 & & F_1 & &  & & F_r & \\
(W_0,E_0)&\longleftarrow & (W_1,E_1)& \longleftarrow &\dots &\longleftarrow &(W_r,E_r) & \\
\end{array}
\end{equation}
then (unique) upper semi-continuous functions, say $$g_i:F_i \longrightarrow T \ \ 0\leq i \leq r,$$
are defined, and they have the following properties:
\begin{enumerate}
\item[A)] If $W_i \longleftarrow W_{i+1}$ is monodal trasformation with center $Y_i$, then 
 $$g_{i+1}(x)=g_i(x)$$
for any  $x\in F_{i+1}\setminus H_{i+1}=F_i\setminus Y_i$.

\item[B)]  If $W_i  \overset{\sigma}{\longleftarrow}  W_{i+1}$ is a smooth morphism then
$$g_{i+1}(x)=g_i(\sigma(x))$$ for any $x\in F_{i+1}=\sigma^{-1}(F_i)$.
\end{enumerate}
\end{enumerate}
\end{definition}
\end{parrafo}
%


\begin{example}\label{34} ({\it Hironaka's assignment}). Take $T=\mathbb Q$. For any basic object $\B_0=(W_0,(J_0,b),E_0)$, an assignment of closed sets and functions, say $(\mathcal HF, (W_0,E_0),T=\mathbb Q)$,  is defined by setting:
\begin{enumerate}
\item $F_0=\Sing(J_0,b)$ and $g_0=\ord_{\B_0}: \Sing(J_0,b) \longrightarrow T$ (see (\ref{ord})).

\item Assign, for any sequence (\ref{ejyy}), $F_i= \Sing(J_i,b)$ and  $g_i=\ord_{\B_i}: \Sing(J_i,b) \longrightarrow T$.
\end{enumerate}
\end{example}

Here the role of $(\mathcal F, (W_0,E_0))$ in the previous Definition \ref{def229} is played by $(\mathcal F_{\mathcal B_0}, (W_0,E_0))$ in \ref{ejdebo}. Corollary \ref{rfj} ensures that if 
$\B_0=(W_0,(J_0,b),E_0)$ and $\B_0'=(W_0,(K_0,d),E_0)$ are weakly equivalent, then they define the same assignment of closed sets $(\mathcal F_{\mathcal{B}_0}, (W_0,E_0))$, and also the same assignment of closed sets and functions.

\begin{example}\label{35} 

 Fix $T= \mathbb N^{\mathbb N}$ with the lexicographic order. Given a reduced subscheme $X_0\subset W_0$ and $(W_0,E_0)$, consider the assignment of closed sets in \ref{agte} together with the Hilbert Samuel functions:
 \begin{enumerate}
\item Assign to $W_0$ the closed set $F_0=X_0$, and the function $HS_{X_0}:X_0 \longrightarrow T$, the Hilbert Samuel function of $X_0$.
 
\item For any sequence (\ref{1271}), set $F_i=X_i$ and $HS_{X_i}:X_i \longrightarrow T$, the Hilbert Samuel function of $X_i$, $0\leq i \leq r$.
\end{enumerate}
\end{example}

Recall Bennett's inequality in (\ref{jurdw}) as a first motivation for the next definition:

\begin{definition}\label{ddusc}
An assignment of closed sets and functions, say $(\mathcal EF, (W_0,E_0),T)$, will be said to be {\em non-increasing} if the following property holds:

 Whenever sets and functions are assigned to a sequence (\ref{n12rs}), say 
\begin{equation}\label{ruij461}
\begin{array}{cccccccc}
 F_0 & & F_1 & &  & & F_r \\
(W_0,E_0)&\longleftarrow & (W_1,E_1)& \longleftarrow &\dots &\longleftarrow &(W_r,E_r)
\end{array}
\end{equation}
and functions
$$g_0:F_0 \longrightarrow T, \quad g_1:F_1 \longrightarrow T, \dots\ \  g_r:F_r \longrightarrow T,$$
and if, in addition, for each index $i$ for which $W_i \overset{\pi_{i}}{\longleftarrow}  W_{i+1}$ is a monodal trasformation the center is included in $\Max g_i$ (i.e., $Y_i \subset \Max g_i$), then
\begin{equation}\label{eddes}
g_i(\pi_i(x)) \geq g_{i+1}(x)
\end{equation}
for any $x\in F_{i+1}$.
\end{definition}

\begin{remark}\label{kkr}
Note that in the setting of the previous definition:
$$ \max g_0\geq  \max g_1 \dots \geq  \max g_r. $$

Check that the assignment of closed sets and functions 
in Example \ref{35} is non-increasing 
(see (\ref{jurdw})), whereas the assignment in Example \ref{34}  does not have this property.
\end{remark}

Non-increasing assignments will be very useful for the further development, as we indicate below. The reader might want to look  into the Example \ref{ejHSa}.

\begin{lemma} {\rm ({\bf Handy Lemma})}\label{Handy}
 Suppose that the assignment of closed sets and functions $(\mathcal EF, (W_0,E_0),T)$, is {\em non-increasing}.
 Let $\max g_0$ be the highest value achieved by $g_0: F_0 \longrightarrow T$.

Then a new assignment of closed set, say $(\mathcal F(\max g_0), (W_0,E_0))$, is defined by setting 
\begin{enumerate}
\item $F_0(\max g_0)=\Max g_0$ in $ W_0$ (the closed subset of points of $F_0$ where the function $g_0$ takes the maximum value $\max g_0$).

\item  Firstly, assume that a sequence (\ref{n12rs}) is such that sets and functions are assigned by  $(\mathcal EF, (W_0,E_0),T)$, say
\begin{equation}
\begin{array}{cccccccc}
 F_0 & & F_1 & &  & & F_r & \\
(W_0,E_0)&\longleftarrow & (W_1,E_1)& \longleftarrow &\dots &\longleftarrow &(W_r,E_r) & \\
\end{array}
\end{equation}
and  $$g_i:F_i \longrightarrow T, \ \ 0\leq i \leq r.$$
Secondly assume that $Y_i\subset F_i(\max g_0)$
for any index $i$ for which $(W_i,E_i)\longleftarrow (W_{i+1},E_{i+1}) $ is a blow-up with center $Y_i$, where $F_i(\max g_0) $ denotes the subset of $F_i$ at which $g_{i}$ takes the value $\max g_0$.

When these conditions hold, attach to this sequence the closed sets
\begin{equation}
\begin{array}{cccccccc}
 F_0(\max g_0) & & F_1(\max g_0) & &  & & F_r(\max g_0) & \\
(W_0,E_0)&\longleftarrow & (W_1,E_1)& \longleftarrow &\dots &\longleftarrow &(W_r,E_r) & \\
\end{array}
\end{equation}
\end{enumerate}
\end{lemma}

\end{parrafo}

\begin{parrafo}
One can check from \ref{kkr} that either $F_r(\max g_0)=\emptyset$ or 
$$ \max g_0= \max g_1 \dots =  \max g_r $$
and $F_i(\max g_0)= \Max g_i$ for $ 0\leq i \leq r$.
A first example in which the Handy Lemma applies is on the assignment of closed sets and functions in \ref{35}:

\end{parrafo}

\begin{example}\label{ejHSa} ({\bf  The Hilbert Samuel assignment of closed sets}).

Fix $X_0\subset W_0$, $(W_0,E_0)$, and set $\overline h =\max HS_{X_0}$. We now define an assignment of closed sets over $(W_0,E_0)$
corresponding to the value $\overline h =\max HS_{X_0}$, say
$$ (\mathcal F({\overline h}), (W_0,E_0)).$$

Recall the definition of  a local sequence (\ref{rlcwwW}): 
\begin{equation}\label{rlcbw}
\begin{array}{cccccccc}
 X_0 & & X_1 & & &  &X_r\\
(W_{0}, E_0) & \longleftarrow & (W_{1}, E_1) & \longleftarrow&\cdots  &\longleftarrow &
(W_{r}, E_r)\end{array}
\end{equation}
where each  
 \begin{equation}\label{psoi}
\begin{array}{ccc}
 X_i & & X_{i+1}\\
(W_i , E_i) & \longleftarrow & (W_{i+1} ,E_{i+1})\\ 
\end{array}
\end{equation}
is obtained either by a suitable blow-up with center $Y_i \subset X_i$, or by a smooth morphism $\sigma_i:W_{i+1} \longrightarrow W_i$. Recall also the assignment of closed sets and functions in Example \ref{35}, obtained from the functions $HS_{X_i}:X_i\longrightarrow \mathbb{N}^{\mathbb{N}}$.

Remark \ref{rk4} shows that a new assignment of closed sets is defined by setting 
\begin{enumerate}
\item $F_0=X(\overline h)$

\item For any local sequence (\ref{rlcbw}),  with the extra condition that 
$Y_i\subset X_i(\overline h)$ if (\ref{psoi}) is obtained by a blow-up, set
$$F_i=X_i(\overline h) \ \ 0\leq i \leq r.$$
\end{enumerate}
\end{example}

\end{parrafo}

\section{On generalized basic objects}
\begin{parrafo} {\bf General basic objects of dimension $d$}.

We want to identify the basic objects that are weakly equivalent. The notion of assignment of closed sets was appropriate since two basic objects are weakly equivalent if and only if they define the same assignment (see Remark \ref{ej123}).

Moreover, Theorem \ref{lfdhi} of Hironaka says that a basic object defines an assignment of closed sets and functions, with the functions  order introduced in (\ref{ord}).

The objective of this section is to introduce new examples of assignments. They will all be relevant for our further discussion:
\begin{enumerate}
\item In Definition \ref{nonemb} new assignments will be introduced by patching assignments obtained from basic objects.

\item In \ref{paso1} and \ref{ods} we introduce an assignment of closed sets in terms of a basic object together with a closed immersion (Embedded Basic Objects).

\item In Definition \ref{ebgbo} a class of assignments of closed sets is defined by patching assignments as in 2).
\end{enumerate}

\vskip 0.4cm
Fix notation as in Definition \ref{afcs} for an assignment $(\mathcal F, (W_0,E_0))$. Assume now that a short  sequence
\begin{equation}\label{eq312}
(W_0,E_0) \longleftarrow (W_1,E_1)
\end{equation}
(obtained either by a smooth morphism or by a blow-up) is such that there are closed sets assigned to it, say
\begin{equation}\label{r666}
\begin{array}{ccc}
 F_0 & & F_1  \\
(W_0,E_0)&\longleftarrow & (W_1,E_1)
\end{array}
\end{equation}

Consider now all sequences over $(W_0,E_0)$ of the form 
\begin{equation}
(W_0,E_0)\longleftarrow (W_1,E_1) \longleftarrow\ \ \dots\ \ \longleftarrow (W_r,E_r)
\end{equation} 
namely all those sequences which begin with the short sequence (\ref{eq312}), for which 
 closed sets, say 
\begin{equation}\label{ru4666}
\begin{array}{cccccccc}
 F_0 & & F_1 & &  & & F_r & \\
(W_0,E_0)&\longleftarrow & (W_1,E_1)& \longleftarrow &\dots &\longleftarrow &(W_r,E_r) & \\
\end{array}
\end{equation}
are assigned.

One can check that a new assignment of closed sets is defined now on $(W_1, E_1)$, say $(\mathcal F_1, (W_1,E_1))$, by setting $F_1$ in $W_1$, and taking, in general, all sequences (\ref{ru4666}) where we neglect the first step.
\begin{definition}\label{def32} Fix the notation as above, define $(\mathcal F_1, (W_1,E_1))$ as the \emph{transform of $(\mathcal F, (W_0,E_0))$}, and denote it by 
\begin{equation}\label{r5436}
\begin{array}{ccc}
 \mathcal F & & \mathcal F_1 \\
(W_0,E_0)&\longleftarrow & (W_1,E_1) \\
\end{array}
\end{equation}

\end{definition}
\begin{example} Let the notation be as in Example \ref{agte}, where the assignment of closed sets defined by $X_0\subset W_0$, $(W_0,E_0)$ is introduced. Set $r=1$ in (\ref{1271}), say 
\begin{equation} 
\begin{array}{ccc}
 X_0 & & X_1 \\
(W_0,E_0)&\longleftarrow & (W_1,E_1). \\
\end{array}
\end{equation}
The transform $(\mathcal F_1, (W_1, E_1)) $, in (\ref{r5436}), is the assignment defined by $X_1\subset W_1$, $(W_1,E_1)$. 
\end{example}

\begin{example} A similar observation applies to the assignment in Example \ref{ejdebo}, defined by $\B_0=(W_0,(J_0,b),E_0)$. If we set $r=1$ in (\ref{ejyy}), say 
\begin{equation}\label{ejyy2}
\begin{array}{ccc}
 \Sing(J_0,b) & &\Sing (J_1,b) \\
(W_{0}, E_0) & \longleftarrow & (W_{1}, E_1),
\end{array}
\end{equation}
the assignment $(\mathcal F_1, (W_1, E_1)) $ in (\ref{r5436}) is that defined by $\B_1=(W_1,(J_1,b),E_1)$.

\end{example} 
\begin{parrafo} A transformation of an assignment 
$(\mathcal F , (W_0,E_0))$ arises naturally if we take $W_0 \longleftarrow W_1$ to be an open immersion in (\ref{r5436}). In this case, set 
$W_1=U$ open in $W_0$, and set (\ref{r5436}) as
\begin{equation}\label{rcab}
\begin{array}{ccc}
 \mathcal F & & (\mathcal F)_U \\
(W_0,E_0)&\longleftarrow & (U,E_U)
\end{array}
\end{equation}
This is a natural notion of restriction, which we also 
apply when $W_0\leftarrow U$ is \'etale.
\end{parrafo}

A new class of assignments of closed sets will be introduced, called general basic objects. There will be two such notions: embedded and non-embedded, we begin by the latter.

\begin{definition}\label{nonemb} Fix $(W_0,E_0)$, where $W_0$ is smooth of dimension $d$. An assignment of closed sets
$(\mathcal F , (W_0,E_0))$ is said to be a {\em non-embedded general basic object} of dimension $d$ if there is an \'etale cover of $W_0$ by charts, say $\{U_{\lambda}\}$, so that each restriction  
$$((\mathcal F)_{U_\lambda}, (U_\lambda, E_\lambda))$$
is the assignment defined by a basic object $\B_\lambda=(U_\lambda, (J_\lambda, b_\lambda), E_\lambda)$.
\end{definition}

\begin{remark}\label{rk72}
1) If $(\mathcal F_\B , (W_0,E_0))$ is the assignment defined by a basic object $\B=(W_0,(J_0,E_0), E_0))$,
as in (\ref{ejdebo}), it is clearly a general basic object.

2) If $\B_\lambda=(U_\lambda, (J_\lambda, b_\lambda), E_\lambda)$ and $\B'_\lambda=(U_\lambda, (K_\lambda, d_\lambda), E_\lambda)$ are weakly equivalent (see Definition \ref{weq}), then $\B_\lambda$  can be replaced in the previous definition by $\B'_\lambda$ as both define the same assignment.

3) Take two charts of the cover, say $U_\lambda$ and $U_\beta$.
The restrictions of $\B_\lambda=(U_\lambda, (J_\lambda, b_\lambda), E_\lambda)$ and of $\B_\beta=(U_\beta, (J_\beta, b_\beta), E_\beta)$  to $U_\lambda \cap U_\beta$ are weakly equivalent as both define the same assignment.

4) If 
\begin{equation}
\begin{array}{ccc}
 \mathcal F & & \mathcal F_1 \\
(W_0,E_0)&\longleftarrow & (W_1,E_1)
\end{array}
\end{equation}
is a transformation as in (\ref{r5436}), then $(\mathcal F_1, (W_1,E_1))$ is also a general basic object. To check this, take the natural lifting of the cover $\{U_\lambda\}$ on $W_0$, say
\begin{equation}
\begin{array}{ccc}
 (J_\lambda, b_\lambda) & & \mathcal ((J_\lambda)_1, b_\lambda) \\
(U_\lambda ,E_\lambda)&\longleftarrow & ((U_\lambda)_1 ,(E_\lambda)_1)\\
\end{array}
\end{equation}

5) If $F_0\subset W_0$ is the closed set assigned by $(\mathcal F , (W_0,E_0))$, then there is a function 
$$ \ord: F_0 \longrightarrow \mathbb Q$$
obtained by patching the functions $ \ord^\lambda: \Sing(J_\lambda, b_\lambda) \longrightarrow  \mathbb Q$ (see 3) and Corollary \ref{rfj}). 
\end{remark}

\begin{corollary}\label{cor28}
A non-embedded general basic object, say $(\mathcal F, (W_0,E_0))$, defines an assignment  of closed sets and functions  (see Definition \ref{aoech}), say $(\mathcal HF, (W_0,E_0),\mathbb Q)$, by choosing Hironaka's functions order: 
\begin{enumerate}
\item Set $g_0= \ord: F_0 \longrightarrow \mathbb Q$ as above.

\item Consider a sequence over $(W_0,E_0)$, say 
$(W_0,E_0)\longleftarrow (W_1,E_1) \longleftarrow \dots \longleftarrow (W_r,E_r)$
for which closed sets are assigned, say
\begin{equation}\label{ruij462}
\begin{array}{cccccccc}
 F_0 & & F_1 & &  & & F_r \\
(W_0,E_0)&\longleftarrow & (W_1,E_1)& \longleftarrow &\dots &\longleftarrow &(W_r,E_r). \\
\end{array}
\end{equation}
Set 
$ g_i=\ord_i: F_i \longrightarrow \mathbb Q 
 \mbox{ for all } 0\leq i\leq r$.
\end{enumerate}
\end{corollary}

\end{parrafo}

\vspace{0.3cm}

Let us reformulate Hironaka's Theorem \ref{H4} in terms of assignments. More precisely, in terms of general basic objects as defined in \ref{nonemb}.
 
 \begin{theorem}{\rm ({\bf Hironaka})}.\label{H5} Fix $X_0\subset W_0$ and $(W_0,E_0)$. Set $\overline h =\max HS_{X_0}$. The assignment of closed sets $ (\mathcal F({\overline h}), (W_0,E_0))$   defined by $X_0$ and $\overline h $ in Example \ref{ejHSa}, is a non-embedded general basic object of dimension $d=\dim W_0$.

\end{theorem}

 \begin{parrafo}

We generalize now the concept of a basic object, by taking a basic object together with a closed immersion. This will define an assignments of closed sets and, in order to ease the notation, we shall call it an embedded basic object. 

So here an embedded basic object will be an assignment of closed sets defined by a basic object and a closed immersion. This notion will be essential for our forthcoming discussion. There will be two distinguished cases: tame, and non-tame, according to the conditions on  the immersion.

\end{parrafo}

\begin{parrafo}\label{paso1}{\bf Embedded basic objects: the Tame Case}.
Fix $(N_0,E''_0)$,  where $N_0$ is smooth and $E''_0=\{H_1, \dots ,H_s\}$ is a set
of hypersurfaces with only normal crossings. Let ${W}_0$ be a closed smooth subscheme in $N_0$ and assume that each $H_i\in E''_0$ intersects ${W}_0$ transversally, defining a smooth hypersurface $\overline{H}_i$. Moreover, assume that ${ E}_0=\{\overline{H}_{1}, \dots ,\overline{H}_s\}$ have normal crossings in $W_0$. In what follows this strong transversality condition will be indicated by:
\begin{equation}\label{31110}
{E}_0=E''_0\pitchfork W_0.
\end{equation}

Fix a non-zero ideal $J_0$ in $\calo_{{W}_0}$, a positive integer $b$ and the basic object $\B_0=(W_0, (J_0,b), E_0)$. Set $F_0=\Sing(J_0,b)$, which is closed in 
${ W}_0$ and hence in $N_0$.
$\B_0$ defines an assignment of closed sets over
$({ W}_0, { E}_0)$ (see Example \ref{ejdebo}). We claim that it also defines an assignment of closed sets over $(N_0, E''_0)$, say 
$$(\mathcal{F}',(N_0,E''_0)),$$
that we call an {\em embedded basic object in the tame case}:
If $Y_0\subset F_0=\Sing(J_0,b) \subset W_0$ is permissible for $\B_0$, then the same $Y_0$ defines $(N_0, E''_0)\longleftarrow (N_{1},E''_{1}) $ (due to the condition 
${E}_0=E''_0\pitchfork W_0$, which ensures that $Y_i$ has normal crossing with $E_0''$).

This defines a transform of $\B_0$, say $\B_1=({ W}_1, (J_1,b), { E}_1)$, a closed immersion $W_1\subset N_1$, and again ${E}_1=E''_1\pitchfork W_1$.

On the other hand, a smooth morphism  $(N_0, E''_0)\longleftarrow (N_{1},E''_{1}) $ induces:
 \begin{equation}
\begin{array}{ccc}
 (J_0,b) & & (J_{1},b) \\
({ W}_0, { E}_0) & \longleftarrow & ({ W}_{1} ,{ E}_{1})
\end{array}
\end{equation}
also smooth, where ${ W}_1$ is the pull-back of ${ W}_0$ in $N_1$. Moreover, 
${E}_1=E''_1\pitchfork W_1$ as before.

\begin{parrafo}{\bf Embedded basic objects: the non-tame  case}.\label{ods} 

We generalize now the setting in \ref{paso1}. Consider
\begin{itemize}
\item[$\bullet$] $(N_0,E'_0)$ with $E_0'=(E'_0)^+\cup(E'_0)^-$ (a partition of $E_0'$ as a disjoint union).
\item[$\bullet$] A closed immersion $W_0\subset N_0$ and 
$(W_0,E_0)$ where ${E}_0=(E'_0)^+\pitchfork W_0$.
\item[$\bullet$] A basic object $\B_0=({ W}_0,(J_0,b), { E}_0)$ (see \ref{reform}). 
\end{itemize}

So we assume here that ${E}_0=W_0\pitchfork(E'_0)^+$ (\ref{31110}), but we know nothing about the intersection of hypersurfaces of $(E'_0)^-$ with $W_0$.  If $(E'_0)^-=\emptyset$, or if hypersurfaces of $(E'_0)^-$ do not intersect $F_0=\Sing(J_0.b)$, then we are back to the case in \ref{paso1} (in the tame case).

If $Y_0\subset F_0=\Sing(J_0,b) \subset W_0$ , and if $Y_0$ has normal crossings with $E'_0$ (in $N_0$), then it has normal crossings with $E_0$ (in $W_0$), and defines a transformation of $\B_0$, say $\B_1$, and
\begin{equation}
\begin{array}{ccc}
 (J_0,b) & & (J_1,b) \\
(W_0, E_0)&\longleftarrow & (W_{1},E_{1}).
\end{array}
\end{equation}

We define now an assignment of closed sets 
$$(\mathcal{F}',(N_0,E'_0)),$$ (an embedded basic object) which will rely on the three previous conditions: on the partition of $E'_0$,  on $\B_0=({ W}_0,(J_0,b), { E}_0)$, and on the embedding $W_0\subset N_0$.

A) If $Y_0\subset F_0=\Sing(J_0,b) \subset W_0$, and if $Y_0$ has normal crossings with $E'_0$ (in $N_0$), then set

\begin{equation}
\xymatrix@R=0pc@C=2pc{
F_0=\Sing(J_0,b) & F_1=\Sing(J_1,b)\\
(N_0, E'_0) & (N_{1},E'_{1})\ar[l]
}
\end{equation}
and define a partition $$E'_{1}=(E'_{1})^+ \cup (E'_{1})^-$$
where $(E'_{1})^-$ consists of the strict transform of hypersurfaces 
in $(E'_{0})^-$, and $(E'_{1})^+$ consists of the strict transform of hypersurfaces 
in $(E'_{0})^+$ together with the new exceptional hypersurface.
Note that ${E}_1=(E'_1)^+\pitchfork W_1$. 

\vspace{0.2cm}

B) If $(N_0, E'_0)\longleftarrow  (N_{1},E'_{1})$ is obtained from a smooth 
morphism, then set
\begin{itemize}
\item[$\bullet$] $(N_1,E'_1)$ with $E_1'=(E'_1)^+\cup(E'_1)^-$,
\item[$\bullet$]  $W_1\subset N_1$ and 
$(W_1,E_1)$ where ${E}_1=(E'_1)^+\pitchfork W_1$,
\item[$\bullet$] a basic object $\B_1=({ W}_1,(J_1,b), { E}_1)$, 
\end{itemize}
simply by taking pull-backs.

\end{parrafo}

\begin{definition} \label{ebgbo}
\ \ A) Fix, as above, $(N_0,E'_0)$ together with a partition $E'_0=(E_0')^+\cup (E_0')^-$ 
(disjoint union).
An assignment 
$(\mathcal F , (N_0,E'_0))$ is said to be an {\em embedded general basic object} of dimension $d$ if there is a cover of $N_0$, say $\{U_{\lambda}\}$, and for each restriction  
$$((\mathcal F)_{U_\lambda}, 
(U_\lambda, (E'_0)_\lambda))  \ \hbox{ and }  \  (E'_0)_\lambda=
(E_0')_\lambda^+\cup (E_0')_\lambda^-$$
there is a closed smooth $d$-dimensional subscheme $\overline{U}_\lambda$ of $U_\lambda$, and a basic object 
$$\B_\lambda=({\overline U}_\lambda,(J_\lambda,b_\lambda), ( {\overline E'}^+_0)_\lambda),$$
where $( {\overline E'}^+_0)_\lambda= { E'}^+_0\pitchfork {\overline U}_\lambda$ (\ref{31110}), so that the conditions in \ref{ods} holds. Namely  that $((\mathcal F)_{U_\lambda}, (U_\lambda, (E'_0)_\lambda))$ is the assignment of closed sets defined, as above, by the partition 
$ (E'_0)_\lambda=
(E_0')_\lambda^+\cup (E_0')_\lambda^-$ and $\B_\lambda$. 

\vspace{0.2cm}

B) Fix $(\mathcal F , (N_0,E'_0))$ as in A), and let $F_0$ be the closed set assigned to $(N_0,E'_0)$. Then $(\mathcal F , (N_0,E'_0))$ is a {\em tamely embedded general basic object}, or simply a  {\em tame general basic object}, if no hypersurface of $(E_0')_\lambda^-$ intersects $F_0$ (e.g., if $(E_0')_\lambda^-=\emptyset$).
\end{definition}

The following Theorem of Hironaka, which is an extension of Theorem \ref{rfj}, will allow us to define Hironaka's functions on general basic objects.

\begin{theorem}\label{teo23} Fix $(N_0,E'_0)$, and $E'_0=(E_0')^+\cup (E_0')^-$ 
(disjoint union). Assume that:
\begin{enumerate}
\item there are closed smooth subschemes of $N_0$, say ${\overline W}_0$ and ${\overline V}_0$.

\item there are basic objects $({\overline W}_0, (J_0,b), {\overline E'}^+_0))$ (where $ {\overline E'}^+_0=( {E'}_0)^+\pitchfork W_0$) and  $({\overline V}_0,(K_0,d),   {\overline F'}^+_0))$ (where $ {\overline F'}^+_0=( { E'}_0)^+\pitchfork V_0$) defining assignments of closed sets over $(N_0, E'_0)$ as in \ref{ods}.
\end{enumerate}

If dim ${\overline W}_0$=dim ${\overline V}_0$, and if both assignments of closed sets over $(N_0, E'_0)$ coincide (i.e., define the same closed sets), then 
$$ \frac{\nu_x(J_0)}{b}=\frac{\nu_x(K_0)}{d}$$
at any $x\in \Sing(J_0,b)=\Sing(K_0,d)$.

\end{theorem}

\begin{remark} Note that Remark \ref{rk72} extends to this context:

1) The transform of an embedded general basic object is an embedded general basic object.

2) If $F_0\subset W_0\subset N_0$ is the closed set assigned by the $d$-dimensional general basic object $(\mathcal F , (N_0,E_0))$, then there is a function 
$ \ord^d: F_0 \longrightarrow \mathbb Q$
defined by patching 
$ \ord^d_\lambda: \Sing(J_\lambda, b_\lambda) \longrightarrow  \mathbb Q.$

Also Corollary \ref{cor28} extends word by word.
\end{remark}

\end{parrafo}

The last reformulation of Hironaka's Theorem \ref{H5} can now be strengthen as follows:
 
 \begin{theorem}\label{H6} (\cite{BV1}, Prop 11.4) Fix $X_0\subset N_0$ and $(N_0,E'_0)$. Set $\overline h =\max HS_{X_0}$.  The assignment of closed sets $ (\mathcal F({\overline h}), (W_0,E_0))$   defined by $X_0$ and $\overline h $ in Example \ref{ejHSa}, is an embedded general basic object of dimension $d$, where $d$ is dimension of $X_0$ locally at all closed point of its highest Hilbert-Samuel stratum.

\end{theorem}

\section{The ingredients for Constructive Resolution}

\begin{parrafo}
In this section we aim to define a set $T$ and the function $g$, with values at $T$, leading to the constructive resolution of basic objects as was stated in \ref{crobo}.

Recall that a basic object was defined by $\B_0=(W_{0},(J_{0},b), E_{0})$, where $J_{0}\subset\calo_{W_{0}}$ is a non-zero ideal, $b$ is a positive integer and 
$E_0=\{H_1, \dots H_s\}$ are smooth hypersurfaces in $W_0$ with only normal crossings.  So $(J_{0})_{\xi}\neq 0$ for any $\xi\in W_{0}$.  The singular locus is the closed set $ F_{0}=\Sing(J_{0},b)=\{\xi\in
W_{0}\mid \mathbf{\nu}_{\xi}(J_{0})\geq b\} \subset W_{0}.  $

A center $ Y_{0} $ is said to be \emph{permissible} for the basic
object if $ Y_{0} $ is permissible for $ (W_{0},E_{0}) $
(see \ref{p1}) and $ Y_{0}\subset F_{0} $.  Let
$
W_{0} \overset{\pi_{Y_{0}}}{\longleftarrow} W_{1} $
be the blow-up with center $ Y_{0} $, and denote by $H_{s+1}$ the
exceptional hypersurface.  Assume that $Y_0$ is irreducible with generic point $y_0$ ($\in\Sing(J_0,b)$). There is an ideal $\bar{J}_{1}\subset\calo_{W_{1}} $ such that $
J_{0}\calo_{W_{1}}=I(H_{s+1})^{c_{1}}\bar{J}_{1} $ where $ c_{1}=\nu_{y_0}(J_0)\geq b
$.

We fix on $J_1$ the factorization
\begin{equation}\label{lafac}
	J_{1}=I(H_{s+1})^{c_{1}-b}\bar{J}_{1}  
\end{equation}	
and set
 \begin{equation}
\begin{array}{ccc}
 (J_0,b) & & (J_{1},b) \\
(W_0 , E_0) & \overset{\pi_{Y_0}}{\longleftarrow} & (W_{1} ,E_{1})
\end{array}
\end{equation}
as the \emph{transformation} of the basic object. Here $E_1=\{H'_1, \dots H'_s, H_{s+1}\}$, and $H'_i$ is the strict transform of $H_i$, for $1\leq i\leq s$. To ease the notation we write $E_1=\{H_1, \dots H_s, H_{s+1}\}$.

The value $\frac{\nu_{y_0}(J_0)}{b}(=\ord(y_0))$ depends only on the
weak equivalence class of $\B_0$
and so does
$$\frac{c_1-b}{b}=\frac{\nu_{y_0}(J_0)}{b}-1=\ord(y_0)-1.$$
Note also that $c_1-b$ is the highest exponent of $I(H_{s+1})$ that one can
factor of $J_1$ in (\ref{lafac}).

Consider now a sequence of transformations of basic objects:
\begin{equation}\label{ldes}
\begin{array}{cccccccc}
 (J_0,b) & & (J_1,b) &  & &  &(J_r,b)\\
(W_{0}, E_0) & \overset{\pi_{Y_0}}{\longleftarrow} & (W_{1}, E_1) & \overset{\pi_{Y_1}}{\longleftarrow}  &\cdots  &\overset{\pi_{Y_{r-1}}}{\longleftarrow} &
(W_{r}, E_r)
\end{array}
\end{equation}
with irreducible centers $Y_{i-1}$. For each index $ i $ we fix a factorization
\begin{equation} \label{ExprJi}
	J_{i}=I(H_{s+1})^{a_{1}}\cdots I(H_{s+i})^{a_{i}}\bar{J}_{i}
\end{equation}
so that 
\begin{equation} \label{ExprX}
\frac{a_j}{b}=\ord(y_{j-1})-1,
\end{equation} 
where $y_{j-1}$ denotes the generic point of $Y_{j-1}\subset W_{j-1}$.
Note that $ a_{j}$ is the highest power of
the ideal $I(H_{s+j})$ that divides $J_i$, so this factorization is unique.
\end{parrafo}

\begin{parrafo} \label{DefResBas}
Recall that a sequence (\ref{ldes}) is said to be a \emph{resolution} of
$\B_0=(W_{0},(J_{0},b),E_{0})$ if $ \Sing(J_{r},b)=\emptyset $ (\ref{riic}).  A resolution involves only blow-ups but we will also take into account smooth morphisms.  
These auxiliary morphisms appear in the proof 
of Theorem \ref{lfdhi} (see \ref{trdh}).  From the point of view of constructive resolution it is natural to consider upper semi-continuous functions which are defined up to weak
equivalence.
\end{parrafo}

\begin{definition}\label{casolis} If
 \begin{equation}\label{SLis}
\begin{array}{ccc}
 (J_k,b) & & (J_{k+1},b) \\
(W_k , E_k) & \longleftarrow & (W_{k+1} ,E_{k+1})
\end{array}
\end{equation}
is given by a smooth morphism, lift $J_k$ together with its
factorization in (\ref{ExprJi}), and set
\begin{equation} \label{ExprJ2}
	J_{k+1}=I(H_{s+1})^{a_{1}}\cdots I(H_{s+k})^{a_{k}}\bar{J}_{k}\calo_{W_{k+1}}
\end{equation}
by taking pull back on the previous data.
\end{definition}

This definition allows us to extend expressions (\ref{ExprJi}) to local
sequences (Definition \ref{lsbo}).  Furthermore, Corollary \ref{rfj} ensures
also that $\frac{a_j}{b}$ is determined by the weak
equivalence class of $(W_{0},(J_{0},b),E_{0})$.
Note also that $a_j$ is the highest power of $I(H_{s+j})$ that divides $J_i$.

\begin{definition} \label{Defword}
For any local sequence, say 
\begin{equation}\label{2dee}
\begin{array}{cccccccc}
 (J_0,b) & & (J_1,b) & & &  &(J_r,b)\\
(W_{0}, E_0) & \longleftarrow & (W_{1}, E_1) & \longleftarrow &
\cdots  &\longleftarrow &
(W_{r}, E_r)
\end{array}
\end{equation}
and any index $ i $, we have fixed an expression of $ J_{i} $ in 
(\ref{ExprJi}). We now define functions	
$$\word_{i}:\Sing(J_{i},b)\longrightarrow\frac{1}{b}\mathbb{Z}\subset\mathbb{Q},\quad \word_{i}(\xi)=\frac{\nu_{\xi}(\bar{J}_{i})}{b},$$
$$\hbox{and }\ \ord_{i}:\Sing(J_{i},b)\longrightarrow\frac{1}{b}\mathbb{Z}\subset\mathbb{Q},\quad \ord_{i}(\xi)=\frac{\nu_{\xi}(J_{i})}{b}.$$
The second is Hironaka's function $\ord$ (see (\ref{ord})). This defines two different assignments of functions (see Definition \ref{def229}), and both
depend only on the weak equivalence class of $\B_0=(W_{0},(J_{0},b),E_{0})$.
Note that for $i=0$, $ \word_{0}=\ord_{0} $.
\end{definition}

\begin{remark}\label{satf}
Constructive resolution of singularities is built around Hironaka's function $\ord$. We argue by looking at the rational numbers in
(\ref{ExprX}), and the new function $\word$.  Note that
$$ \word_i(\xi)=\ord_i(\xi)-\sum_{j\  \xi\in H_j}\frac{a_j}{b}$$
for $\xi\in \Sing(J_i,b)$, and that the right hand side is expressed in terms of Hironaka's function.
In fact, each rational number $\frac{a_j}{b}$ is defined in terms of
Hironaka's functions (see (\ref{ExprX})), and hence every $\frac{a_i}{b}$ depends only on the weak equivalence class of $\B_0=(W_0,(J_0,b),E_0)$. The function $\word_i$ is one of the so called satellite functions, as it is expressed entirely in terms of the function $\ord$.

If $\B_0=(W_0,(J_0,b),E_0)$ is replaced by a
weakly equivalent $(W_{0},(K_{0},d),E_{0})$, the invariants introduced
here do not distinguish them.  This is a good starting point in the search of invariants for constructive resolution.

Note finally that if $\sigma: W'_0 \longrightarrow W_0$ is smooth, then successive pull-backs
applied to (\ref{2dee}) will define a local sequence of $(W'_0,(J_0,b),E'_0)$ (the pull-back of $(W_0,(J_0,b),E_0)$), say
\begin{equation}\label{2de3e}
\begin{array}{cccccccc}
 (J'_0,b) & & (J'_1,b) &  & &  &(J'_r,b)\\
(W'_{0}, E'_0) & \longleftarrow & (W'_{1}, E'_1) & \longleftarrow &\cdots  &\longleftarrow &
(W'_{r}, E'_r),
\end{array}
\end{equation}
and smooth morphisms $\sigma_i: \Sing(J'_i,b) \longrightarrow \Sing(J_i,b)$. It follows from (\ref{ExprJ2}) that 
\begin{equation}\label{wordsm}
\word_i(\sigma_i(\xi))=\word_i(\xi)
\end{equation}
 for any $\xi\in \Sing(J'_i,b)$, and that also the rational numbers $\frac{a_j}{b}$ coincide (at the pull-back).
\end{remark}

\begin{remark}\label{rk460}(Non-increasing property).  It is easy to check that if
$W_{i-1}  \overset{\ \pi_{i-1}}{\longleftarrow}W_i $ in (\ref{2dee}) is obtained by blowing up $Y_{i-1}\subset\Max\word_{i-1}\subset\Sing(J_{i-1},b) $, then
\begin{equation} \label{IneqwordPt}
    \word_{i-1}(\pi_{i-1}(\xi_{i}))\geq\word_{i}(\xi_{i})
\end{equation}
for any $ \xi_{i}\in\Sing(J_{i},b) $.  So if $
Y_{i-1}\subset\Max\word_{i-1} $ for any index $i$ for which $ W_{i-1} \longleftarrow
W_i$ is a blow-up with center $Y_{i-1}$,
then
\begin{equation} \label{Ineqword}
\max\word_{0}\geq\cdots\geq\max\word_{r}
\end{equation}
\end{remark}
\begin{corollary}

Fix a basic object $\B_0=(W_{0},(J_{0},b),E_{0})$ and consider the
assignment of closed sets $ (\mathcal F_{\mathcal B}, (W_0,E_0))$ in
Example \ref{ejdebo}.  Namely, for any local sequence as in
(\ref{rlcww}) set
\begin{equation}\label{SeqTrBas}
\xymatrix@C=1.7pc@R=0pc{
F_0= \Sing(J_0,b)  &F_1=\Sing (J_1,b) & &  F_r=\Sing(J_r,b)&\\
(W_{0}, E_0) &  (W_{1}, E_1)\ar[l] & \cdots\ar[l]
&(W_{r}, E_r),\ar[l]
}\end{equation}
 and define functions $$\word_i: F_i \longrightarrow \mathbb Q \ \hbox{ for } \ \ 0 \leq i \leq r.$$

Then:

\begin{enumerate}
\item This defines an assignment of closed sets and functions which is
non-increasing (\ref{ddusc}).

\item This assignment is independent of the weak equivalence class of
$\B_0=(W_{0},(J_{0},b),E_{0})$ (i.e., if $\B_0=(W_{0},(J_{0},b),E_{0})$
and $\B'_0=(W_{0},(K_{0},d),E_{0})$ are weakly equivalent, they both
define the same assignment of closed sets and functions).
\end{enumerate}
\end{corollary}

\begin{remark}\label{rmk48}
The functions $\word$ are defined in terms of the functions order, and satisfies the condition of our Handy Lemma \ref{Handy}. In particular, the inequalities  (\ref{Ineqword}) are as in Remark \ref{kkr}. This is our second example of an assignment of closed sets and functions with the non-increasing property in Definition \ref{ddusc} (see also Example \ref{35}).
\end{remark}

\begin{remark}\label{sfrobo} {\bf (Strategy for resolution of basic objects).}

Fix a basic object $\B_0=(W_{0},(J_{0},b),E_{0})$ and a sequence (\ref{SeqTrBas}); note that for each index $k$ ($0\leq k \leq r$), $ \max\word_{k}\in\dfrac{1}{b}\mathbb{Z} $.  We shall
indicate below that if $ \max\word_{r}=0 $ in (\ref{Ineqword}) then it
is simple to ``extend'' (\ref{SeqTrBas}) to a resolution.
A resolution of a basic object involves
only blow-ups. Consider a sequence of blow-ups
\begin{equation}\label{lant}
\xymatrix@R=0pc@C=3pc{
 (J_0,b) & (J_1,b)   &  &(J_r,b)\\
(W_{0}, E_0) & (W_{1}, E_1)\ar[l]_{\ \ \pi_{Y_0}} & \dots\ar[l]_{\ \ \ \ \ \ \pi_{Y_1}} & (W_{r}, E_r)\ar[l]_{\!\!\pi_{Y_{r-1}}}
}
\end{equation}
 with centers $Y_i \subset
\Max\word_i$ for all index $i$. So
\begin{equation}\label{des491}
\max\word_{0}\geq\cdots\geq\max\word_{r} \hbox{ (as in (\ref{Ineqword})).}
\end{equation}
As $\max\word_{k}\in\dfrac{1}{b}\mathbb{Z} $, in order to define a
resolution of the basic object it would be enough to have a procedure
of choosing centers $Y_i$ so that we may extend (\ref{lant}) in such a way that
$$\max\word_{0}\geq\cdots\geq\max\word_{r}=\cdots=
\max\word_{R-1} > \max\word_{R},$$
for some index $R\geq r$.
In fact, this would lead ultimately to the case
$\max\word_{R}=0$.

If $\max\word_{r}=0$, then $\bar{J}_{r}=\calo_{W_r}$ in
the factorization of $J_r$ presented in
(\ref{ExprJi}), so
\begin{equation} 
	J_{r}=I(H_{s+1})^{a_{1}}\cdots I(H_{s+r})^{a_{r}}	
\end{equation}
(in a neighborhood of $\Sing(J_r,b)$).  In this case it is simple to
define a totally ordered set $\Gamma$, and an upper semi-continuous
function
\begin{equation}\label{lfhd}
 h_r : \Sing(J_r,b) \longrightarrow \Gamma,
\end{equation}
defined entirely in terms of the rational numbers 
\begin{equation}\label{racn}
 \frac{a_i}{b}\ \ \hbox{for}  \  \ 1 \leq i \leq r
\end{equation} 
so that (\ref{lant}) can be extended to a resolution, say
\begin{equation} 
\xymatrix@R=0pc@C=3pc{
 (J_r,b) & (J_{r+1},b)   &  &(J_N,b)\\
(W_{r}, E_r) & (W_{r+1}, E_{r+1})\ar[l]_{\ \ \pi_{Y_r}} & \dots\ar[l]_{\ \ \ \ \ \ \pi_{Y_{r+1}}} & (W_{N}, E_N)\ar[l]_{\!\!\!\!\pi_{Y_{N-1}}}
}\end{equation}
by setting $Y_{r+j}=\Max h_{r+j}$ for  $r\leq r+j<N$ (see \cite{EncVil97:Tirol}).

So the real difficulty, at least for the construction of a resolution of a basic objects $\B_0=(W_{0},(J_{0},b),E_{0})$, is to construct a sequence (\ref{lant}) so as to come
to the case $\max \word_k=0$.

We now introduce a function, called the
``inductive function'' as it is the key for inductive arguments in resolution of
basic objects. It will takes values at $\mathbb{Q}\times\mathbb{Z}$, which will be ordered lexicographically.
\end{remark}

\begin{definition} \label{Deft}
Fix a basic object $\B_0=(W_{0},(J_{0},b),E_{0})$ and a local sequence as in Remark \ref{rk460}, say
\begin{equation}\label{lant2}
\xymatrix@R=0pc@C=3pc{
 (J_0,b) & (J_1,b)   &  &(J_r,b)\\
(W_{0}, E_0) & (W_{1}, E_1)\ar[l]_{\ \ \pi_{Y_0}} & \dots\ar[l]_{\ \ \ \ \ \ \pi_{Y_1}} & (W_{r}, E_r)\ar[l]_{\!\!\pi_{Y_{r-1}}}
}\end{equation}
 with centers $Y_i \subset
\Max\word_i$ for all $i$ for which $W_i \leftarrow W_{i+1}$ is a blow-up. So
\begin{equation}\label{des2491}
\max\word_{0}\geq\cdots\geq\max\word_{r} \hbox{(see  (\ref{Ineqword})).}
\end{equation}
Set $d_r=\dim W_r$ and let $ s_{0} $ be the smallest
index so that 
\begin{equation}\label{leqqf} \max\word_{s_{0}-1}>\max\word_{s_{0}}=\cdots
=\max\word_{r} 
\end{equation}
(so $ s_{0}=0 $ if $ \max\word_{0}=\cdots=\max\word_{r} $).  Set $
E_{r}=E_{r}^{+}\sqcup E_{r}^{-} $, where $ E_{r}^{-} $ are the hypersurfaces of
$ E_{r} $ which are the strict transforms of hypersurfaces of $
E_{s_{0}} $ (and pull-backs if smooth morphisms appear in the sequence). If $ \max\word_{r}\neq 0$ define 
$$ t^{(d_r)}_{r}:\Sing(J_{r},b)\longrightarrow
(\mathbb{Q}\times\mathbb{Z},\leq) \ \mbox{(lexicographic order)}$$
$$ t^{(d_r)}_{r}(\xi)=(\word_{r}(\xi),n_{r}(\xi)), $$ 
$$\hbox{where }\ n_{r}(\xi)=\left\{
\begin{array}{lll}
    \#\{H\in E_{r}\mid \xi\in H\} & {\rm if} & \word_{r}(\xi)<\max\word_{r}  \\
    \#\{H\in E_{r}^{-}\mid \xi\in H\} & {\rm if} & \word_{r}(\xi)=\max\word_{r}
\end{array}
\right. $$
In the same way we define functions $ t^{(d_{r-1})}_{r-1},t^{(d_{r-2})}_{r-2},\ldots,t^{(d_{r-r_0})}_{s_{0}} $.
\end{definition}

We shall later study in \ref{rmk42} the role of this function when constructing resolutions of basic objects. However, the setting of interest here is not only that of basic objects, but also that of basic objects with closed immersions of smooth schemes $W_0\subset N_0$, which we discuss below.

\begin{parrafo}\label{rkccd}
{\bf Immersions in the tame case (1)}. 

A closer look at the setting
in \ref{paso1} is necessary for a better comprehension of our further
discussion.  Fix $\B_0=(W_{0}, (J_0,b) , E_0)$, and assume now that:
\begin{itemize}
\item[$\bullet$] $W_0$ is a (smooth) closed subscheme in a smooth scheme $N_0$, say
$W_0 \subset N_0$.
\item[$\bullet$] There is a set $E''_0$ of hypersurfaces
with normal crossings in $N_0$, and $E_0=E_0''\pitchfork W_0$
(see (\ref{31110})). 
\end{itemize}
 Let $(\mathcal F, (N_{0}, E''_0))$ denote the
assignment of closed sets defined over $(N_0, E''_0)$ by the previous data. The condition $E_0=E_0''\pitchfork W_0$ ensures that any sequence
of blow-ups, say
\begin{equation}\label{lrant}
\begin{array}{cccccccc}
 (J_0,b) & & (J_1,b) &  & &  &(J_k,b)\\
(W_{0}, E_0) & \overset{\pi_{Y_0}}{\longleftarrow} & (W_{1}, E_1) & \overset{\pi_{Y_1}}{\longleftarrow} &
\cdots  &\overset{\pi_{Y_{k-1}}}{\longleftarrow} &
(W_{k}, E_k),
\end{array}
\end{equation}
induces a sequence of blow-ups
\begin{equation}\label{2rvat}
\begin{array}{cccccccc}
(N_{0}, E''_0) & \overset{\pi_{Y_0}}{\longleftarrow} & (N_{1}, E''_1) & \overset{\pi_{Y_1}}{\longleftarrow} &
\cdots  &\overset{\pi_{Y_{k-1}}}{\longleftarrow} &
(N_{k}, E''_k),
\end{array}
\end{equation}
where $W_k$ is closed in $N_k$, and $E_k$ is defined by restricting to $W_k$
the hypersurfaces of $E''_k$.
\end{parrafo}
\begin{parrafo}\label{agre}
In Def \ref{weq} two basic objects $\B_0=(W_0,(J_0,b),E_0)$, and $\B_0'=(W_0,(K_0,d),E_0)$ are said to be weakly equivalent when they define the same closed sets for any local sequence. On the other hand, the notion of local sequence in Def \ref{lsbo} makes use of two kinds of transformations, namely: A) defined by monoidal transformation, B) those defined by a smooth morphism.

It is natural to ask if it suffices to check weak equivalence of $\B_0$ and $\B_1$, by checking the equality of closed sets in Def \ref{weq},
for a certain subclass of local sequences. It can be proved that this is in fact the case. It suffices to consider local sequences where transformations of type B), defined by a
smooth morphism $\sigma_i: W_{i+1} \longrightarrow W_i$, a are restricted to the following two cases:
\begin{itemize}
\item $W_{i+1}=\mathbb{A}^n_k\times W_i$, and $\sigma_i$ is the projection on the first coordinate.

\item $W_{i+1}$ is an open subset of $W_{i}$ and $\sigma_i$ is the inclusion.
\end{itemize}
\end{parrafo}
This fact is well know (see e.g. \cite{BGV}). So throughout this paper one could also have consider only smooth maps of these two prescribed forms. We will not use this fact in this presentation.
\vspace{0.2cm}

\begin{parrafo}
{\bf Immersions in the tame case (2). Neglecting the ambient space}.
\label{38} 

In \ref{rkccd} only blow-ups were considered.  If $(N_0,E_0'')\longleftarrow (N_1,E_1'')$ is now the pull-back defined by a smooth morphism, then the inclusion $W_0\subset N_0$ can be
lifted to $W_1\subset N_1$.  Moreover, this gives rise to a smooth
morphism $(W_0,E_0) \longleftarrow (W_1,E_1)$, and $E_1=E_1''\pitchfork W_1$. But, unfortunately, not every smooth
morphism $(W_0,E_0) \longleftarrow (W_1,E_1)$ arises from one over $(N_0,E_0'')$. The claim is clearly true if we consider smooth morphism as in \ref{agre}, and if the reader is willing to accept the statement therein, then Prop \ref{pr382}, below, is easy to check, and Lemma \ref{lem415} is avoidable.

Firstly we show that if the basic objects $\B_1=(W_{0}, (J_0,b) , E_0)$ and $\B_2=(W_{0}, (K_0,d), E_0)$ are weakly equivalent, then they both define the same assignment of closed sets, say $(\mathcal
F, (N_{0}, E''_0))$, (see \ref{rkccd}).  Fix $(\mathcal F, (N_{0},
E''_0))$, and the notation as in Definition \ref{afcs}, and set
\begin{equation}\label{a}
\begin{array}{cccccccc}
 F_0 & & F_1 & &  & & F_{r}\\
(N_0,E''_0)&\longleftarrow & (N_1,E''_1)& \longleftarrow &\dots &\longleftarrow &(N_{r},E''_{r})\\
\end{array}
\end{equation}
a sequence of blow-ups and smooth morphisms for which closed sets are assigned.
This induces
\begin{equation}\label{b}
\begin{array}{cccccccc}
 F_0 & & F_1 & &  & & F_{r}\\
(W_0,E_0)&\longleftarrow & (W_1,E_1)& \longleftarrow &\dots &\longleftarrow &(W_{r},E_{r})
\end{array}
\end{equation}
(same $F_i$), where 
\begin{enumerate}
\item $W_i$ is a (closed) smooth subscheme in $N_i$.

\item $F_i=\Sing(J_i,b)$.
\end{enumerate}

So (\ref{a}) induces (\ref{b}), and (\ref{b}) is a local sequence of the basic object $\B_1=(W_{0}, (J_0,b) , E_0)$. Thus, if $\B_1$ and $\mathcal B_2$ are weakly equivalent
they both define the same $(\mathcal F, (N_{0}, E''_0))$.

Secondly, we claim that the converse holds, namely that $\mathcal B_1$ and
$\mathcal B_2$ must be weakly equivalent if they both define $(\mathcal
F, (N_{0}, E''_0))$. This converse will be essential for some inductive arguments that we will be used later. Recall that basic objects are to be considered up to weak equivalence.

The difficulty in proving this converse, addressed in Proposition \ref{pr382}, is that it is not clear that any local sequence of $\B_1=(W_{0}, (J_0,b) , E_0)$ will arise in this way (from a local sequence
over $(N_0,E''_0)$).  In fact, given an immersion $W_i\subset N_i$,
and a smooth morphism $W_i\longleftarrow W_{i+1}$, it is not clear that
there will be a smooth morphism $N_i\longleftarrow N_{i+1}$ inducing the
latter.  For this reason it could be expected that two basic
objects, say $(W_{0}, (J_0,b), E_0) $ and $(W_{0}, (K_0,d), E_0)$, define the same $(\mathcal F, (N_{0}, E''_0))$ without being weakly equivalent.  The following Proposition \ref{pr382} settles this point. Moreover, it says that the ambient space $N_0$, in which $W_0$ is
included, can be neglected for the purpose of resolution.
\end{parrafo}

\begin{proposition} \label{pr382}
Fix $(N_{0}, E''_0)$ and $ (W_{0}, E_0)$ as above, so ${E}_0=E''_0\pitchfork W_0$ as in (\ref{31110}). 
\begin{enumerate}
\item Two basic objects $\B_1=(W_{0}, (J_0,b) , E_0)$ and $\B_2=(W_{0}, (K_0,d), E_0) $, define the same assignment of closed sets over $(N_{0},E''_0)$ if and only if they are weakly equivalent.

\item {\bf  Neglecting the ambient space:} A resolution of the basic object $(W_{0}, (J_0,b) , E_0)$
defines a resolution of $(\mathcal F, (N_{0}, E''_0))$ (see Definition \ref{racs}).
\end{enumerate}
\end{proposition}

\begin{proof} 

Suppose  $\B_1$ and  $\B_2$ are not weakly equivalent. There must be an index $r\geq 0$, and a common local sequence  for both, say
\begin{equation}\label{lranpp}
\begin{array}{cccccccc}
 (J_0,d) & & (J_1,d) &  & &  &(J_r,d)\\
(W_{0}, E_0) &  \overset{\gamma_0}{\longleftarrow} & (W'_{1}, E_1) &   \overset{\gamma_1}{\longleftarrow}  & \dots
& \overset{\gamma_{r-1}}{\longleftarrow}  &
(W'_{r}, E_r)
\end{array}
\end{equation}
and 
\begin{equation}\label{lran1qp}
\begin{array}{cccccccc}
 (K_0,d) & & (K_1,d) &  & &  &(K_r,d)\\
(W_{0}, E_0) &  \overset{\gamma_0}{\longleftarrow} & (W'_{1}, E_1) &  \overset{\gamma_1}{\longleftarrow} &\dots 
& \overset{\gamma_{r-1}}{\longleftarrow}  &
(W'_{r}, E_r)
\end{array}
\end{equation}
so that $\Sing(J_r,b)\neq \Sing(K_r,d)$ in $W'_r$. (Take $r=0$ if 
$\Sing(J_0,b)\neq \Sing(K_0,d)$).

If all $\gamma_i$ are blow-ups, then the discussion in \ref{rkccd} says that this sequence can be lifted 
to a sequence over $(N_0,E''_0)$, and this is a contradiction as we assume that both $\B_1$ and  $\B_2$ define the same assignment over $(N_0,E''_0)$. The problem arises if, for some index $i$, $\gamma_i$ is smooth.

Fix a point $x_r\in W'_r$ so that $\Sing(J_r,b)\neq \Sing(K_r,d)$  locally at $x_r$. Let $x_i\in W'_i$ denote the image of $x_r$ for each index $0\leq i \leq r-1$. There is no harm in modifying this sequence by restrictions to neighborhoods of $x_i$. If one could lift this restricted sequence to a sequence over $(N_0,E''_0)$, this would also lead to a contradiction, as $\Sing(J_r,b)\neq \Sing(K_r,d)$  locally at $x_r$.

The following Lemma ensures that such restriction can be defined, by replacing $W_i$ by an \'etale neighborhood of $x_i$ (say $W_i$ again), so that it can be lifted 
to a sequence over $(N_0,E''_0)$. This also leads to a contradiction as 
$\Sing(J_r,b)\neq \Sing(K_r,d)$  at an \'etale neighborhood of $x_r$. It suffices to treat the case in which $\Sing(J_i,b)= \Sing(K_i,d)$ locally at $x_i$, for each index $0\leq i \leq r-1$. 

Part (2) follows from \ref{rkccd}.
\end{proof}

\begin{lemma}\label{lem415} Fix a closed immersion of smooth schemes $W\subset N$, a smooth morphism 
$ W \overset{\gamma_{}}{\longleftarrow} W_1$ and a closed point $x_1\in W_1$. After restriction to a suitable \'etale neighborhood of $x_1\in W_1$, we may assume that there is a smooth 
morphism $ N \overset{\Lambda}{\longleftarrow} N_1$, so that 
$W_1=\Lambda^{-1} (W)(\subset N_1)$, and $\gamma$ is defined by restriction of 
$\Lambda$ to $W_1$.

\end{lemma}

\begin{proof}
Recall the characterization of smooth morphisms. In this case, if $n$ denotes the dimension of the fiber of $ W \overset{\gamma_{}}{\longleftarrow} W_1$ locally at $x_1$, then $(W_1,x_1)$ is an \'etale 
neighborhood of $(x,\mathbb{O})$ in $W \times \mathbb{ A}^n$
($x=\gamma(x_1)$). There is a natural lifting of the inclusion $W\subset N$ to 
$W \times \mathbb{ A}^n\subset N \times \mathbb{ A}^n$. 

An \'etale neighborhood of  $  N \times \mathbb{ A}^n $ at $(x,\mathbb{O})$ induces, by taking the pull-back of the inclusion, an \'etale neighborhood of  $  W \times \mathbb{ A}^n $ at $(x,\mathbb{O})$.
It suffices to show that this induced neighborhood can be defined so as to dominate $(W_1,x_1)$. In other words, fix an inclusion 
of smooth schemes, say $Z_1\subset Z_2$, a closed point $x\in Z_1$, and an \'etale neighborhood $(Z'_1, x')$ of $(Z_1, x)$. It remains to prove that there is a \'etale neighborhood $(Z_2, x)$, which induces on the subscheme $Z_1$ an \'etale neighborhood that dominates $(Z'_1, x')$.  To prove this last claim make use of a well known property of \'etale topology for schemes over perfect fields: given smooth scheme and a closed immersion, as is the case for $Z_1\subset Z_2$ (locally at $x$), after restriction to a suitable  \'etale neighborhood of  $(Z_2, x)$, there is a retraction of $Z_2$
on the restriction of $Z_1$ (see, e.g. \cite{BV1}). Note that if $Z_1\subset Z_2$ admits a retraction, say $Z_2\to Z_1$, then the fiber product with an \'etale morphism $Z'_1\to Z_1$ define a scheme with an \'etale morphism over $Z_2$ that fulfills the required condition.
%
\end{proof}

\begin{parrafo}\label{rk399}
{\bf Functions  on embedded basic objects (non-tame case).}

Set $(N_{0}, E'_0) $, and consider, as in \ref{ods},
the assignment of closed sets on $(N_{0}, E'_0) $ defined by a closed immersion $W_0 \subset N_0$, a partition $E'_0=(E_0')^+\cup (E_0')^-$, and a basic object  $\B_0=(W_{0}, (J_0,b), E_0)$ where $E_0$ is defined by restriction to $W_0$ of hypersurfaces in $(E_0')^+$ (namely, ${E}_0=(E_0')^+\pitchfork W_0$ as in (\ref{31110})).
 Let $( \mathcal F, (N_{0}, E'_0))$ be the assignment defined in this way.

\vspace{0.2cm}

\noindent{\bf Assumption 1:} Suppose that $( \mathcal F, (N_{0}, E'_0))$ assigns closed sets to
the local sequence
\begin{equation}\label{4rvat}
\begin{array}{cccccccc}
(N_{0}, E'_0) & \longleftarrow & (N_{1}, E'_1) & \longleftarrow &
\cdots  &\longleftarrow &
(N_{k}, E'_k).
\end{array}
\end{equation}
In such a case this sequence induces a local sequence over $\B_0$, say
\begin{equation}\label{54ant}
\begin{array}{cccccccc}
 (J_0,b) & & (J_1,b) &  & &  &(J_k,b)\\
(W_{0}, E_0) & \longleftarrow & (W_{1}, E_1) & \longleftarrow &
\cdots  &\longleftarrow &
(W_{k}, E_k),
\end{array}
\end{equation}
where 
 \begin{equation}\label{88Lis}
\begin{array}{ccc}
 (J_i,b) & & (J_{i+1},b) \\
(W_i , E_i) & \longleftarrow & (W_{i+1} ,E_{i+1})
\end{array}
\end{equation}
is obtained by a pull-back if $(N_{i}, E'_i) \longleftarrow (N_{i+1},
E'_{i+1}) $ is given by a smooth morphism, or (\ref{88Lis}) is a
blow-up with a  center $Y_i$ with normal crossings with
$E'_i$ in $N_i$ (in particular with $E_i$ in $W_i$); and of course the closed sets assigned to (\ref{4rvat}) are $F_i=\Sing(J_i,b)$, $i=0, \dots ,k$.

\noindent{\bf Assumption 2:} Suppose that $Y_i \subset \Max\word_i$ for
each index $i$ for which $(N_{i}, E'_i) \longleftarrow (N_{i+1},
E'_{i+1}) $ is a blow-up.
According to (\ref{Ineqword}), this second assumption ensures that
$$\max\word_{0}\geq\cdots\geq\max\word_{k}. $$
\end{parrafo}

\begin{definition} \label{Deftem}
Set $
E'_{k}=(E')_{k}^{+}\sqcup (E')_{k}^{-} $, where $ (E')_{k}^{-} $ consists of all hypersurfaces of
$ E'_{k} $ which are strict transforms of hypersurfaces of $
(E'_{{0}})^- $. Set $d_k=\dim W_k$ and define
$$ t{(em)}^{(d_k)}_{k}:\Sing(J_{k},b)\longrightarrow
(\mathbb{Q}\times\mathbb{Z},\leq) \ \mbox{(lexicographic order)}$$
$$ t{(em)}^{(d_k)}_{k}(\xi)=(\word_{k}(\xi),n{(em)}_{k}(\xi)), $$ 
$$ n{(em)}_{k}(\xi)=
\#\{H\in (E')_{k}^{-}\mid \xi\in H\} 
 $$
for $\xi \in \Sing(J_k,b)$. In the same way we define functions $ t{(em)}^{(d_{k-1})}_{k-1},t{(em)}^{(d_{k-2})}_{k-2},\ldots,t{(em)}^{(d_0)}_{{0}} $.
\end{definition}

\begin{corollary}\label{cor225}
Let $(\mathcal F_{\mathcal B}, (N_0,E'_0))$ be defined by $\B_0=(W_{0},(J_{0},b),E_{0})$, and by $(N_0, E'_0)$ with a partition $E'_0=(E_0')^+\cup (E_0')^-$, as
in \ref{rk399}.  Consider all local sequences 
 \begin{equation} 
\xymatrix@C=1.5pc@R=0pc{
F_0= \Sing(J_0,b)  &F_1=\Sing (J_1,b) &  &F_r=\Sing(J_r,b)\\
(N_{0}, E'_0)  & (N_{1}, E'_1)\ar[l]  &\cdots\ar[l]  &(N_{r}, E'_r)\ar[l]\\ 
}\end{equation}
with $ Y_{i}\subset\Max\word_{i} $ whenever $W_i
\longleftarrow W_{i+1} $ is a blow-up with
center $Y_i$; and the functions 
$t(em)_i: F_i \longrightarrow \mathbb Q\times \mathbb{Z} \ \ 0 \leq i \leq r.$
Then:
\begin{enumerate}
\item This defines an assignment of closed sets and functions which are
non-increasing (see Definition \ref{ddusc}). Namely, 
if $ Y_{i}\subset\Max t(em)_{i} (\subset \Max\word_{i}) $, then
$ t(em)_{i-1}(\pi_i(x)) \geq t(em)_i(x)$ for any $x\in\Sing(J_i,b)$.

\item This assignment depends only on the weak equivalence class of
$\B_0=(W_{0},(J_{0},b),E_{0})$.
\end{enumerate}
In particular, it follows that  $$\max t(em)_{0}\geq\cdots\geq\max t(em)_{r},$$ 
if 
$ Y_{i}\subset\Max t(em)_{i}$ for every index $i$ for which $N_i \leftarrow N_{i+1}$ is defined with center $Y_i$.

\end{corollary}

\begin{remark}\label{rmk417}
A similar statement as that in Corollary \ref{cor225} holds for the functions 
$t_i$ in  \ref{Deft}, defined for a basic object 
$B_0 = (W_0 , (J_0 , b), E_0 )$, and setting $(\mathcal F_{\mathcal B_0}, (W_0,E_0))$ as in Example \ref{ejdebo}, and a sequence (\ref{lant2}). Note that the functions $t$ are defined in terms of the functions order, and that
\begin{equation} \label{Ineqword_1}
    \max t_{0}\geq\cdots\geq\max t_{r}
\end{equation}
holds for a sequence satisfying the conditions of Definition \ref{ddusc} (see Remark \ref{kkr}).

\end{remark}
		
\begin{parrafo}\label{sehl}
Our Handy Lemma \ref{Handy} and Corollary \ref{cor225}, (2), say that one can attach to the maximum value
$\max t(em)_0$ an assignment of closed sets.  Say
$ (\mathcal F_0(\max t(em)_0), (N_0, E'_0)).$

Similarly, the Handy Lemma and \ref{rmk417} also say that one can attach to the maximum value
$\max t_0$ an assignment of closed sets.  Say
$ (\mathcal F_{B_0}(\max t_0), (W_0, E_0)).$

 Let us emphasize here that $ \max t(em)=(p,q)\in \mathbb Q\times \mathbb N$ can have first coordinate $p=0$ (as opposed to 
$ \max t=(p,q)$, defined with $p>0$ in \ref{Deft}).

\end{parrafo}

\begin{remark}\label{rmk42}
 \textbf{On the inductive nature of the function $ t^{(d)} $.}
	
The inductive functions $ t^{(d)} $ were defined in Definition \ref{Deft} for basic objects. Theorem \ref{314} will show how these functions lead 
to a unique resolution by induction on the dimension $d$.

One can check that a basic object $\B_0=(W_{0},(J_{0},b),E_{0})$ is, in particular,
a tame general basic object by setting $N_0=W_0$ and 
$E'_0=E_0$ in Definition \ref{ebgbo}, B).
We shall ultimately show, in Section 5, that the inductive functions $ t^{(d)} $ can also be defined in the context of tame general basic objects. Moreover, the form of induction on $d$, given in
Theorem \ref{314}, is defined in terms of general basic objects (of smaller dimension) which are also tame. 

This will ultimately lead to Theorem \ref{316}, which ensure that 
the functions $ t^{(d)} $ define, by induction, a unique resolution for
general basic objects which are tame. In particular they define resolutions for basic objects, and here we discuss why these functions lead to this particular strategy of resolution, starting with a basic object.
Fix $\B_0=(W_{0},(J_{0},b),E_{0})$ and a sequence of blow-ups
\begin{equation}\label{lamm}
\xymatrix@C=2.7pc@R=0pc{
(J_0,b) &  (J_1,b) &  &(J_r,b)\\
(W_{0}, E_0) & (W_{1}, E_1)\ar[l]_{\ \ \ \pi_{Y_{0}}} &
\cdots\ar[l]_{\ \ \ \ \ \ \ \pi_{Y_{1}}}  &
(W_{r}, E_r)\ar[l]_{\!\!\!\!\pi_{Y_{r-1}}}
}
\end{equation}
 with centers $Y_i \subset
\Max\word_i$ for all index $i$. So
$\max\word_{0}\geq\cdots\geq\max\word_{r}$ (see (\ref{Ineqword})).

As was indicated in Remark \ref{sfrobo}, in order to achieve a resolution of $\B_0$ it suffices to construct the sequence so that 
$\max\word_{r} =0$. So if $\max\word_{r} =0$ we are done. Suppose in what follows that $\max\word_{r} \neq 0$.
Let us indicate why a sequence with this condition can be constructed in terms of the functions 
$ t^{(d)}_{r}:\Sing(J_{r},b)\longrightarrow
(\mathbb{Q}\times\mathbb{Z},\leq) \ \mbox{(lexicographic order)}$
($d=\dim W_0$).
 Recall that 
$t^{(d)}_{r}$ is defined under the assumption that $\max\word_{r} \neq0$, and the second coordinate is bounded by the dimension $d$.
In particular $\max t_r=(p,q)$, $p\neq 0$, $p \in \frac{1}{b}\mathbb N$,
and $q$ is bounded by $d$.

\noindent{\bf 1) Canonical choice of centers}.  Note that $ \Max{t^{(d)}_{r}} $
is a closed set in $W_r$.  Let $R(1)(\Max{t^{(d)}_{r}})$ denote the union of
components of the closed set $ \Max{t^{(d)}_{r}} $ of codimension one (in 
$W_{r}$).
Then $R(1)(\Max{t^{(d)}_{r}})$ is both open and closed in $ \Max{t^{(d)}_{r}} $. After open restriction we may assume that $R(1)(\Max{t_{r}})=\Max{t_{r}}$.
The property is that $R(1)(\Max{t_{r}})$ is a smooth permissible
center, and setting $Y_r=R(1)(\Max{t_{r}})$, we get
\begin{equation}\label{ant2_1}
\begin{array}{cccccccccc}
 (J_0,b) & & (J_1,b) & & & & (J_r,b)  &  & (J_{r+1},b)&\\
(W_{0}, E_0) & \overset{\pi_{Y_0}}{\longleftarrow} & (W_{1}, E_1) & \overset{\pi_{Y_1}}{\longleftarrow} & \dots & \overset{\pi_{Y_{r-1}}}{\longleftarrow} &
(W_{r}, E_r)   &\overset{\pi_{Y_{r}}}{\longleftarrow} &
(W_{r+1}, E_{r+1})
\end{array}
\end{equation}
and $\max t_r> \max t_{r+1}.$ This shows that $R(1)(\Max{t_{r}})$ is our canonical choice of center, and reduces the problem to the case $R(1)(\Max{t_{r}})=\emptyset$.

\vspace{0.2cm}

\noindent{\bf 2) The case $\max t^{(d)}=(p,q)$ and $R(1)(\Max{t^{(d)}_{r}})=\emptyset$}. According to \ref{sehl},
one can attach to the maximum value, say 
$\max t_r$, an assignment of closed sets; say
$ (\mathcal F_r(\max t_r), (W_r, E_r)),$ and the property is that this assignment  can be 
endowed with a structure of a {\em tame} $(d-1)$-dimensional general basic object over $(W_r, E_r^{(1)})$, for a suitable subset 
$E_r^{(1)}\subset E_r$.

\begin{theorem}\emph{(Inductive property of $t^{(d)}$.)}\label{314} Set $\mathcal B_r=(W_r,(J_r,b), E_r)$ as above and $d= \dim W_r$. If 
$\Max t_r$ has no component of dimension $d-1$ (i.e., if  $R(1)(\Max{t^{(d)}_{r}})=\emptyset$), then the assignment of closed sets $ (\mathcal F_r(\max
t_r), (W_r, E_r))$ is equivalent to a tamely 
 embedded general basic object  of dimension $d-1$, defined over $(W_r, E_r^{(1)})$, 
 in terms of a suitable basic object $(W_r,(D_r,b),  E_r^{(1)})$.
 \end{theorem}
 \proof We refer to Lemma 6.12 in \cite{EncVil97:Tirol} for details and a precise statement. First set $s_0$ as in 
 (\ref{leqqf}), and let $E_r^{(1)}$ be the set of exceptional hypersurfaces that arise by blowing up at 
 $Y_{s_0}, Y_{s_0+1}, \dots , Y_{r-1}$. We attach 
 a basic object $(W_r,(D_r,b),  E_r^{(1)} )$ (a so called {\em simple} basic object) which in turn gives rise to a $d-1$-dimensional tame general basic object over $(W_r, E_r^{(1)})$. This $d-1$-dimensional structure is defined by local restrictions 
 to smooth hypersurfaces of maximal contact.
\endproof

\vskip 0.4cm

In particular, if we assume by induction, the existence of resolutions for tame general basic objects of dimension $d-1$, then there is an integer $R>r$, and  a sequence 

\begin{equation}\label{ant2_22}
\begin{array}{ccccccccccc}
 (J_0,b) & & (J_1,b) & & & & (J_r,b)  &  &  & & (J_{r+1},b)\\
(W_{0}, E_0) & \overset{\pi_{Y_0}}{\longleftarrow} & (W_{1}, E_1) & \overset{\pi_{Y_1}}{\longleftarrow} & \dots & \overset{\pi_{Y_{r-1}}}{\longleftarrow} &
(W_{r}, E_r)   &\overset{\pi_{Y_{r}}}{\longleftarrow} & \dots &  \overset{\pi_{Y_{R-1}}}{\longleftarrow} &
(W_{R}, E_{R})
\end{array}
\end{equation}
so that  $(p,q)=\max t^{(d)}_r=\max t^{(d)}_{r+1}= \dots = \max t^{(d)}_{R-1}> \max t^{(d)}_{R}.$

As $p\in\dfrac{1}{b}\mathbb{Z} $, and $q$ is a positive integer bounded by $d$, at some point the value $p$ must drop. So by successive applications of this method we finally come to the case in which 
$\max\word_{R}=0$, as was required.

\end{remark}
\begin{parrafo} As for the uniqueness of the resolution of $\B_0=(W_0,(J_0,b), E_0)$, recall that if we choose centers 
$Y_i \subset \Max t_i^{(d)}$, then 
$\max t_0\geq  \max t_1 \geq \cdots \geq \max t_{r}$ (Remark \ref{rmk417}).
Set $r_0$ as the smallest index for which $\max t_{r_0}=
\max t_{r_0+1}=\dots =\max t_{r}$, and apply Theorem \ref{314} at $r_0$, which states that 
$ (\mathcal F_{r_0}(\max
t_{r_0}), (W_{r_0}, E_{r_0}))$ is equivalent to a tame 
 assignment  of dimension $d-1$ over $(W_{r_0}, E_r^{(1)})$.

 So far we have discussed matters over the basic object $\B_0=(W_0,(J_0,b), E_0)$, but the same applies for a tamely embedded 
 basic object, say $(\mathcal F_{\mathcal B_0}, (N_0,E''_0))$ defined by  $\B_0=(W_{0},(J_{0},b),E_{0})$, and $(N_0, E''_0)$, as in \ref{rkccd}. Recall that a sequence (\ref{lrant}) induces a sequence (\ref{2rvat})

We can reformulate Theorem \ref{314} by:
\begin{theorem}\emph{(Inductive property of $t^{(d)}$ for tamely embedded basic objects.)}\label{315} Set $(\mathcal F_{\mathcal B_0}, (N_0,E''_0))$, $\mathcal B_0$, $r$ and $r_0$ as above. Set $\mathcal B_r=(W_r,(J_r,b), E_r)$ and $d= \dim W_r$. 
If $\Max t^{(d)}_r$ has no component of dimension $d-1$, then the assignment of closed sets $ (\mathcal F_r(\max
t^{(d)}_r), (N_r, E''_r))$ is equivalent to a tamely 
 embedded general basic object  of dimension $d-1$, defined over
 $(N_r, (E''_r)^{(1)})$, where $(E''_r)^{(1)}$ are the hypersurfaces of $E''_r$ that arise by blowing up at $Y_{s_0}$, $Y_{s_0+1}$,\dots , $Y_{r-1}$, and $s_0$ is as in (\ref{leqqf}).
 \end{theorem}
\end{parrafo}
\begin{parrafo}\label{ptfec} 
The definition of the inductive functions $t^{(d)}$ for {\em tamely} embedded general basic objects, to be discussed  Section 5, will be done by patching the functions on the $d$-dimension basic objects involved in Definition \ref{ebgbo}, B). Once this point is settled, the Theorem will be:
 
\begin{theorem}\emph{(Inductive property of $t^{(d)}$  for tamely embedded general basic objects).}\label{316} Let $(\mathcal{F},(N_r,E_r))$
be a tamely embedded general basic object of dimension $d$ (see Definition \ref{ebgbo}). If $\Max t^{(d)}_r$ has no component of dimension $d-1$ then  $(\mathcal F(\max
t^{(d)}_r), (N_r, E_r))$ is $d-1$-dimensional and tamely embedded over $(N_r, E'_r)$, where $E'_r$ is a suitable subset of $E_r$.
 \end{theorem}
 \color{black}
\end{parrafo}

\begin{parrafo}\label{rmk433}
 \textbf{On the role of the function $ t(em)^{(d)} $.}

The functions $ t(em)^{(d)} $ where defined in Definition \ref{Deftem} only for embedded basic objects of dimension $d$ (for the non-tame case). These are, in particular, general basic objects (\ref{ebgbo}). In Section 5, we also prove that 
the functions $ t(em)^{(d)} $ can be defined for any $d$-dimensional general basic object. We shall also indicate why these functions are well adapted to the Hilbert-Samuel stratification and particularly with Theorem \ref{H6}.
 They will enable us to reduce the  resolution of a $d$-dimensional general basic object to that 
of resolution of a {\em tame} general basic object of the same dimension $d$. The latter will be guaranteed by the inductive function $ t^{(d)} $.

In this section we explain why this property will hold by studying 
the functions $ t(em)^{(d)} $ in the (restricted) context of embedded basic objects.

  \end{parrafo}
   
   \begin{parrafo}

 Let $(\mathcal F_{\mathcal B}, (N_0,E'_0))$ be the assignment defined by a basic object $\B_0=(W_{0},(J_{0},b),E_{0})$, and by $(N_0, E'_0)$ with a partition $E'_0=(E_0')^+\cup (E_0')^-$. Consider assignments over $(N_0, E'_0)$ defined by
$(\mathcal F_{\mathcal B_0}, (N_0,E'_0))$, say
 \begin{equation}\label{nredt}
\xymatrix@C=1.5pc@R=0pc{
F_0= \Sing(J_0,b)  &F_1=\Sing (J_1,b) &  &F_r=\Sing(J_r,b)\\
(N_{0}, E'_0)  & (N_{1}, E'_1)\ar[l]  &\cdots\ar[l]  &(N_{r}, E'_r)\ar[l]\\ 
}\end{equation}
with $ Y_{i}\subset\Max t(em)_{i} $ for each $i$ for which $W_i
\longleftarrow W_{i+1} $ is defined with
center $Y_i$. Set $$\mathcal B_r=(W_r,(J_r,b), E_r),$$ 
where $W_r\subset N_r$, and  
$E'_r=(E_r')^+\cup (E_r')^-$ as in Definition \ref{Deftem}. In particular, 
${E}_r=W_r\pitchfork(E'_r)^+$.

In this setting 
$$\max t(em)_{0}^{(d)}\geq\cdots\geq\max t(em)_{r}^{(d)} \ (\hbox{see Corollary }\ref{cor225})$$
and the Handy Lemma says that a new assignment 
$ (\mathcal F_r(\max t(em)_r^{(d)}), (N_r, E'_r))$
is defined for $ (\mathcal F_{\mathcal B_r}, (N_r, E'_r))$ with the previous partition $E'_r=(E_r')^+\cup (E_r')^-$.

 The closed set assigned by $ (\mathcal F_r(\max t(em)_r^{(d)}), (N_r, E'_r))$ to $(N_r, E'_r)$ is $\Max t(em)_r^{(d)}$. To ease the notation assume that 
 (\ref{nredt}) is a sequence of blow-ups, so $\dim W_r=\dim W_0=d$.
\end{parrafo}

\begin{theorem}\emph{(Descending property of $t^{(d)}(em)$ for embedded basic objects)}\label{317em} Set $\mathcal B_r=(W_r,(J_r,b), E_r)$  and $E'_r=(E_r')^+\cup (E_r')^-$ as above, and $d= \dim W_r$. The assignment of closed sets $ (\mathcal F_r(\max
t(em)_r), (N_r, E'_r))$ is equivalent to a {\em tamely} 
 embedded basic object,  of the same dimension $d$, defined over
 $(N_r, (E'_r)^+)$ by a basic object $\mathcal D_r=(W_r,(D_r,e),  E_r)$.
 \end{theorem}

Note that $E_r$ is the same for  $\mathcal D_r$ and $\mathcal B_r$
(in particular ${E}_r=W_0\pitchfork(E'_r)^+$).
Before we address the proof of Theorem \ref{317em} let us indicate that it states that $ (\mathcal F_r(\max
t(em)_r), (N_r, E'_r))$ can be identified with an assignment of closed sets 
over $(N_r, (E'_r)^+)$ defined by the closed immersion $W_r\subset N_r$ and a basic object $(W_r,(D_r,b),  E_r)$. This is an embedded assignment in the setting of \ref{rkccd} and \ref{38} (in the tame case). We may therefore 
apply Proposition \ref{pr382} which says that 
a resolution of $ (\mathcal F_r(\max
t(em)_r^{(d)}), (N_r, E'_r))$ is achieved by a resolution of the basic object 
$(W_r,(D_r,b),  E_r)$ (disregarding the embedding in $N_r$).

Since we know how to achieve resolution of basic objects of dimension $d$ (see Remark \ref{rmk42}), one may extend the previous sequence to, say
 \begin{equation}\label{nred1}
\xymatrix@C=1.5pc@R=0pc{
F_0= \Sing(J_0,b)&  &F_r=\Sing (J_r,b) &  &F_R=\Sing(J_R,b)\\
(N_{0}, E'_0)  & \dots\ar[l] &(N_{r}, E'_r)\ar[l]  &\cdots\ar[l]  &(N_{R}, E'_R)\ar[l]\\ 
}\end{equation}
for some $R\geq r$, so that 
$$\max t(em)_{0}^{(d)}\geq\cdots\geq\max t(em)_{r}^{(d)} =\max t(em)_{r+1}^{(d)}=\cdots =\max t(em)_{R-1}^{(d)}> \max t(em)_{R}^{(d)}$$

Let us emphasize here that $ \max t(em)^{(d)}=(p,q)\in \mathbb Q\times \mathbb N$ can have first coordinate $p=0$ (as opposed to 
$ \max t^{(d)}=(p,q)$ with $p>0$).
After successive applications of resolution of basic objects we come to the case in which 
$ \max t(em)_R^{(d)}=(0,0)$, which is the natural analog of the case 
$ \max \word=0$ for basic objects, and can be therefore resolved,
as was indicated in \ref{sfrobo},  by means of the functions $h$ with values on $\Gamma$ (see \ref{lfhd}). In fact, as the second coordinate of $ \max t(em)_R^{(d)}=(0,0)$ is zero, no hypersurface of $(E'_R)^-$ intersects $F_R$ so the conditions in Proposition \ref{pr382} apply for 
$(\mathcal F_{\mathcal B_R}, (N_R,E'_R))$.

\vspace{0.25cm}

\begin{proof428} The sequence (\ref{nredt}) defines:
 
1) $(N_r, E'_r)$ and a partition $E'_r=(E'_r)^+ \cup (E'_r)^-$. 

2) The basic object $\mathcal B_r=(W_r,(J_r,b), E_r)$ and a closed immersion $W_r\subset N_r$.

3) A factorization 
\begin{equation} 
	J_{r}=I(H_{s+1})^{a_{1}}\cdots I(H_{s+r})^{a_{r}}\bar{J}_{r}
\end{equation}
as in (\ref{ExprJi}).

Set $\max t(em)_r^{(d)}=(p,q)$, so 
$p=\max \word_r=\frac{a}{b}\in \frac{1}{b}\mathbb N$, and of course 
$\Max t(em)_r^{(d)}$ is a closed subset in $\Sing(J_r,b)$. 

Note first that if 
$a\neq 0$ then $\Max \word_r=\Sing(J_r,b) \cap \Sing(\bar{J}_{r},a)$. Note also that the second coordinate $q$ is a non-negative integer, and that there are points 
in $\Max \word_r$ which may be in $q$ different hypersurfaces of $(E'_r)^-$, but no point of $\Max \word_r$ is contained in $q+1$ hypersurfaces of $(E'_r)^-$. 

Assume that $(E'_r)^-$ has $M$ hypersurfaces, say 
$(E'_r)^{-}$ =  $ \{ H_1,H_2, \dots
 , H_M  \}$,
and let $\mathcal C (q)$ denote the set of all subsets 
of $ \{ 1,2, \dots
 , M  \}$ with $q$ elements. 

Set $K(q)=\prod_{F\in \mathcal C (q) } \sum_{i \in F}I(H_i)$. Here $K(q)$ is an ideal over $N_r$. As a closed set in $W_r$ is closed also in $N_r$ and one can check that:

A) If $a\neq 0$, then 
$$\Max t(em)_r^{(d)}=\Sing(J_r,b) \cap \Sing(\bar{J}_{r},a)\cap \Sing(K(q),1).$$

B) If $a= 0$, then 
$$\Max t(em)_r^{(d)}=\Sing(J_r,b) \cap \Sing(K(q),1)$$

One can also check that the following properties hold:

{\bf P1)} If $Y$ is a closed and smooth subscheme 
in $\Sing(K(q),1)$, then $Y\subset H_i$ 
whenever $H_i \cap Y\neq \emptyset$ and $H_i \in E_r^-$. In particular, 
if $Y\subset \Max t(em)_r^{(d)}$, then $Y$ has normal crossings with 
$(N_r, E'_r)$ if and only if it has normal crossings with  $(N_r, (E')^+_r)$.

{\bf P2)} 
If
$$
\begin{array}{ccc}
F_r=\Sing (J_r,b) & & F_{r+1}=\Sing(J_{r+1},b) \\
(N_r , E'_r) & \longleftarrow & (N_{r+1} ,E'_{r+1})
\end{array}
$$
is defined by choosing a center $Y$ as above, then it induces 
transforms of the basic objects $(W_r, (J_r,b), E_r)$, 
$(W_r, (\bar{J}_{r},a), E_r)$, $(W_r, (K(q),1), E_r)$, and either 
$ (p,q)=\max t(em)_r^{(d)}>\max t(em)_{r+1}^{(d)},$
or $ (p,q)=\max t(em)_r^{(d)}=\max t(em)_{r+1}^{(d)}.$

In this last case either A) or B) holds for $\Max t(em)_{r+1}^{(d)}$, where the pairs in the right hand side of the equalities are replaced by their transforms.

There is a standard argument to define $(D,e)$ so that $\Sing(D,e)$ is the right hand side in A), or in B), and that such equality is preserved by transformations as in {\bf P2)} (see Exercise 14.4 in \cite{EncVil97:Tirol}).
\end{proof428}

\begin{parrafo}
It will be shown in Section 5, that the functions $ t(em)^{(d)}$ can be defined for embedded general basic objects.

 \begin{theorem}\emph{(Descending property of $t^{(d)}(em)$ for embedded general basic objects).}\label{316em} 
 
 Let $(\mathcal{F},(N_0,E'_0))$ and $E_0'=(E'_0)^+\cup(E'_0)^-$ define a general basic object of dimension $d$, and let $(\mathcal{F}_r',(N_r,E'_r))$ and $E_r'=(E'_r)^+\cup(E'_r)^-$ be defined by a sequence of transformations as in (\ref{nredt}). 
 Then $ (\mathcal F_r(\max
t(em)_r^{(d)}), (N_r, E'_r))$ is equivalent to a tame 
 embedded assignment  of dimension $d$, defined over
 $(N_r, (E'_r)^+)$.
 \end{theorem}
\end{parrafo}

\section{Adaptability of the inductive function $t$ and resolution of singularities}

\begin{parrafo} \label{pare}
{\bf Summarizing the previous discussion}.  
We want to prove that resolution of singularities grows from resolution of basic objects, subjects to the condition that the later is compatible with Hironaka's notion of equivalence (see Definition \ref{weq}). 
In \ref{ods} and in \ref{paso1} we introduce the notions of embedded and of tamely embedded basic objects, respectively. Finally these two notions led us to the notions of embedded {\em general} basic objects in
\ref{ebgbo} (and tamely embedded in \ref{ebgbo}, B)). This extension requires some patching which we discuss below.

The functions $t^{(d)}$ and $t(em)^{(d)}$ were studied in the setting of embedded basic objects. We show now that:

\begin{itemize}
\item[{\bf I)}] The inductive function $t^{(d)}$ can be defined for any tamely embedded {\em general} basic object. And a resolution is attained (essentially) in terms 
of $t^{(d)}$, $t^{(d-1)}$,\dots.

\item[{\bf II)}]  The function $t(em)^{(d)}$ can be defined for any embedded
{\em general} basic object. And it allows us to reduce the resolution of embedded general basic objects to that of tamely embedded general basic objects. 
Resolution of embedded general basic objects of dimension $d$ is obtained (essentially) in terms 
of $t(em)^{(d)}$, $t^{(d)}$, $t^{(d-1)}$,\dots, $t^{(1)}$.
\end{itemize}

These two results will be discussed in Case 0). Some more generality 
will be needed to come from resolution of general basic objects to resolution of singularities. This is discussed in Case 1 and Case 2.
The point is that we have defined the notion of general basic objects making use of an embedding in a smooth scheme $N$ and more precisely as an assignment of closed sets on, say $(N,E)$, where $E$ are hypersurfaces in $N$ with only normal crossings.
A property of constructive resolution of singularities is that one can easily adapt it so that it provides a resolution of singularities which is independent 
of the embedding: Suppose that a reduced scheme $X$ is embedded in two different smooth schemes, say $X\subset N$ and $X\subset M$, where $N$ and $M$ may have different dimension. The problem of resolution of singularities of $X$ will lead us to that of constructing a resolution of a general basic object over $N$, on the one hand, and that of  constructing a resolution of a general basic object over $M$, on the other.
So we will want to know that the two general basic objects, embedded in different spaces, undergo the same constructive resolution. Precise statements of these facts are discussed in cases 1) and 2).

\end{parrafo}

\begin{parrafo}
{\bf Case 0: Two basic object and the same embedding.}

Fix, as in \ref{ods}, a smooth
scheme $N_0$, a set $E'_0$ of hypersurfaces with normal crossings at
$N_0$, together with a partition in two disjoint sets: say $(E'_0)^+$
and $(E'_0)^-$.  Assume now that there are two closed immersions
$W_0\subset N_0$, and $V_0\subset N_0$ of smooth schemes, both $W_0$ and $V_0$ of the same
dimension $d$.

Moreover, suppose that there are two basic objects, say $\B=(W_0,(J_0,b),E_0)$, and $\B'=(V_0,(K_0,d), F_0)$, so that 
${E}_0=(E_0')^+\pitchfork W_0$ and ${F}_0=(E_0')^+\pitchfork V_0$.

Assume finally that both basic objects $\B$ and $\B'$ define {\em the same} assignment of closed sets over
$(N_0,E'_0)$, say $(\mathcal F, (N_0,E'_0))$.

Strictly speaking, our proof in \ref{trdh} will show that Hironaka's
functions order, introduced in (\ref{ord}), coincide for two weakly equivalent basic objects say,
$(W_0,(J_0,b),E_0)$ and $(W_0,(K_0,d),E_0)$ (both with the same $(W_0,E_0)$, which is not the case in our setting).
And the main argument in such proof is that two such basic objects {\em define the same closed
sets}.
If we expect to argue similarly in our context, in which
$\Sing(J_0,b)$ is a closed set in $W_0$ and $\Sing(K_0,d)$ is a closed
set in $V_0$ then, in principle, it makes no sense to say that two closed sets
in different spaces are the same.

We overcome this difficulty simply by viewing them as the same closed
set in $N_0$.  This suffices to show that Hironaka's function
order is well defined on a $d$-dimensional general basic object $(\mathcal F,
(N_0,E'_0))$ (see Theorem \ref{teo23}).
So in this case we do not want to disregard the embedding in $N_0$,
because its is through this embedding that $(W_0,(J_0,b), E_0)$, and
$(V_0,(K_0,d), F_0)$ {\em define the same closed sets} (see Corollary
\ref{rfj}).

The previous discussion already ensures that the functions $t^{(d)}({em})$ are well defined for the embedded basic object $(\mathcal F,
(N_0,E'_0))$ and its transforms, independently of the choice of 
$(W_0,(J_0,b), E_0)$ or
$(V_0,(K_0,d), F_0)$. 

In the particular case in which the embedded basic objects are tame (\ref{paso1}) (namely, if $(E'_0)^-=\emptyset$)  a similar argument proves that the inductive function $t^{(d)}$ is well defined on 
$(\mathcal F,
(N_0,E'_0))$ and its transforms (they are the same independently of the choice of 
$(W_0,(J_0,b), E_0)$ or
$(V_0,(K_0,d), F_0)$ and their transforms).

Theorem \ref{317em} says that after a suitable sequence of, say $r$ transformations, 
$$(\mathcal F_r(\max
t(em)_r^{(d)}), (N_r, E'_r))$$ 
is equivalent to:
\begin{enumerate}
\item[1)] a tamely 
 embedded basic object of dimension $d$, defined over
 $(N_r, (E'_r)^+)$ by a basic object $(W_r,(D_r,b),  E_r)$, where 
 ${E}_r=W_r\pitchfork(E'_r)^+$.
 
\item[2)] a tamely 
 embedded basic object  of dimension $d$, defined over
 $(N_r, (E'_r)^+)$ by a basic object $(V_r,(G_r,c),  F_r)$, where 
 ${F}_r=V_r\pitchfork(E'_r)^+$.
\end{enumerate}

We claim now that the resolution of the $d$-dimensional basic objects $(W_r,(D_r,b),  E_r)$ and that of $(V_r,(G_r,c),  F_r)$ define the same sequence of transformations over $(N_r, (E'_r)^+)$.

We argue again as before. Here $\Max t(em)_r^{(d)}=\Sing (D_r,b)=
\Sing (G_r,c)$ and the same holds for transformations on centers included in $\Max t(em)_{r'}^{(d)}$, $r' \geq r$. So again, these two $d$-dimensional basic objects define the same closed sets if we view them in $(N_r, (E'_r)^+)$ (or in a transform $(N_{r'}, (E'_{r'})^+)$, $r' \geq r$). So both basic objects give rise now to the same tamely embedded basic object, namely $ (\mathcal F_r(\max
t(em)_r^{(d)}), (N_r, (E'_r)^+))$,  of dimension $d$.

The previous discussion also applies to shows that the inductive functions $t^{(d)}$ are defined for tamely embedded general basic objects. As these inductive functions are well defined, the closed sets $\Max t^{(d)}$ are well defined. Now the theorems of the inductive property of the functions  $t^{(d)}$ apply again, and defines now a $d-1$ dimensional tamely embedded general basic object. 
These leads to a full resolution of $ (\mathcal F_r(\max
t(em)_r^{(d)}), (N_r, (E'_r)^+))$, which is defined, by induction on $d$, essentially in terms of the inductive functions, namely $t^{(d)}$, $t^{(d-1)}$, \dots, $t^{(0)}$. Here each $t^{(d-i)}$ is applied on a $(d-i)$-dimensional tamely embedded general basic object, and a $(d-i-1)$-dimensional tamely embedded general basic object is defined by $\max t^{(d-i)}$.
\end{parrafo}

\begin{parrafo}
{\bf Case 1: A same basic object and two different embeddings.}

In the definition of an embedded basic object
$(\mathcal F,
(N_0,E'_0))$ in \ref{ods} we fix 
\begin{enumerate}
\item $(N_0,E'_0)$, and a decomposition $E_0'=(E_0')^+  \cup (E_0')^-$,

\item a closed embedding $W_0\subset N_0$ and a basic object 
$(W_0, (J_0,b), E_0)$, where $(E_0')^+\pitchfork W_0=E_0$.

\end{enumerate}

Suppose that another assignment, say 
$(\mathcal G,
(M_0,F'_0))$,
is defined by
\begin{enumerate}
\item $(M_0,F'_0)$, and a decomposition $F_0'=(F_0')^+  \cup (F_0')^-$,

\item a closed embedding $W_0\subset M_0$ and 
$(W_0, (J_0,b), E_0)$, where $(E_0')^+\pitchfork W_0=E_0$.

\end{enumerate}
Note that both make use of the same $(W_0, (J_0,b), E_0)$. Moreover, quite often this situation arises together with a natural bijection of hypersurfaces in $(E_0')^ +$ with those in  $(F_0')^ +$, and a bijection of hypersurfaces in $(E_0')^ -$ with those in  $(F_0')^ -$, in such a way that the functions 
$t(em)$ {\em naturally} coincide.
For example, take $(W_0,(J_0,b), E_0)$ with $E_0=\emptyset$, and two
different closed embeddings: say $W_0\subset N_0$ and $W_0\subset M_0$
in arbitrary smooth schemes.

A sequence of {\em blow-ups}  over 
$(W_0, (J_0,b),E_0=\emptyset )$, say 
\begin{equation}\label{attet}
\begin{array}{cccccccc}
 (J_0,b) & & (J_1,b) & & & &  &(J_{k_0},b)\\
(W_{0}, \emptyset) & \overset{\pi_{Y_0}}{\longleftarrow} & (W_{1}, E_1) & \overset{\pi_{Y_1}}{\longleftarrow} &
 &\cdots  &\overset{\pi_{Y_{k_0-1}}}{\longleftarrow} &
(W_{k_0}, E_{k_0})
\end{array}
\end{equation}
defines, via the two immersions, two sequences, say
\begin{equation}
\begin{array}{ccccccccc}
(N_{0}, E'_0=\emptyset) & \longleftarrow & (N_{1}, E'_1) & \longleftarrow &
&\cdots  &\longleftarrow &
(N_{k_0}, E'_{k_0})
\end{array}
\end{equation}	
and
\begin{equation}
\begin{array}{cccccccccc}
(M_{0}, F'_0=\emptyset) & \longleftarrow & (M_{1}, F'_1) & \longleftarrow &
 &\cdots  &\longleftarrow &
(M_{k_0}, F'_{k_0})
\end{array}
\end{equation}	
Now we obtain two inclusions, say: $W_{k_0} \subset N_{k_0}$ and
$W_{k_0} \subset M_{k_0}$, and two different assignments defined by:
\begin{enumerate}
\item $(W_{k_0}, (J_{k_0},b), E_{k_0})$ and $(N_{k_0}, E'_{k_0})$,

\item $(W_{k_0}, (J_{k_0},b), E_{k_0})$ and $(M_{k_0}, F'_{k_0})$.
\end{enumerate}

Each exceptional hypersurface arises from a blow-up.  This
provides a natural correspondence of hypesurfaces in $E'_{k_0}$ (in
$N_{k_0}$) with hypersurfaces in $F'_{k_0}$ (in $M_{k_0}$).

These last two embedded basic objects are tame. But we can also modify them so as to be non-tame. For example by taking

\begin{enumerate}
\item $(W_{k_0}, (J_{k_0},b), \emptyset)$ and $(N_{k_0}, E'_{k_0})$,

\item $(W_{k_0}, (J_{k_0},b), \emptyset)$ and $(M_{k_0}, F'_{k_0})$,
\end{enumerate}
where now both are in the setting of \ref{rk399}: each is an immersion which is not tame.

Let $(\mathcal F,
(N_{k_0}, E'_{k_0}))$ denote the first assignment and let $(\mathcal G,
(M_{k_0},F'_{k_0}))$ denote the second where now both are non-tame. One check that this is a situation in which the functions $t(em)^{(d)}$ ($d=\dim W_0$) are the same both for 
$(\mathcal F,
(N_{k_0}, E'_{k_0}))$, for $(\mathcal G,
(M_{k_0},F'_{k_0}))$, and for simultaneous transforms.

Following the notation of
Theorem \ref{317em}, this leads, on the one hand to $ (\mathcal F_{k_0}(\max
t(em)_{k_0}^{(d)}), (N_{k_0}, E'_{k_0}))$, and, on the other hand to $ (\mathcal F_{k_0}(\max
t(em)_{k_0}^{(d)}), (M_{k_0}, F'_{k_0}))$. The theorem asserts that both are now tamely embedded basic objects of dimension $d=$ dim($W_{k_0}$), and both defined by basic objects on $W_{k_0}$.

 Proposition \ref{pr382} ensures that these two tamely embedded basic object define two basic objects on $W_{k_0}$
which are weakly equivalent.
For the resolution of a tamely embedded basic object we make use of the function $t^{(d)}$. In this case the descending properties in Theorem \ref{315} says that the $ (\mathcal
F_{k_0}(\max t_{k_0}^{(d)}), (N_{k_0}, E'_{k_0}))$ defines a general basic object on
$W_{k_0}$, and also $ (\mathcal F_{k_0}(\max t_{k_0}^{(d)}), (M_{k_0},
F'_{k_0}))$ defines a basic object on $W_{k_0}$, and both are weakly equivalent. So both 
$ (\mathcal
F_{k_0}(\max t(em)_{k_0}^{(d)}), (N_{k_0}, E'_{k_0}))$ and $ (\mathcal F_{k_0}(\max t(em)_{k_0}^{(d)}), (M_{k_0},
F'_{k_0}))$
lead to the same {\em  non-embedded} $d$-dimensional general basic object on 
$W_{k_0}$.
And here it makes sense to say that the two assignments 
$ (\mathcal
F_{k_0}(\max t(em)_{k_0}^{(d)}), (N_{k_0}, E'_{k_0}))$ and $ (\mathcal F_{k_0}(\max t(em)_{k_0}^{(d)}), (M_{k_0},
F'_{k_0}))$
are the same, as their closed sets lie
as subsets of the same space ($W_{k_0}$ or a transform of $W_{k_0}$).
\end{parrafo}

\begin{parrafo}
{\bf Case 2: Two different basic objects and two different embeddings.}\label{tde}

The problem of resolution of singularities will lead us to this further generalization. The following example illustrates this fact.

Fix a singular reduced scheme $X_0$ and:
\begin{enumerate}
\item a closed embedding $X_0\subset W_0$, and

\item a closed embedding $X_0\subset V_0$,
\end{enumerate}
where $V_0$ and $W_0$ are smooth schemes which might have different dimension.

Hironaka says that (after taking suitable \'etale neighborhoods),
there are two basic objects:
\begin{enumerate}
\item[(1')] $(W_0,(J_0,b), \emptyset)$,

\item[(2')] $(V_0, (K_0,c), \emptyset)$,
\end{enumerate}
and satisfying the following conditions: Here (1') defines an assignment of closed sets, say 
$(\mathcal F_{}, (W_0, \emptyset))$, and 2') defines $(\mathcal G_{}, (V_0, \emptyset))$.  The closed set assigned to $W_0$, and to $V_0$,  is the Hilbert-Samuel stratum of $X_0$. The embedded dimension (and also the dimension) of $X_0$
locally at any closed point in the Hilbert stratum is constant. Take $d$
to be, for example, the dimension. Then Theorem \ref{H6} says that both $(\mathcal F_{}, (W_0, \emptyset))$ and $(\mathcal G_{}, (V_0, \emptyset))$ have a structure 
of $d$-dimensional general basic objects. Actually both are 
tame in this case as the second coordinates are $\emptyset$.

The closed set defined by $(\mathcal F_{0}, (W_0, \emptyset))$ and 
by $(\mathcal G_{0}, (V_0, \emptyset))$ are $\Sing(J_0,b)$ and $\Sing(K_0,c)$ which we
naturally identify with the Hilbert-Samuel stratum
of $X_0$. The same holds for transformations.

As both are general basic objects of the same dimension $d$,  \ref{trdh} will show that there is a well defined 
function 
$$\ord^{(d)}: \Sing(J_0,b) \longrightarrow \mathbb Q\ \hbox{ and }\ \ord^{(d)}: \Sing(K_0,d) \longrightarrow \mathbb Q,$$
which coincide as functions on the Hilbert stratum.
The same will hold if we take a common sequence of blow-ups as both are the basic objects corresponding to the highest Hilbert
Samuel function of $X_0$ and its transforms.  So the inductive functions $t^{(d)}$ will coincide via this identification, and moreover,
if 
\begin{equation}\label{lda22}
\begin{array}{cccccccccc}
\mathcal F_{0}& &\mathcal F_{1}&&&&\mathcal F_{k_0}\\
(W_{0}, E'_0=\emptyset) & \longleftarrow & (W_{1}, E'_1) & \longleftarrow &
\cdots  &\longleftarrow &
(W_{k_0}, E'_{k_0})& &
\end{array}
\end{equation}	
and
\begin{equation}\label{lda33}
\begin{array}{cccccccccc}
\mathcal G_{0}& &\mathcal G_{1}&&&&\mathcal G_{k_0}\\
(V_{0}, F'_0=\emptyset) & \longleftarrow & (V_{1}, F'_1) & \longleftarrow &
\cdots  &\longleftarrow &
(V_{k_0}, F'_{k_0})& &
\end{array}
\end{equation}	
are the resolutions defined by blowing up at the maximum of the functions (essentially $t^{(d)}$, $t^{(d-1)}$, \dots ,$t^{(1)}$), they both induce a sequence of blow-ups
over $X_0$, say
\begin{equation}\label{5lda}
\begin{array}{cccccccc}
X_{0} & \overset{\pi_{Y_0}}{\longleftarrow} & X_{1}& \overset{\pi_{Y_1}}{\longleftarrow} &
 &\cdots  &\overset{\pi_{Y_{k_0-1}}}{\longleftarrow} &
X_{k_0},
\end{array}
\end{equation}
together with closed immersions  $X_i \subset W_i$, and $X_i \subset V_i$. In fact, for each index $i$, centers $Y_i$ in (\ref{lda22}) and 
in (\ref{lda33}) coincide as subsets of the Hilbert-Samuel stratum of $X_i$. Note that $Y_i$ arises as the maximum of a function defined in terms of the inductive functions $t^{(d)}$, $t^{(d-1)}$,\dots. Moreover:
\begin{itemize}
\item $\max HS_{X_0}=\max HS_{X_0}=\cdots =\max HS_{X_{k_0-1}}>\max HS_{X_{k_0}}.$
\item There are closed immersions $X_{k_0}\subset W_{k_0}$,  $X_{k_0}\subset V_{k_0}$, and a natural correspondence between $E'_{k_0}$ and $F'_{k_0}$ as discussed in Case 1).
\end{itemize}

Hironaka says that (after taking suitable \'etale neighborhoods),
there are two basic objects:
\begin{enumerate}
\item[(1')] $(W_{k_0},(J_{k_0},b'), E'_{k_0})$,

\item[(2')] $(V_{k_0}, (K_{k_0},c'), F'_{k_0})$,
\end{enumerate}
where (1') defines an assignment of closed sets, say 
$(\mathcal F_{k_0}, (W_{k_0}, E'_{k_0}))$ and (2') defines $(\mathcal G_{k_0}, (V_{k_0}, F'_{k_0}))$ both attached to the Hilbert-Samuel stratum of $X_{k_0}$. 

Take $d$
to be the dimension of $X_{k_0}$ at any closed point of the stratum. Then Theorem \ref{H6} says both $(\mathcal F_{k_0}, (W_{k_0}, E'_{k_0}))$ and $(\mathcal G_{k_0}, (V_{k_0}, F'_{k_0}))$ are $d$-dimensional general basic objects. In addition  there is a natural bijection of $E'_{k_0}$ with $F'_{k_0}$. Set $(E'_{k_0})^-=E'_{k_0}$ and $(F'_{k_0})^-=F'_{k_0}$.

As both are general basic objects of the same dimension $d$, 
the previous discussion, and that in Case 1), show that the inductive functions $t(em)^{(d)}$ will coincide via this identification, and moreover,
if 
\begin{equation}\label{lda221}
\begin{array}{cccccccccc}
\mathcal F_{k_0}& &\mathcal F_{k_0+1}&&&&\mathcal F_{k_1}\\
(W_{k_0}, E'_{k_0}) & \longleftarrow & (W_{k_0+1}, E'_{k_0+1}) & \longleftarrow &
\cdots  &\longleftarrow &
(W_{k_1}, E'_{k_1})& &
\end{array}
\end{equation}	
and
\begin{equation}\label{lda331}
\begin{array}{cccccccccc}
\mathcal G_{k_0}& &\mathcal G_{k_0+1}&&&&\mathcal G_{k_1}\\
(V_{k_0}, F'_{k_0}) & \longleftarrow & (V_{k_0+1}, F'_{k_0+1}) & \longleftarrow &
\cdots  &\longleftarrow &
(V_{k_1}, F'_{k_1})& &
\end{array}
\end{equation}	
are the resolutions defined by blowing up at centers included in the maximum of the functions, they both induce a sequence of blow-ups
over $X_{k_0}$, say
\begin{equation}\label{5lday}
\xymatrix@R=0pc@C=3pc{
X_{k_0} & X_{k_0+1}\ar[l]_{\pi_{Y_{k_0}}}& \cdots\ar[l]_{\ \ \ \ \pi_{Y_{k_0+1}}}  &
X_{k_1}\ar[l]_{\pi_{Y_{k_1-1}}}
}
\end{equation}
together with closed immersions  $X_i \subset W_i$, and $X_i \subset V_i$. In fact, for each index $i$, centers $Y_i$ in (\ref{lda221}) and 
in (\ref{lda331}) coincide as subsets of the Hilbert-Samuel stratum of $X_i$. Note that $Y_i$ arises as the maximum of a function defined in terms of the function $t(em)^{(d)}$ (see Case 1)), and the inductive functions $t^{(d)}$, $t^{(d-1)}$,\dots. Moreover:
\begin{itemize}
\item $\max HS_{X_{k_0}}=\max HS_{X_{k_0+1}}=\cdots =\max HS_{X_{k_1-1}}>\max HS_{X_{k_1}}.$
\item There are closed immersions $X_{k_1}\subset W_{k_1}$,  $X_{k_1}\subset V_{k_1}$, and a natural correspondence between $E'_{k_1}$ and $F'_{k_1}$ as discussed in Case 1).
\end{itemize}
so again we set $(E'_{k_1})^-=E'_{k_1}$ and $(F'_{k_1})^-=F'_{k_1}$ and repeat the previous argument.

Finally, Hironaka proves that a sequence of transformations as above, so that 
$$\max {HS_{X_{0} }}>  \max {HS_{X_{k_0} }} >\max {HS_{X_{k_1} }}>\dots$$ 
leads to a resolution of singularities of $X_0$ if it is reduced, since 
$\max {HS_{X_i }}$ cannot decrease infinitely many times.

\end{parrafo}
\begin{parrafo}{\bf Constructive resolution}.

We first address the constructive resolution of a basic objects $\B_0=(W_0,(J_0,b),E_0)$ as stated in 
\ref{crobo}. It also applies to tame embedded general basic objects.

\begin{definition}\label{thefunction}
Set $T^d=\{\infty\}\sqcup ({\mathbb Q}\times {\mathbb Z}) \sqcup
\Gamma$ where this disjoint union is totally ordered by setting that
$\infty$ is the biggest element, $(\mathbb{Q}\times \mathbb{Z})$ is ordered by the lexicographic order, and that $\alpha < \beta $ if
$\beta \in ({\mathbb Q}\times {\mathbb Z})$ and $\alpha \in \Gamma
$ (see (\ref{lfhd})).  Set now $ I_{d}=T^d \times I_{d-1} $ ordered
lexicographically, and define $ g_r^d: \Sing(J_r,b) \longrightarrow I_{d}$.

\begin{enumerate}
\item[(i)]If  $\max \word_r=0$: $
g_r^d(x)= (h_{}(x),\infty_{d-1})$ (see (\ref{lfhd})). 

\item[(ii)] If $\max \word_r> 0$, then it will be enough to define
the function for $x \in \Max t^{(d)}_r (\subset \Sing(J_r,b))$ In fact, we
will define the resolution function, and hence the resolution
sequence, so that centers be included in $\Max t_r^{(d)}$.  Let $(J'',b'')$
be the $d$-dimensional basic object attached to $\max t_r^{(d)} $ (defining the
closed set $\Max t_r^{(d)}$, (see \ref{315})), and finally set: 

\vskip 0.1cm
\begin{enumerate}
\item[A)] $ g_r^d(x)=(\max t_r^{(d)}, \infty_{d-1})$ if $x\in R(1)(\Max t_r^{(d)})$ (see Remark \ref{rmk42}).
\vskip 0.1cm

\item[B)] $ g_r^d(x)=(\max t_r^{(d)}, g_r^{d-1}(x))$, if $x\notin
R(1)(\Max t_r^{(d)})$, where $g_r^{d-1}(x)$ is defined in
accordance to the $(d-1)$-dimensional general basic object attached to
$(J'',b'')$ (see Theorem \ref{315}).
\end{enumerate}
\end{enumerate}
\end{definition}

The precise definition in this case B) requires some clarification:
Assume, by induction, that $I_{d-1} $ has been introduced, and also
functions $g_r^{(d-1)}$ which define resolution of $(d-1)$-dimensional
general basic objects.

$T^d$ is the totally ordered set introduced for constructive
resolution for tame general basic objects, in terms of the
functions $g^d_r$.  The string of invariants attached to $\max g_r^d $
looks like:

\begin{itemize}
\item[1)] $\max  g_r^d =
(\alpha_1, \alpha_2, \dots, \alpha_e, \infty, \infty, \dots , \infty )$,
where $\alpha_i \in  ({\mathbb Q}\times {\mathbb Z}) $.

\item[2)] $\max g_r^d =(\alpha_1, \alpha_2, \dots, \alpha_e, \gamma, \infty,
\infty, \dots , \infty )$, $\alpha_i \in ({\mathbb Q}\times {\mathbb
Z}) $, where $\gamma \in \Gamma$.  One can also read the dimension of the
canonically defined center $\Max g_r^d$ from this datum (see
\cite{BHV}).
\end{itemize}

\end{parrafo}
\begin{remark}
Resolution of a $d$-dimensional reduced scheme over fields of
characteristic zero is achieved by setting the totally ordered set
$\mathbb N^{\mathbb N}\times(\mathbb{Q}\times\mathbb{Z})\times I_{d}$, ordered lexicographically, and
for any such scheme $X$ set the functions
$$f^{(d)}(x): (HS_X(x), t(em)^{(d)}(x),g^{(d)}(x)).$$
Follow the indication in Definition  \ref{thefunction} for the definition of the third
coordinate in terms of the first two coordinates.

The compatibility of the constructive resolution with smooth morphisms in
\ref{pt5}, follows from:
\begin{enumerate}
\item  the discussion in \ref{ladeHS}, and 

\item the discussion in Remark \ref{satf} which shows that this property relies
entirely on the compatibility of Hironaka's function $\ord$ with
smooth morphisms.
\end{enumerate}
\end{remark}

\section{On Hironaka's main invariant}

\begin{parrafo}	\textbf{On Hironaka's tricks} 
	
\begin{proof1}\label{trdh}
Fix a basic object  $	\B_{0}=(W_{0},(J_{0},b),E_{0})$, and set $d=\dim W_0$.
Recall that this basic object defines an $d$-dimensional  assignment of closed sets $ 
(\mathcal{F}_{0},(W_{0},E_{0})) $
by assigning for all local sequence (see Definition \ref{lsbo}):
\begin{equation}\label{rlcmw}
\begin{array}{ccccccc}
 (J_0,b) & & (J_1,b) &  & &  &(J_r,b)\\
(W_{0}, E_0) & \longleftarrow & (W_{1}, E_1) & \longleftarrow &
\cdots  &\longleftarrow &
(W_{r}, E_r)
\end{array}
\end{equation}
closed sets
\begin{equation}\label{rlcm2w}
\begin{array}{ccccccc}
 \Sing(J_0,b) & &  \Sing(J_1,b) &  & &  & \Sing(J_r,b)\\
(W_{0}, E_0) & \longleftarrow & (W_{1}, E_1) & \longleftarrow &
\cdots  &\longleftarrow &
(W_{r}, E_r)
\end{array}
\end{equation}
This is an assignment of closed sets where $ F_{i}=\Sing(J_{i},b) $.

Hironaka indicates that if you consider all possible sequences
(\ref{rlcm2w}), then you can find out the rational number
$\frac{\nu_{x_{0}}(J_{0})}{b}$ at any $x_0\in \Sing(J_0,b)$. More precisely, if you consider all sequences (\ref{rlcm2w}) together with points $x_i \in F_i$, each $x_i$ mapping to $x_{i-1}$ (and hence all mapping to $x_0\in F_0$), then the rational number $\frac{\nu_{x_{0}}(J_{0})}{b}$ is completely determined by the {\em codimension} of $F_i$ in $W_i$ locally at $x_i$, for all sequences as before.
In particular, if $\B_{0}=(W_{0},(J_{0},b), E_{0})$ and $\B'_{0}=(W_{0},(K_{0},d),E_{0})$
are weakly equivalent, then at any point $x_{0}\in\Sing(J_{0},b)=\Sing(K_{0},d) $,
$$ \frac{\nu_{x_{0}}(J_{0})}{b}=\frac{\nu_{x_{0}}(K_{0})}{d}.$$
Set $ \nu_{x_{0}}(J_{0})=b' $, so
$ \frac{\nu_{x_{0}}(J_{0})}{b}=\frac{b'}{b}.$
	
\noindent {\bf Assume first that $x_0$ is closed}. Define
$$ W_{0} \stackrel{\pi_{0}}{\longleftarrow} W_{1}=W_{0}\times\mathbb{A}^{1}_{k}$$
as the projection, so the fiber over $ x_{0} $ is a line, say $ L_{1}=\{x_{0}\}\times\mathbb{A}^{1}_{k} $. Set $ x_{1}=(x_{0},0)\in L_{1} $. Here $ \pi_{1} $ is smooth and defines, by taking pull-backs:
 \begin{equation}\label{grls}
\begin{array}{ccc}
 (J_0, b) & & (J_{1},b) \\
(W_0 , E_0) & \overset{\pi_0}{\longleftarrow} & (W_{1} ,E_{1})\\
\end{array}
\end{equation}
Now $ x_{1} $ is a closed point on the line $ L_{1} $, and we define, for any integer $ N $, a sequence 
\begin{equation}\label{HT1}
(W_{1},E_{1})\stackrel{\pi_{1}}{\longleftarrow}
(W_{2},E_{2})\stackrel{\pi_{2}}{\longleftarrow}\cdots
\stackrel{\pi_{N-1}}{\longleftarrow}(W_{N},E_{N})
\end{equation}
defined as follows: Let $ \pi_{1} $ be the blow-up at $ x_{1} $.
	
For any index $ i>1 $, set $ H_{i} $ the exceptional locus of $ \pi_{i-1} $, $ L_{i} $ the strict transform of $ L_{i-1} $ and $ x_{i}=L_{i}\cap H_{i} $, finally define $ \pi_{i} $ as the blow-up at the closed point $ x_{i} $.
	
In our example $ x_{1}\in L_{1}\subset\Sing(J_{1},b) $. One can check by induction that for any index $ i $, $ x_{i}\in L_{i}\subset\Sing(J_{i},b) $. So the sequence (\ref{HT1}) induces 
\begin{equation}\label{rlc5w}
\begin{array}{cccccccc}
 (J_0,b) & & (J_1,b) &  & &  &(J_N,b)\\
(W_{0}, E_0) & \overset{\pi_0}{\longleftarrow} & (W_{1}, E_1) & \overset{\pi_1}{\longleftarrow} &
\cdots  &\overset{\pi_{N-1}}{\longleftarrow} &
(W_{N}, E_N)
\end{array}
\end{equation}
and closed sets
\begin{equation}\label{rlcm32w}
\begin{array}{cccccccc}
 \Sing(J_0,b) & &  \Sing(J_1,b) & & &  & \Sing(J_N,b)\\
(W_{0}, E_0) & \overset{\pi_0}{\longleftarrow} & (W_{1}, E_1) & \overset{\pi_1}{\longleftarrow} &
\cdots  &\overset{\pi_{N-1}}{\longleftarrow} &
(W_{N}, E_N)
\end{array}
\end{equation}
Now check that locally at $ x_{N} $:
\begin{equation}\label{HT3}
J_{N}=I(H_{N})^{(N-1)(b'-b)}\bar{J}_{N}.
\end{equation}
Recall that $\dim W_0=d$, so 
$$\dim W_1=\dim W_2=\cdots =\dim W_N=d+1.$$ 
Set $F_N=\Sing(J_N, b)$. Note that $\dim F_N \leq d$. Formula (\ref{HT3}) says that
\begin{equation}\label{HT5}
(N-1)(b'-b)\geq b
\quad\Longleftrightarrow\quad
\dim\left(F_{N}\cap H_{N}\right)=d
\quad\Longleftrightarrow\quad
H_{N}\subset F_{N}.
\end{equation}
Note that $ b'-b=0 $ (i.e. $ \dfrac{b'}{b}=1 $) if and only if for \emph{all} $ N $, $ \dim(F_{N}\cap H_{N})<d$ (a formula that involves only the assignment of closed sets $ (\mathcal{F}_{0},(W_{0},E_{0})) $ defined by $\B_{0} $).
	
On the other hand, if $ b'-b>0 $, for any integer $ N $ big enough $ H_{N}\subset F_{N} $, we may define a blow-up $\pi_N$ with center $ Y_N=H_{N} $:
 \begin{equation}\label{gwrls}
\begin{array}{ccc}
 (J_N, b) & & (J_{N+1},b) \\
(W_N , E_N) & \overset{\pi_{Y_N}}{\longleftarrow} & (W_{N+1} ,E_{N+1})\\
\end{array}
\end{equation}
Note that $ H_{N}\subset W_{N} $ is a hypersurface and the blow-up at a hypersurface is an isomorphism, so $ W_{N} $ may be identified with $ W_{N+1} $, $ H_{N} $ with the exceptional locus of $ \pi_{N} $, say $ H_{N+1} $, and $ x_{N} $ with a unique point, say $ x_{N+1}\in W_{N+1} $.Note that locally at $ x_{N+1}$
$$ J_{N+1}=I(H_{N+1})^{(N-1)(b'-b)-b}\bar{J}_{N+1}. $$
Set $ F_{N+1}=\Sing(J_{N+1},b) $ and note that
$$ (N-1)(b'-b)-b\geq b\ \Longleftrightarrow\ \dim\left(F_{N+1}\cap H_{N+1}\right)=d\ \Longleftrightarrow\ H_{N+1}\subset F_{N+1} .$$
If these equivalent conditions hold we may blow-up again along the hypersurface $ H_{N+1} $. So, whenever possible, set
\begin{equation}\label{r3c5w}
\xymatrix@R=0pc@C=3pc{
 (J_N,b) & (J_{N+1},b) &  & (J_{N+S},b)\\
(W_{N}, E_N)  & (W_{N+1}, E_{N+1})\ar[l]_{\pi_{N}} & \cdots\ar[l]_{\ \ \ \ \ \ \ \ \ \ \pi_{N+1}} &
(W_{N+S}, E_{N+S})\ar[l]_{\!\!\!\!\!\!\!\!\!\!\!\!\!\!\!\pi_{N+S-1}}
}
\end{equation}
by blowing-up the same hypersurface. One can check that locally at $x_{N+S} $ (mapping to $ x_{N} $ via the identity map):	
$$ J_{N+S}=I(H_{N+S})^{(N-1)(b'-b)-bS}\bar{J}_{N+S}.$$
Now sequences (\ref{rlc5w}) and (\ref{r3c5w}) induce a sequence over the 
assignment of closed sets $ (\mathcal{F}_{0},(W_{0},E_{0})) $:
\begin{multline} \label{HT7}
		(\mathcal{F}_{0},(W_{0},E_{0}))\longleftarrow(\mathcal{F}_{1},(W_{1},E_{1}))
		\longleftarrow\cdots\longleftarrow
		(\mathcal{F}_{N},(W_{N},E_{N}))\longleftarrow \\
		\longleftarrow(\mathcal{F}_{N+1},(W_{N+1},E_{N+1}))
		\longleftarrow\cdots\longleftarrow
		(\mathcal{F}_{N+S},(W_{N+S},E_{N+S}))
\end{multline}
where $ F_{i}=\Sing(J_{i},b) $ for $ i=0,1,\ldots,N+S $.
	 
Finally the sequence (\ref{HT7}) can be defined if and only if:
 \begin{align*}
(N-1)(b'-b)-(S-1)b\geq b \quad \Longleftrightarrow \quad & 
(N-1)(b'-b)\geq Sb \\	
 \Longleftrightarrow\quad & S\leq\left\lfloor\frac{(N-1)(b'-b)}{b}\right\rfloor=\left\lfloor(N-1)\Big(\frac{b'}{b}-1\Big)\right\rfloor
 \end{align*}
 where $ \lfloor\ \rfloor $ denotes the integer part.
	 
 This means that, for a fixed integer $ N $,  $ \Big\lfloor(N-1)\dfrac{b'}{b}-1\Big\rfloor $ is the 
 biggest integer $ S $ so that (\ref{HT7}) is defined.	 Thus for any integer $ N $ the integer
 $ \Big\lfloor(N-1)\dfrac{b'}{b}-1\Big\rfloor $ is determined by the assignment of closed sets defined by $\B_{0} $. Finally note that  
 $$ \frac{\nu_{x_{0}}(J_{0})}{b}-1= \frac{b'}{b}-1=
 \lim_{N\rightarrow\infty}\frac{1}{N-1} \Big\lfloor(N-1)\dfrac{b'}{b}-1\Big\rfloor $$
 In particular the rational number $ \dfrac{\nu_{x_{0}}(J_{0})}{b}  $ is defined in terms of the assignment of closed sets $ (\mathcal{F}_{0},(W_{0},E_{0})) $ defined by $ \B_{0} $,  at least when $x_0$ is a closed point in $\Sing(J_0,b)$.
 
 \vspace{0.2cm}
  
\noindent{\bf Suppose now that $x_0$ is not closed}. 
	
In this case restrict $W_0$ to an open subset so that the closure of $x_0$ is a smooth scheme, say $Y_0$, which has normal crossings with $E_0$. Define, as before 
$$ W_{0} \stackrel{\pi_{1}}{\longleftarrow}W_{1}=W_{0}\times\mathbb{A}^{1}_{k} $$
as the projection, so the fiber over $ Y_{0} $ is a smooth subscheme, say $ L_{1}=\{Y_{0}\}\times\mathbb{A}^{1}_{k} $. Set $ Y_{1}=(Y_{0},0)\subset L_{1} $.
	
As $ \pi_{0} $ is smooth, define as before:
 \begin{equation}\label{grls_1}
\begin{array}{ccc}
 (J_0, b) & & (J_{1},b) \\
(W_0 , E_0) & \overset{\pi_{Y_1}}{\longleftarrow} & (W_{1} ,E_{1})
\end{array}
\end{equation}
 $ Y_{1} $ is a smooth subscheme in $ L_{1} $ and a permissible center $Y_1\subset \Sing(J_1,b)$. We define, for any integer $ N $, a sequence 
\begin{equation}\label{HT1_1}
\xymatrix@C=2.8pc{
(W_{1},E_{1}) &(W_{2},E_{2})\ar[l]_{\pi_{1}} & \cdots\ar[l]_{\ \ \ \ \ \pi_{2}} & (W_{N},E_{N})\ar[l]_{\!\!\!\!\!\!\!\pi_{N-1}}
}\end{equation}
defined as follows: Let $ \pi_{1} $ be the blow-up at $ Y_{1} $.
For all index $ i>1 $, set $ H_{i} $ the exceptional locus of $ \pi_{i-1} $, $ L_{i} $ the strict transform of $L_{i-1} $ and $ Y_{i}=L_{i}\cap H_{i} $, finally define $\pi_{i} $ as the blow-up at $ Y_{i} $ which turns to be a permissible center $Y_i\subset \Sing(J_i,b)$.  The extension of the previous discussion to this context is now straightforward. 
\end{proof1}
\end{parrafo}

\end{document}